\pdfoutput=1
\RequirePackage{ifpdf}
\ifpdf 
\documentclass[pdftex]{sigma}
\else
\documentclass{sigma}
\fi

\usepackage{mathrsfs}
\usepackage{mathtools}

\usepackage{tikz}
\usepackage{tikz-cd}
\usetikzlibrary{arrows}
\usetikzlibrary{calc,patterns}
\usetikzlibrary{decorations.markings}
\usetikzlibrary{patterns,decorations.pathreplacing}

\numberwithin{equation}{section}

\newtheorem{Theorem}{Theorem}[section]
\newtheorem*{Theorem*}{Theorem}
\newtheorem{Corollary}[Theorem]{Corollary}
\newtheorem{Lemma}[Theorem]{Lemma}
\newtheorem{Proposition}[Theorem]{Proposition}

\theoremstyle{definition}
\newtheorem{Definition}[Theorem]{Definition}

\newtheorem{Example}[Theorem]{Example}
\newtheorem{Observation}[Theorem]{Observation}
\newtheorem{Setting}[Theorem]{Setting}
\newtheorem{Remark}[Theorem]{Remark}

\newcommand{\ZZ}{\mathbb{Z}}

\newcommand{\RR}{\mathbb{R}}
\newcommand{\CC}{\mathbb{C}}
\newcommand{\BS}{\mathbb{S}}
\newcommand{\TT}{\mathbb{T}}

\newcommand{\calD}{\mathcal{D}}
\newcommand{\calJ}{\mathcal{J}}
\newcommand{\calL}{\mathcal{L}}
\newcommand{\calM}{\mathcal{M}}
\newcommand{\calO}{\mathcal{O}}
\newcommand{\calR}{\mathcal{R}}
\newcommand{\calT}{\mathcal{T}}
\newcommand{\calZ}{\mathcal{Z}}

\newcommand{\rmc}{\mathrm{c}}
\newcommand{\rmD}{\mathrm{D}}
\newcommand{\rmH}{\mathrm{H}}

\newcommand{\fkc}{\mathfrak{c}}
\newcommand{\fkS}{\mathfrak{S}}

\newcommand{\sfM}{\mathsf{M}}
\newcommand{\sfN}{\mathsf{N}}
\newcommand{\sfP}{\mathsf{P}}

\newcommand{\Spec}{\operatorname{Spec}}
\newcommand{\GL}{\operatorname{GL}}
\newcommand{\SL}{\operatorname{SL}}
\newcommand{\Hom}{\operatorname{Hom}}
\newcommand{\End}{\operatorname{End}}
\newcommand{\Supp}{\operatorname{Supp}}
\newcommand{\Pic}{\operatorname{Pic}}
\newcommand{\id}{\operatorname{id}}
\newcommand{\sgn}{\operatorname{sgn}}

\newcommand{\hd}{\mathsf{hd}}
\newcommand{\tl}{\mathsf{tl}}
\newcommand{\barone}{\underline{1}}
\newcommand{\Zig}{\mathsf{Zig}}
\newcommand{\Zag}{\mathsf{Zag}}
\newcommand{\aZig}{\mathsf{zig}}
\newcommand{\aZag}{\mathsf{zag}}
\newcommand{\PM}{\mathsf{PM}}
\newcommand{\cPM}{\sfP_\diamond}
\newcommand{\Hex}{\mathsf{Hex}}
\newcommand{\Graph}{\mathsf{Graph}}

\newcommand{\typeII}{{\rm I\hspace{-0.7pt}I}}
\newcommand{\typeIII}{{\rm I\hspace{-0.7pt}I\hspace{-0.7pt}I}}
\newcommand{\perm}{\pi}
\newcommand{\permII}{\omega}
\newcommand{\dd}{\underline{d}}

\newcommand{\CrepantResola}
{
\foreach \n/\a/\b in {00/0/0,10/1/0,01/0/1,11/1/1,20/2/0} {
\coordinate (V\n) at (\a,\b); };
\draw [line width=\edgewidth, fill=white] (V00)--(V20)--(V11)--(V01)--(V00) ;
\foreach \s/\t in {10/01,10/11} {\draw [line width=\edgewidth] (V\s)--(V\t) ;};
\foreach \n in {00,10,01,11,20}{
\draw [fill=black] (V\n) circle [radius=\noderad] ; };
}

\newcommand{\CrepantResolb}
{
\foreach \n/\a/\b in {00/0/0,10/1/0,01/0/1,11/1/1,20/2/0} {
\coordinate (V\n) at (\a,\b); };
\draw [line width=\edgewidth, fill=white] (V00)--(V20)--(V11)--(V01)--(V00) ;
\foreach \s/\t in {00/11,10/11} {\draw [line width=\edgewidth] (V\s)--(V\t) ;};
\foreach \n in {00,10,01,11,20}{
\draw [fill=black] (V\n) circle [radius=\noderad] ; };
}

\newcommand{\CrepantResolc}
{
\foreach \n/\a/\b in {00/0/0,10/1/0,01/0/1,11/1/1,20/2/0} {
\coordinate (V\n) at (\a,\b); };
\draw [line width=\edgewidth, fill=white] (V00)--(V20)--(V11)--(V01)--(V00) ;
\foreach \s/\t in {10/01,20/01} {\draw [line width=\edgewidth] (V\s)--(V\t) ;};
\foreach \n in {00,10,01,11,20}{
\draw [fill=black] (V\n) circle [radius=\noderad] ; };
}

\begin{document}

\allowdisplaybreaks

\newcommand{\arXivNumber}{2309.16112}

\renewcommand{\PaperNumber}{083}

\FirstPageHeading

\ShortArticleName{Variations of GIT Quotients and Dimer Combinatorics}

\ArticleName{Variations of GIT Quotients and Dimer Combinatorics\\ for Toric Compound Du Val Singularities}

\Author{Yusuke NAKAJIMA}

\AuthorNameForHeading{Y.~Nakajima}

\Address{Department of Mathematics, Kyoto Sangyo University,\\
Motoyama, Kamigamo, Kita-Ku, Kyoto, 603-8555, Japan}
\Email{\mail{ynakaji@cc.kyoto-su.ac.jp}}
\URLaddress{\url{https://sites.google.com/view/nakajima-math}}

\ArticleDates{Received June 19, 2025, in final form August 18, 2026; Published online August 26, 2026}

\Abstract{A dimer model is a bipartite graph described on the real two-torus, from which we can define the associated quiver.
It is known that for any three-dimensional Gorenstein toric singularity, there exists a dimer model such that a GIT quotient parametrizing stable representations of the associated quiver is a projective crepant resolution of this singularity for some stability parameter. It is also known that the space of stability parameters has a wall-and-chamber structure, and any projective crepant resolution of a three-dimensional Gorenstein toric singularity can be realized as the GIT quotient associated to a stability parameter contained in some chamber. In this paper, we consider dimer models giving rise to projective crepant resolutions of a toric compound Du Val singularity. We show that sequences of zigzag paths, which are special paths on a dimer model, determine the wall-and-chamber structure of the space of stability parameters. Moreover, we can track variations of stable representations under wall-crossing using the sequences of zigzag paths.}


\Classification{14E15; 16G20; 14D22; 14M25; 14L24}

\section{Introduction}\label{sec_intro}

\subsection{Moduli descriptions of crepant resolutions of singularities}\label{subsec_intro_moduli}

For some singularities, resolutions of singularities can be described as moduli spaces of certain objects.
For example, a minimal resolution of a two-dimensional Gorenstein quotient singularity~$\CC^2/G$ defined by the action of a finite subgroup
$G\subset\SL(2,\CC)$ on $\CC^2$ can be given as the $G$-Hilbert scheme {\rm $G$-Hilb\,$\CC^2$},
which is the Hilbert scheme parametrizing certain $G$-invariant subschemes, see~\cite{ito_nakamura_GHilb}.
This result was generalized to three-dimensional Gorenstein quotient singularities by~\cite{Nak_GHilb} for the abelian case and by~\cite{BKR} for arbitrary cases.
Precisely, for a~quotient singularity~$\CC^3/G$ defined by the action of a finite subgroup $G\subset\SL(3,\CC)$ on $\CC^3$,
a crepant resolution of $\CC^3/G$ can be given as {\rm $G$-Hilb\,$\CC^3$}.
Also, {\rm $G$-Hilb\,$\CC^3$} is described as the moduli space of representations of the McKay quiver of $G$ satisfying some stability condition.
The moduli space of representations of a quiver, introduced in~\cite{King_moduli}, is defined as the GIT quotient associated to a stability parameter
(see Section~\ref{subsec_crepant_resolution} for the detail).
The space $\Theta(Q)_\RR$ of stability parameters associated to a quiver~$Q$ has a wall-and-chamber structure,
that is, it is decomposed into chambers (open cones in~$\Theta(Q)_\RR$) separated by walls (codimension one faces of the closures of chambers).
The moduli spaces associated to stability parameters contained in the same chamber are isomorphic, but if we take a stability parameter from
other chambers, then it would give a~different moduli space.

For a particular choice of stability parameters, the moduli space of representations of the McKay quiver of $G$ is
isomorphic to {\rm $G$-Hilb\,$\CC^3$}.
On the other hand, a crepant resolution of $\CC^3/G$ is not unique in general, thus it is natural to expect that
any crepant resolution has a moduli description.
In fact, it was shown in \cite{CrawIshii} that for any finite abelian subgroup $G\subset\SL(3,\CC)$,
any projective crepant resolution of $\CC^3/G$ is isomorphic to the moduli space of representations of the McKay quiver of $G$ for some stability parameter.
Recently, it was shown in \cite{Yamagishi_CIconj} that the same statement holds for any finite subgroup $G\subset\SL(3,\CC)$.
These results can be obtained by observing the variations of moduli spaces under crossing walls.
In particular, it is important to classify walls in $\Theta(Q)_\RR$ according to an effect on moduli spaces.
Along this line, it is also important to detect the precise description of chambers and walls (i.e., detect the complete structure of $\Theta(Q)_\RR$),
although it would not be necessary to obtain the result of \cite{CrawIshii,Yamagishi_CIconj}.
For example, there are some results, e.g., \cite{craw_thesis,MT_stability,NS_flops,worm_GHilb}, which describe walls and chambers for some McKay quivers.

Also, there are several generalizations of the result in \cite{CrawIshii} for other singularities, e.g., \cite{IU_anycrepant,Jung_CIconj,Wem_MMP}.
In particular, it was shown in \cite{IU_anycrepant} that any projective crepant resolution of a three-dimensional Gorenstein toric singularity can be
described as the moduli space of representations of a quiver associated to a dimer model for some stability parameter (see Theorem~\ref{thm_crepant_dimer}).
A~\emph{dimer model} $\Gamma$ is a bipartite graph described on the real two-torus $\TT\coloneqq\RR^2/\ZZ^2$, which will be introduced in Section~\ref{subsec_dimermodel} in detail.
As the dual of a dimer model $\Gamma$, we can obtain the quiver $Q_\Gamma$ with relations, see Section~\ref{subsec_quiver} for more details.
For example, the left side of Figure~\ref{ex_dimer1_basic} shows a dimer model on $\TT$, where the outer frame represents a fundamental domain of $\TT$, and the right side shows the associated quiver.

\newcommand{\basicdimerA}{
\foreach \n/\a/\b in {1/1.5/3, 2/3.5/3, 3/5.5/3, 4/7.5/3} {\coordinate (B\n) at (\a,\b);}
\foreach \n/\a/\b in {1/0.5/1,2/2.5/1, 3/4.5/1, 4/6.5/1} {\coordinate (W\n) at (\a,\b);}
\draw[line width=\edgewidth] (0,0) rectangle (8,4);
\foreach \w/\b in {1/1,2/2,3/3,4/3} {\draw[line width=\edgewidth] (W\w)--(B\b); };
\foreach \w/\s/\t in {1/0/2,1/0/0,1/1/0,2/2/0,2/3/0,3/4/0,3/5/0,4/6/0,4/7/0} {\draw[line width=\edgewidth] (W\w)--(\s,\t); };
\foreach \b/\s/\t in {1/1/4,1/2/4,2/3/4,2/4/4,3/5/4,3/6/4,4/7/4,4/8/4,4/8/2} {\draw[line width=\edgewidth] (B\b)--(\s,\t); };
\foreach \x in {1,2,3,4} {\filldraw [fill=black, line width=\nodewidthb] (B\x) circle [radius=\noderad] ;};
\foreach \x in {1,2,3,4} {\filldraw [fill=white, line width=\nodewidthw] (W\x) circle [radius=\noderad] ;};
}

\begin{figure}[!ht]
\centering
\scalebox{0.6}{
\begin{tikzpicture}[myarrow/.style={black, -latex}]
\newcommand{\noderad}{0.18cm} 
\newcommand{\edgewidth}{0.05cm} 
\newcommand{\nodewidthw}{0.05cm} 
\newcommand{\nodewidthb}{0.04cm} 
\newcommand{\arrowwidth}{0.06cm} 

\node at (0,0){
\begin{tikzpicture}
\basicdimerA \end{tikzpicture}};

\node at (12,0){
\begin{tikzpicture}
\foreach \n/\a/\b in {1/1.5/3, 2/3.5/3, 3/5.5/3, 4/7.5/3} {\coordinate (B\n) at (\a,\b);}
\foreach \n/\a/\b in {1/0.5/1,2/2.5/1, 3/4.5/1, 4/6.5/1} {\coordinate (W\n) at (\a,\b);}
\draw[line width=\edgewidth] (0,0) rectangle (8,4);
\foreach \w/\b in {1/1,2/2,3/3,4/3} {\draw[line width=\edgewidth] (W\w)--(B\b); };
\foreach \w/\s/\t in {1/0/2,1/0/0,1/1/0,2/2/0,2/3/0,3/4/0,3/5/0,4/6/0,4/7/0} {\draw[line width=\edgewidth] (W\w)--(\s,\t); };
\foreach \b/\s/\t in {1/1/4,1/2/4,2/3/4,2/4/4,3/5/4,3/6/4,4/7/4,4/8/4,4/8/2} {\draw[line width=\edgewidth] (B\b)--(\s,\t); };
\foreach \x in {1,2,3,4} {\filldraw [fill=black, line width=\nodewidthb] (B\x) circle [radius=\noderad] ;};
\foreach \x in {1,2,3,4} {\filldraw [fill=white, line width=\nodewidthw] (W\x) circle [radius=\noderad] ;};

\node[blue] (V0) at (0.5,3) {\LARGE$0$}; \node[blue] (V1) at (2,2) {\LARGE$1$};
\node[blue] (V2) at (4,2) {\LARGE$2$}; \node[blue] (V3) at (5.5,1) {\LARGE$3$}; \node[blue] (V4) at (7,2) {\LARGE$4$};

\foreach \s/\t in {V0/V1, V1/V2, V2/V3, V3/V4} {
\draw[shorten >=0.0cm, shorten <=0.0cm, myarrow, blue, line width=\arrowwidth] (\s)--(\t); };
\draw[shorten >=0.0cm, shorten <=0.0cm, myarrow, blue, line width=\arrowwidth] (V0)--(0,4);
\draw[shorten >=0.0cm, shorten <=0.0cm, myarrow, blue, line width=\arrowwidth] (1,4)--(V0);
\draw[shorten >=-0.1cm, shorten <=0.0cm, myarrow, blue, line width=\arrowwidth] (0,2.666)--(V0);
\draw[shorten >=0.0cm, shorten <=0.0cm, myarrow, blue, line width=\arrowwidth] (V1)--(2,4);
\draw[shorten >=0.0cm, shorten <=0.0cm, myarrow, blue, line width=\arrowwidth] (2,0)--(V1);
\draw[shorten >=0.0cm, shorten <=0.0cm, myarrow, blue, line width=\arrowwidth] (V1)--(1,0);
\draw[shorten >=0.0cm, shorten <=0.0cm, myarrow, blue, line width=\arrowwidth] (3,4)--(V1);
\draw[shorten >=0.0cm, shorten <=0.0cm, myarrow, blue, line width=\arrowwidth] (V2)--(4,4);
\draw[shorten >=0.0cm, shorten <=0.0cm, myarrow, blue, line width=\arrowwidth] (4,0)--(V2);
\draw[shorten >=0.0cm, shorten <=0.0cm, myarrow, blue, line width=\arrowwidth] (V2)--(3,0);
\draw[shorten >=0.0cm, shorten <=0.0cm, myarrow, blue, line width=\arrowwidth] (V3)--(5,0);
\draw[shorten >=0.0cm, shorten <=0.0cm, myarrow, blue, line width=\arrowwidth] (5,4)--(V2);
\draw[shorten >=0.0cm, shorten <=0.0cm, myarrow, blue, line width=\arrowwidth] (V4)--(6,4);
\draw[shorten >=0.0cm, shorten <=0.0cm, myarrow, blue, line width=\arrowwidth] (6,0)--(V3);
\draw[shorten >=0.0cm, shorten <=0.0cm, myarrow, blue, line width=\arrowwidth] (V4)--(7,0);
\draw[shorten >=0.0cm, shorten <=0.0cm, myarrow, blue, line width=\arrowwidth] (7,4)--(V4);
\draw[shorten >=0.0cm, shorten <=0.0cm, myarrow, blue, line width=\arrowwidth] (8,0)--(V4);
\draw[shorten >=0.0cm, shorten <=0.0cm, myarrow, blue, line width=\arrowwidth] (V4)--(8,2.666);
\end{tikzpicture}};
\end{tikzpicture}}

\caption{An example of a dimer model (left) and the associated quiver (right).}
\label{ex_dimer1_basic}
\end{figure}

In this paper, using dimer models, we discuss the wall-and-chamber structure and the variations of moduli spaces (projective crepant resolutions) under crossing walls for a particular class of three-dimensional Gorenstein toric singularities called toric compound Du Val (cDV) singularities.

\subsection{Toric compound Du Val singularities}
\label{subsec_cDVtoric}

Compound Du Val (cDV) singularities, which are fundamental pieces in the minimal model program,
are singularities giving rise to Du Val (or Kleinian, ADE) singularities as hyperplane sections.
In this paper, we mainly consider toric cDV singularities.
It is known that toric cDV singularities can be classified into the following two types (e.g., see \cite[footnote~18]{Dais_toric}):
\begin{align*}
& (cA_{a+b-1})\colon \ \CC[x,y,z,w]/\bigl(xy-z^aw^b\bigr), \\
& (cD_4)\colon \ \CC[x,y,z,w]/\bigl(xyz-w^2\bigr),
\end{align*}
where $a$, $b$ are integers with $a\ge 1$ and $a\ge b\ge0$.
Note that the former one is a cDV singularity of type $cA_{a+b-1}$ and the latter one is of type $cD_4$.
These can be described combinatorially as follows.
If $R\coloneqq\CC\bigl[\sigma^\vee\cap \ZZ^3\bigr]$ is a three-dimensional Gorenstein toric ring, then we have the lattice polygon $\Delta_R$,
called the \emph{toric diagram} of $R$, as the intersection of the cone $\sigma$ and a hyperplane at height one (see Section~\ref{subsec_toric}).
The toric diagram of the above toric cDV singularities of type $cA_{a+b-1}$ and $cD_4$ take the forms
as shown in Figure~\ref{fig_hypersurf_lattice}, respectively, up to unimodular transformations (see Examples~\ref{ex_toric1} and \ref{ex_toric2}).
We will denote the polygon of type $cA_{a+b-1}$ by $\Delta(a,b)$.

\begin{figure}[!ht]
\centering
\begin{tikzpicture}
\newcommand{\edgewidth}{0.035cm}
\newcommand{\noderad}{0.08} 

\node at (-3.7,0) {$(cA_{a+b-1})$:};
\node at (4.6,0) {$(cD_4)$:};

\node at (0,0)
{\scalebox{0.9}{
\begin{tikzpicture}
\foreach \n/\a/\b in {00/0/0,10/1/0,01/0/1, 11/1/1, 20/2/0, 30/3/0, 40/4/0, 50/5/0, 21/2/1, 31/3/1} {
\coordinate (V\n) at (\a,\b);
};

\foreach \s/\t in {00/10,20/30,40/50,00/01,01/11,21/31,31/50} {
\draw [line width=\edgewidth] (V\s)--(V\t) ;};
\draw [line width=\edgewidth, dotted] (1.2,0)-- ++(0.6,0);
\draw [line width=\edgewidth, dotted] (3.2,0)-- ++(0.6,0);
\draw [line width=\edgewidth, dotted] (1.2,1)-- ++(0.6,0);

\foreach \n in {00,10,01,11,20,30,40,50,21,31}{
\draw [fill=black] (V\n) circle [radius=\noderad] ; };

\draw [line width=0.02cm, decorate, decoration={brace, amplitude=10pt}](0,1.3)--(3,1.3) node[black,midway,xshift=0cm,yshift=0.65cm] {\footnotesize $b$ segments};
\draw [line width=0.02cm, decorate, decoration={brace, mirror, amplitude=10pt}](0,-0.3)--(5,-0.3) node[black,midway,xshift=0cm,yshift=-0.65cm] {\footnotesize $a$ segments};
\end{tikzpicture}
}};

\node at (6.5,0)
{\scalebox{0.9}{
\begin{tikzpicture}
\foreach \n/\a/\b in {00/0/0,10/1/0, 20/2/0, 01/0/1, 02/0/2, 11/1/1} {
\coordinate (V\n) at (\a,\b);
};

\foreach \s/\t in {00/10,10/20,00/01,01/02,02/11,11/20} {
\draw [line width=\edgewidth] (V\s)--(V\t) ;};

\foreach \n in {00,10,20,01,02,11}{
\draw [fill=black] (V\n) circle [radius=\noderad] ; };
\end{tikzpicture}
}};

\end{tikzpicture}\vspace{-2mm}

\caption{Toric diagrams of toric cDV singularities.}\label{fig_hypersurf_lattice}
\end{figure}

The toric diagrams of type $cA_{a+b-1}$ and $cD_4$ contain no interior lattice points,
which means that the exceptional locus of a crepant resolution of a toric cDV singularity consists of curves
by the orbit-cone correspondence (e.g., see Section~\ref{subsec_crepant_resolution}).
See, e.g., \cite{reid1983minimal,wemyss_lockdown} for more details on cDV singularities.

\subsection{Summary of results}
We now summarize the main results of the paper.
Let $\Gamma$ be a dimer model and $Q\coloneqq Q_\Gamma$ be the associated quiver.
For the quiver~$Q$, we consider the space $\Theta(Q)_\RR$ of stability parameters which takes the form
\[
\Theta(Q)_\RR=\biggl\{\theta=(\theta_v)_{v\in Q_0}\in\RR^{Q_0} \mid \sum_{v\in Q_0}\theta_v=0\biggr\},
\]
where $Q_0$ is the set of vertices of $Q$.
For a stability parameter $\theta\in\Theta(Q)_\RR$, there is a moduli space $\calM_\theta(Q,\barone)$ parametrizing $\theta$-stable representations of $Q$ of dimension vector $\barone\coloneqq(1,\dots,1)$, see Section~\ref{subsec_crepant_resolution}.
Under some conditions, the moduli space $\calM_\theta\coloneqq\calM_\theta(Q,\barone)$ is
a projective crepant resolution of a three-dimensional Gorenstein toric singularity.
The space~$\Theta(Q)_\RR$ has a~wall-and-chamber structure, and
the moduli spaces (projective crepant resolutions) associated to stability parameters contained in the same chamber are isomorphic,
whereas a change of stability parameters crossing a wall would cause a change of the associated moduli space.
Precisely, let~$C$,~$C^\prime$ be adjacent chambers separated by a wall~$W$, and consider stability parameters $\theta\in C$, $\theta^\prime\in C^\prime$.
Then we sometimes have that $\calM_\theta\not\cong\calM_{\theta^\prime}$, in which case they are related by a flop and the wall~$W$ is called of type~I.
Also, if $\calM_\theta\cong\calM_{\theta^\prime}$, then although the moduli spaces are isomorphic the parametrized representations of $Q$ differ from each other, in which case the wall is either of type~0 or type~$\typeIII$, see Section~\ref{subsec_wall_chamber}.

The purposes of this paper are to detect the wall-and-chamber structure of $\Theta(Q)_\RR$
and to observe the variations of projective crepant resolutions under crossing walls in $\Theta(Q)_\RR$ for toric cDV singularities.
As we mentioned, toric cDV singularities are classified into type $cA_{a+b-1}$ and $cD_4$.
We here consider the case $cA_{n-1}$ where $n\coloneqq a+b$.
In this situation, the walls in $\Theta(Q)_\RR$ are either of type I or type $\typeIII$.
The wall-and-chamber structure of $\Theta(Q)_\RR$ and the types of walls can be determined by the combinatorics of the associated dimer models.
To state our theorem, we consider a special class of paths on a dimer model called \emph{zigzag paths} (see Definition~\ref{def_zigzag}).
A zigzag path $z$ can be considered as an element in the homology group $\rmH_1(\TT)\cong\ZZ^2$, in which we denote by $[z]\in\rmH_1(\TT)$.
By \cite[Theorem~6.1]{Gul} or \cite[Corollary~1.2]{IU_special}, for any lattice polygon $\Delta$,
there exists a consistent dimer model $\Gamma$ such that $\Delta$ coincides with the \emph{zigzag polygon} of $\Gamma$.
Under the identification of zigzag paths with elements in $\rmH_1(\TT)$, the zigzag paths on $\Gamma$ correspond one-to-one to the outer normal vectors of the primitive side segments of~$\Delta$ (see Section~\ref{subsec_dimermodel} for more details on the zigzag polygon).
Thus, we can consider a consistent dimer model giving the outer normal vectors of the polygon $\Delta(a,b)$.
In general, such a dimer model is not unique, thus we choose one of them and denote it by $\Gamma_{a,b}$.
Then we consider the set $\{u_1, \dots, u_n\}$ of zigzag paths on $\Gamma_{a,b}$ such that $[u_k]$ is either $(0,-1)$ or $(0,1)$ for $k=1, \dots, n$,
and $a=\#\{k \mid [u_k]=(0,-1)\}$, $b=\#\{k \mid [u_k]=(0,1)\}$.
Note that this set of zigzag paths is determined uniquely by the correspondence between zigzag paths on $\Gamma_{a,b}$ and the outer normal vectors of $\Delta(a,b)$.
We also define a total order $<$ on $\{u_1, \dots, u_n\}$ as $u_n<u_{n-1}<\cdots< u_2< u_1$.
Then the sequences of zigzag paths $u_1, \dots, u_n$ control the wall-and-chamber structure as follows.

\begin{Theorem}[{see Theorem~\ref{thm_main_wall} and Corollary~\ref{cor_chamber_zigzag_correspondence} for more details}]
\label{thm_summary_main}
Let $\Delta(a,b)$ be the toric diagram of the toric cDV singularity $R_{a,b}\coloneqq\CC[x,y,z,w]/\bigl(xy-z^aw^b\bigr)$.
Let $\Gamma\coloneqq\Gamma_{a,b}$ be a~consistent dimer model associated to $\Delta(a,b)$
and~$Q$ be the quiver obtained as the dual of $\Gamma$.
Let $n\coloneqq a+b$, and consider the set of zigzag paths $\{u_1, \dots, u_n\}$ as above.
Then, there exists a~one-to-one correspondence between the following sets:
\begin{itemize}\itemsep=0pt
\item[$(a)$] the set of chambers in $\Theta(Q)_\RR$,
\item[$(b)$] the set $\{\calZ_\permII=(u_{\permII(1)}, \dots, u_{\permII(n)}) \mid \permII\in\fkS_n\}$ of sequences of zigzag paths,
\end{itemize}
such that under this correspondence, if a chamber $C\subset\Theta(Q)_\RR$ corresponds to a sequence $\calZ_\permII$,
then for any $k=1, \dots, n-1$, we have the following:
\begin{itemize}\itemsep=0pt
\item[$(1)\hphantom{{}'}$]
We see that $W_{k}\coloneqq\bigl\{\theta\in\Theta(Q)_\RR \mid \sum_{v\in \calR_k}\theta_v=0\bigr\}$ is a wall of $C$, where $\calR_k\coloneqq\calR(u_{\permII(k)},\allowbreak u_{\permII(k+1)})$ is the region determined by the zigzag paths $u_{\permII(k)}, u_{\permII(k+1)}$ as shown in Figure~{\rm \ref{fig_region_intro}} $($see Section~{\rm \ref{subsec_main_wallcrossing}} for more details$)$.
\item[$(2)\hphantom{{}'}$] The wall $W_k$ is of type {\rm I} if and only if $[u_{\permII(k)}]= -[u_{\permII(k+1)}]$.
\item[$(2)'$] The wall $W_k$ is of type {\typeIII} if and only if $[u_{\permII(k)}]=[u_{\permII(k+1)}]$.
\item[$(3)\hphantom{{}'}$] Any parameter $\theta\in C$ satisfies $\sum_{v\in \calR_k}\theta_v>0$ if $u_{\permII(k)}< u_{\permII(k+1)}$.
\item[$(3)'$] Any parameter $\theta\in C$ satisfies $\sum_{v\in \calR_k}\theta_v<0$ if $u_{\permII(k)}> u_{\permII(k+1)}$.
\end{itemize}
\end{Theorem}

\begin{figure}[!ht]
\centering
\scalebox{0.85}{
\begin{tikzpicture}
\newcommand{\edgewidth}{0.06cm} 
\foreach \n/\a/\b in {1/2/0, 2/1.25/0.75, 3/2/1.5, 4/1.25/2.25, 5/2/3} {\coordinate (z1\n) at (\a+0.5,\b); \coordinate (z2\n) at (\a+4,\b);}
\draw[line width=0.03cm] (0,0) rectangle (8,3);
\draw[fill=red!20] (z11)--(z12)--(z13)--(z14)--(z15)--(z25)--(z24)--(z23)--(z22)--(z21)--(z11);
\draw[line width=\edgewidth, red] (z11)--(z12)--(z13)--(z14)--(z15); \draw[line width=\edgewidth, red] (z21)--(z22)--(z23)--(z24)--(z25);
\node at (4,1.5) {$\calR(u_{\permII(k)},u_{\permII(k+1)})$};
\node[red] at (2.5,-0.37) {\large $u_{\permII(k)}$}; \node[red] at (6,-0.37) {\large $u_{\permII(k+1)}$};
\end{tikzpicture}}

\vspace{-2mm}

\caption{The region $\calR(u_{\permII(k)}, u_{\permII(k+1)})$ which determines the wall $W_k$ given in Theorem~\ref{thm_summary_main}\,(1),
where the outer frame is a fundamental domain of the real two-torus $\TT$.}
\label{fig_region_intro}
\end{figure}

Furthermore, wall-crossing in $\Theta(Q)_\RR$ is controlled by the action of the symmetric group $\fkS_n$, and consequently,
the wall-and-chamber structure of $\Theta(Q)_\RR$ has the structure of a root system of type $A_{n-1}$.

\begin{Theorem}[{see Theorem~\ref{thm_adjacentchamber_reflection} for more details}]
Let the notation be as in Theorem~{\rm \ref{thm_summary_main}}.
The action of the adjacent transposition $s_k\in\fkS_n$ swapping $k$ and $k+1$ on $\calZ_\permII$
induces a crossing of the wall $W_k$ in $\Theta(Q)_\RR$.
In particular, the chambers in $\Theta(Q)_\RR$ can be identified with the Weyl chambers of type $A_{n-1}$.
\end{Theorem}

Note that this result has already been proved in~\cite{Wem_MMP} using a different approach, called the homological minimal model program,
whereas we show this using the combinatorics of dimer models (see Remark~\ref{rem_homMMP} for more details).

By combining these theorems with the observations in Proposition~\ref{prop_chamber_correspond_zigzag} regarding the signs of zigzag paths (see Section~\ref{subsec_zigzag_ass_chamber}), we also obtain the following result.

\begin{Theorem}
Under the correspondence in Theorem~{\rm \ref{thm_summary_main}}, suppose that a chamber $C\subset\Theta(Q)_\RR$ corresponds to a sequence $\calZ_\permII$.
Then, for any $\theta\in C$, the projective crepant resolution $\calM_\theta$ of~$\Spec R_{a,b}$
is the toric variety associated to the smooth toric fan induced from the triangulation of $\Delta(a,b)$ having the same sign with $\calZ_\permII$.
\end{Theorem}

As mentioned earlier, when two chambers are adjacent across a wall of type I, the corresponding moduli spaces are related by a flop.
Therefore, the collection of chambers connected via walls of type I (such a collection of chambers produces a \emph{GIT region} introduced in \cite{BCS_GIT}) induces the flop graph in birational geometry.
Using the result that the chambers of $\Theta(Q)_\RR$ can be identified with the Weyl chambers of type $A_{n-1}$, we also obtain the following proposition regarding GIT regions.

\begin{Proposition}[{see Proposition~\ref{prop_numberx_gallery}}]
Let $G$ be a GIT region of $\Theta(Q)_\RR$. Then $G$ contains $\frac{n!}{a! b!}$ chambers of $\Theta(Q)_\RR$ and
any projective crepant resolution of $\Spec R_{a,b}$ can be obtained as the moduli space $\calM_C$ for some $C\subset G$.
In particular, the number of GIT regions in $\Theta(Q)_\RR$ is~$a ! b!$.
\end{Proposition}

The sequences of zigzag paths appeared in Theorem~\ref{thm_summary_main} give descriptions of $\theta$-stable representations
associated to each chamber in $\Theta(Q)_\RR$.
Thus, as an application, we can track the variations of stable representations under wall-crossings
as shown in Section~\ref{sec_variation_representation}.

In addition, for the $cD_4$ case, we obtain results similar to Theorem~\ref{thm_summary_main}, as shown in Theorem~\ref{thm_main_cD4_1}, although some modifications are required.

We conclude this subsection by discussing some related topics and future perspectives concerning the results of this paper.
First, the toric diagram of a toric cDV singularity (see Figure~\ref{fig_hypersurf_lattice}) contains no interior lattice points.
As a geometric consequence of this, the exceptional locus appearing in a crepant resolution consists of curves.
This property is crucial both in this paper and in the homological minimal model program \cite{Wem_MMP}, which leads to similar results.
For a~three-dimensional Gorenstein toric singularity whose toric diagram has interior lattice points, toric divisors appear as exceptional loci upon resolution,
and the presence of these divisors complicates the situation.
However, by using dimer models, a toric divisor can be understood as a~perfect matching (see Section~\ref{subsec_cone_pm}).
Furthermore, zigzag paths, which play a key role in the main theorem of this paper, can be interpreted as the symmetric difference of two perfect matchings.
Under this background, even for more general three-dimensional Gorenstein toric singularities, it is expected that the combinatorial structure of the associated dimer model will be effective for describing crepant resolutions as moduli spaces and for understanding a wall-and-chamber structure.
On the other hand, in the context of physics, dimer models are referred to as \emph{brane tilings} and are connected to various topics such as superstring theory,
the AdS/CFT correspondence, and quiver gauge theories.
The toric cDV singularities considered in this paper also appear in physics.
Specifically, the toric cDV singularity $R_{a,b}$ of type $cA_{a+b-1}$ corresponds to a family of singularities known as $L_{a,b,a}$ (or generalized conifolds),
and the $cD_4$ type corresponds to the~$\CC^3/(\ZZ_2\times \ZZ_2)$ singularity, both of which have been extensively investigated.
For example, flops of crepant resolutions (Calabi--Yau threefolds) treated in this paper have been investigated from a~physical perspective, see, e.g.,~\cite{ura_brane}.
Furthermore, these singularities are studied from a~physical perspective using brane tilings, see, e.g.,~\cite{fra_etal_gauge}.
Therefore, we expect that the results obtained in this paper will contribute to a deeper understanding of these fields in physics.

\subsection{The structure of the paper}
In Section~\ref{sec_pre_dimer}, we prepare some notation concerning dimer models,
and discuss toric rings (singularities) arising from dimer models.
In Section~\ref{sec_quiver_rep}, we consider representations of the quiver obtained as the dual of a dimer model,
and review some basic facts concerning moduli spaces of stable representations obtained as GIT quotients,
which are projective crepant resolutions of a~three-dimensional Gorenstein toric singularity.
Since the moduli space is a smooth toric variety, it can also be understood by using a toric fan.
Thus, we review a correspondence among stable representations, cones in a toric fan, and torus orbits in a toric variety.
In particular, perfect matchings of a dimer model explain this correspondence in terms of dimer models.
In Section~\ref{section_boundaryPM}, we observe some properties of ``boundary'' perfect matchings which we will use in later sections.
In Section~\ref{sec_dimer_ab}, we focus our attention on toric cDV singularities of type $cA_{a+b-1}$,
and explain how to construct a dimer model giving rise to a projective crepant resolution of this singularity.
Section~\ref{sec_type1and3_wallcrossing} is dedicated to show our main theorems.
First, for the dimer model constructed in Section~\ref{sec_dimer_ab}, we prepare some notions such as
sequences of zigzag paths, fundamental hexagons, and jigsaw pieces, which are the main ingredients of our proof.
Then we show our main results concerning the wall-and-chamber structure of the space of stability parameters and
the variations of projective crepant resolutions,
see Theorems~\ref{thm_main_wall}, \ref{thm_adjacentchamber_reflection} and Corollary~\ref{cor_chamber_zigzag_correspondence}.
These results enable us to observe variations of stable representations and torus orbits under wall-crossings.
Thus, we study such variations in Section~\ref{sec_variation_representation}.
In Section~\ref{sec_D4}, we focus on the toric cDV singularity of type $cD_4$,
and show some results similar to the ones for type $cA_{a+b-1}$.

\section{Preliminaries on dimer models and associated toric rings}\label{sec_pre_dimer}

\subsection{Dimer models}\label{subsec_dimermodel}

In this subsection, we introduce dimer models and related notions which are originally derived from theoretical physics (e.g., \cite{fra_etal_gauge,FHK,HV}).

A \emph{dimer model} $\Gamma$ on the real two-torus $\TT\coloneqq\RR^2/\ZZ^2$ is a finite bipartite graph on~$\TT$
inducing a~polygonal cell decomposition of $\TT$, see, e.g., the left side of Figure~\ref{ex_dimer1_basic}.
Since $\Gamma$ is a bipartite graph, the set $\Gamma_0$ of nodes of $\Gamma$ is divided into two subsets $\Gamma_0^+$, $\Gamma_0^-$,
and edges of $\Gamma$ connect nodes in $\Gamma_0^+$ with those in $\Gamma_0^-$. We denote by $\Gamma_1$ the set of edges.
We color the nodes in $\Gamma_0^+$ white, and those in $\Gamma_0^-$ black throughout this paper.
A \emph{face} of $\Gamma$ is a connected component of $\TT{\setminus}\Gamma_1$. We denote by $\Gamma_2$ the set of faces.

We then consider a special class of paths on a dimer model.

\begin{Definition}\label{def_zigzag}
We say that a path on a dimer model is a \emph{zigzag path} if it makes a maximum turn to the right on a black node
and a maximum turn to the left on a white node.
An edge in a zigzag path $z$ is called a \emph{zig} (resp.\ \emph{zag}) of $z$ if it is directed from white to black (resp.\ black to white) along~$z$.
We denote by $\Zig(z)$ (resp.\ $\Zag(z)$) the set of zigs (resp.\ zags) appearing in a~zigzag path~$z$.
\end{Definition}

For example, the paths in Figure~\ref{ex_dimer1_zigzag} are all zigzag paths on the dimer model given in Figure~\ref{ex_dimer1_basic}.

\begin{figure}[!ht]
\centering
\scalebox{0.5}{
\begin{tikzpicture}
\newcommand{\noderad}{0.18cm} 
\newcommand{\edgewidth}{0.05cm} 
\newcommand{\nodewidthw}{0.05cm} 
\newcommand{\nodewidthb}{0.04cm} 
\newcommand{\zigzagwidth}{0.15cm} 
\newcommand{\zigzagcolor}{red} 

\node at (0,0){
\begin{tikzpicture}
\basicdimerA
\draw[->, line width=\zigzagwidth, rounded corners, color=\zigzagcolor] (0,2)--(W1)--(2,4);
\draw[->, line width=\zigzagwidth, rounded corners, color=\zigzagcolor] (2,0)--(4,4);
\draw[->, line width=\zigzagwidth, rounded corners, color=\zigzagcolor] (4,0)--(B3)--(7,0);
\draw[->, line width=\zigzagwidth, rounded corners, color=\zigzagcolor] (7,4)--(8,2);
\end{tikzpicture}};

\node at (11,0){
\begin{tikzpicture}
\basicdimerA
\draw[->, line width=\zigzagwidth, rounded corners, color=\zigzagcolor] (0,0)--(W1)--(0,2);
\draw[->, line width=\zigzagwidth, rounded corners, color=\zigzagcolor] (6,0)--(W4)--(B3)--(6,4);
\draw[->, line width=\zigzagwidth, rounded corners, color=\zigzagcolor] (8,2)--(B4)--(8,4);
\end{tikzpicture}};

\node at (-0.15,-5.5){
\begin{tikzpicture}
\basicdimerA
\draw[->, line width=\zigzagwidth, rounded corners, color=\zigzagcolor] (8,4)--(B4)--(7,4);
\draw[->, line width=\zigzagwidth, rounded corners, color=\zigzagcolor] (6,4)--(B3)--(5,4);
\draw[->, line width=\zigzagwidth, rounded corners, color=\zigzagcolor] (4,4)--(B2)--(3,4);
\draw[->, line width=\zigzagwidth, rounded corners, color=\zigzagcolor] (2,4)--(B1)--(1,4);
\draw[->, line width=\zigzagwidth, rounded corners, color=\zigzagcolor] (7,0)--(W4)--(6,0);
\draw[->, line width=\zigzagwidth, rounded corners, color=\zigzagcolor] (5,0)--(W3)--(4,0);
\draw[->, line width=\zigzagwidth, rounded corners, color=\zigzagcolor] (3,0)--(W2)--(2,0);
\draw[->, line width=\zigzagwidth, rounded corners, color=\zigzagcolor] (1,0)--(W1)--(0,0);
\end{tikzpicture}};

\node at (11,-5.5){
\begin{tikzpicture}
\basicdimerA
\draw[->, line width=\zigzagwidth, rounded corners, color=\zigzagcolor] (1,4)--(B1)--(W1)--(1,0);
\draw[->, line width=\zigzagwidth, rounded corners, color=\zigzagcolor] (3,4)--(B2)--(W2)--(3,0);
\draw[->, line width=\zigzagwidth, rounded corners, color=\zigzagcolor] (5,4)--(B3)--(W3)--(5,0);
\end{tikzpicture}};

\end{tikzpicture}}

\caption{Zigzag paths on the dimer model given in Figure~\ref{ex_dimer1_basic}.}\label{ex_dimer1_zigzag}
\end{figure}

Then we fix two $1$-cycles on $\TT$ generating the homology group $\rmH_1(\TT)$, and take a fundamental domain of $\TT$ along such two cycles.
Since we can consider a zigzag path $z$ on $\Gamma$ as a $1$-cycle on~$\TT$, we have the homology class $[z]\in\rmH_1(\TT)\cong\ZZ^2$,
which is called the \emph{slope} of $z$.

Also, taking the universal cover $\RR^2\rightarrow\TT$, the preimage of a dimer model $\Gamma$ determines the bipartite graph $\widetilde{\Gamma}$ on $\RR^2$, which induces a $\ZZ^2$-periodic polygonal cell decomposition of $\RR^2$.
We call $\widetilde{\Gamma}$ the universal cover of $\Gamma$.
For a zigzag path $z$ on a dimer model $\Gamma$, we also consider the lift of $z$ to the universal cover $\widetilde{\Gamma}$,
that is, for $\alpha\in\ZZ$, let $\widetilde{z}{(\alpha)}$ denote a zigzag path on $\widetilde{\Gamma}$ whose projection on $\Gamma$ coincides with $z$.
When we do not need to specify these paths, we simply denote each of them by $\widetilde{z}$.
Then, we see that a zigzag path on $\widetilde{\Gamma}$ is either periodic or infinite in both directions.

In the rest of this paper, we assume that any dimer model satisfies the \emph{consistency condition}, see Definition~\ref{def_consistent} below.
In the literature, there are several consistency conditions equivalent to the one given in Definition~\ref{def_consistent} (see, e.g., \cite{Boc_consistent,IU_consistent}).

\begin{Definition}[{see \cite[Definition~3.5]{IU_consistent}}]\label{def_consistent}
A dimer model is said to be (\emph{zigzag}) \emph{consistent} if it satisfies the following conditions:
\begin{itemize}\itemsep=0pt
\item [(1)] there is no homologically trivial zigzag path, that is, $[z]\neq (0,0)$,
\item [(2)] no zigzag path on the universal cover has a self-intersection,
\item [(3)] any pair of zigzag paths on the universal cover does not intersect each other in the same direction more than once (note that two zigzag paths are said to \emph{intersect} if they share an edge, not a node).
In other words, if a pair of zigzag paths $(\widetilde{z},\widetilde{w})$ on the universal cover intersect at two edges $a_1$ and $a_2$,
and $\widetilde{z}$ passes through $a_1$ first, followed by $a_2$,
then $\widetilde{w}$ passes through $a_2$ first, followed by $a_1$.
\end{itemize}
\end{Definition}

Note that any edge of a dimer model is contained in at most two zigzag paths.
By the condition (2) in Definition~\ref{def_consistent},
if a dimer model is consistent, then any edge is contained in exactly two zigzag paths and any slope is a primitive element.
For example, by observing the zigzag paths in Figure~\ref{ex_dimer1_zigzag}, we see that the dimer model given in Figure~\ref{ex_dimer1_basic} is consistent.

Then, for a consistent dimer model $\Gamma$, we can assign the lattice polygon called the \emph{characteristic polygon} as explained in
\cite[Section~5.2]{IU_special}.
Here, we introduce another polygon called the \emph{zigzag polygon}, which coincides with the characteristic polygon up to translation, see \cite[Theorem~12.1]{IU_special}.
Let $[z]$ be the slope of a zigzag path $z$ on $\Gamma$, which is not homologically trivial.
By normalizing $[z]\coloneqq (a,b)\in\ZZ^2$, we consider it as an element of the unit circle $S^1$:
\[
\frac{(a,b)}{\sqrt{a^2+b^2}}\in S^1.
\]
Then, the set of slopes has a natural cyclic order along $S^1$.
We consider the sequence $([z_i])_{i=1}^k$ of slopes of zigzag paths on $\Gamma$ such that
they are cyclically ordered starting from $[z_1]$, where $k$ is the number of zigzag paths.
We note that in general, some slopes may coincide.
We then set another sequence $(w_i)_{i=1}^{k}$ in $\ZZ^2$ defined as $w_0=(0,0)$ and
\[
w_{i+1}=w_i+[z_{i+1}]^\prime, \qquad i=0,1,\dots,k-1.
\]
Here, $[z_{i+1}]^\prime\in\ZZ^2$ is the element obtained from $[z_{i+1}]$ by rotating 90 degrees in the anti-clockwise direction.
One can see that $w_k=(0,0)$ since the sum of all slopes is equal to zero.
We call the convex hull of $\{w_i\}_{i=1}^{k}$ the \emph{zigzag polygon} and denote it by $\Delta_\Gamma$.
Note that there are several choices of an initial zigzag path $z_1$, but the zigzag polygon is determined up to unimodular transformations,
and it does not affect our problem, see Section~\ref{subsec_toric}.

By definition, we see that the slope of a zigzag path is an outer normal vector of some side of $\Delta_\Gamma$,
and the number of zigzag paths having the same slope $v\in\ZZ^2$ coincides with the number of primitive segments of
the side of $\Delta_\Gamma$ whose outer normal vector is $v$.
Here, a \emph{primitive side segment} of $\Delta_\Gamma$ means a line segment on the boundary of $\Delta_\Gamma$
divided by a pair of lattice points not containing any lattice point in its interior.

\begin{Example}
We consider the dimer model in Figure~\ref{ex_dimer1_basic} and its zigzag paths as in Figure~\ref{ex_dimer1_zigzag}.
Then, we have the cyclically ordered sequence of slopes
\[
((0,-1), (0,-1), (0,-1), (1,1), (0,1), (0,1), (-1,0)),
\]
where we take a $\ZZ$-basis of $\rmH_1(\TT)\cong\ZZ^2$ along the vertical and horizontal lines of the fundamental domain of $\TT$.
Thus, the zigzag polygon is $\Delta(3,2)$ as in Figure~\ref{ex_zigzagpolygon1}.

\begin{figure}[!ht]
\centering
\scalebox{1}{
\begin{tikzpicture}
\newcommand{\edgewidth}{0.035cm}
\newcommand{\noderad}{0.08} 
\foreach \n/\a/\b in {00/0/0,10/1/0,01/0/1,11/1/1,20/2/0,21/2/1,30/3/0} {
\coordinate (V\n) at (\a,\b); };
\draw [line width=\edgewidth, fill=white] (V00)--(V30)--(V21)--(V01)--(V00) ;
\foreach \n in {00,10,01,11,20,21,30}{
\draw [fill=black] (V\n) circle [radius=\noderad] ; };
\end{tikzpicture}}

\vspace{-2mm}

\caption{The zigzag polygon $\Delta(3,2)$ of the dimer model given in Figure~\ref{ex_dimer1_basic}.}\label{ex_zigzagpolygon1}
\end{figure}
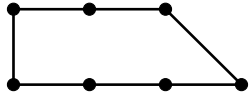
\end{Example}

As we have seen in this section, we have a lattice polygon from a dimer model.
On the other hand, any lattice polygon can be described as the zigzag polygon of a consistent dimer model.
However, we note that such a consistent dimer model is not unique in general.

\begin{Theorem}[{see, e.g., \cite[Theorem~6.1]{Gul}, \cite[Corollary~1.2]{IU_special}}]
\label{exist_dimer}
For any lattice polygon $\Delta$, there exists a consistent dimer model $\Gamma$ such that $\Delta=\Delta_\Gamma$.
\end{Theorem}

\subsection{Toric rings associated to dimer models}\label{subsec_toric}

Let $\Gamma$ be a consistent dimer model.
We next consider the cone $\sigma_\Gamma$ over the zigzag polygon $\Delta_\Gamma$,
that is, $\sigma_\Gamma$ is the cone whose section on the hyperplane at height one is $\Delta_\Gamma$.

Let $\sfN\coloneqq \ZZ^3$ be a lattice and $\sfM\coloneqq\Hom_\ZZ(\sfN,\ZZ)$ be the dual lattice of $\sfN$.
We set $\sfN_\RR\coloneqq\sfN\otimes_\ZZ \RR$ and $\sfM_\RR\coloneqq\sfM\otimes_\ZZ \RR$.
We denote the standard inner product by $\langle\,\, , \,\,\rangle\colon \sfM_\RR\times \sfN_\RR\rightarrow \RR$.
For the vertices $\widetilde{v}_1,\dots, \widetilde{v}_n\in\ZZ^2$ of $\Delta_\Gamma$,
we let $v_i\coloneqq (\widetilde{v_i}, 1)\in\sfN$ $(i=1, \dots, n)$. The cone $\sigma_\Gamma$ over $\Delta_\Gamma$ is defined as
\[
\sigma_\Gamma\coloneqq\RR_{\ge0}v_1+\cdots +\RR_{\ge0}v_n \subset\sfN_\RR.
\]
Then, we consider the dual cone
\[
\sigma_\Gamma^\vee\coloneqq\{x\in\sfM_\RR \mid \langle x, v_i\rangle\ge0 \text{ for any } i=1,\dots,n \},
\]
where $\langle - ,- \rangle$ is the natural inner product.
Using this cone we can define the \emph{toric ring $($toric singularity$)$ $R$ associated to $\Gamma$} as
\[
R_\Gamma\coloneqq \CC\bigl[\sigma_\Gamma^\vee\cap\sfM\bigr]=\CC\bigl[t_1^{a_1}t_2^{a_2} t_3^{a_3}\mid (a_1,a_2,a_3)\in\sigma_\Gamma^\vee\cap\sfM\bigr].
\]
By construction, $R_\Gamma$ is Gorenstein in dimension three.
We note that any three-dimensional Gorenstein toric ring can be described with this form.
Precisely,
let $\sigma$ be a strongly convex rational polyhedral cone in $\sfN_\RR$ which defines a three-dimensional Gorenstein toric ring $R$.
Then, it is known that, after applying an appropriate unimodular transformation (which does not change the associated toric ring up to isomorphism) to $\sigma$,
the cone $\sigma$ can be described as the cone over a certain lattice polygon $\Delta_R$.
We call the lattice polygon $\Delta_R$ the \emph{toric diagram} of~$R$.
By Theorem~\ref{exist_dimer}, there exists a consistent dimer model $\Gamma$ such that $\Delta_\Gamma=\Delta_R$ for any three-dimensional Gorenstein toric ring $R$, in which case we have $R=R_\Gamma$.
We note that unimodular transformations and parallel translations of~$\Delta_R$ do not change the associated toric ring in the following sense.
Let $\Delta^\prime\subset\RR^2$ be a lattice polygon obtained by applying a unimodular transformation or a parallel translation to $\Delta_R$
and let $\sigma^\prime\subset\sfN_\RR$ be the cone over $\Delta^\prime$.
Then, we see that $\sigma$ and $\sigma^\prime$ are unimodularly equivalent, and hence the associated toric rings are isomorphic.

\begin{Example}
\label{ex_toric1}
Let $\Delta(a,b)$ be the trapezoid shown in the left of Figure~\ref{fig_hypersurf_lattice}, where $a, b$ are integers with $a\ge 1$ and $a\ge b\ge0$.
By Theorem~\ref{exist_dimer}, there exists a consistent dimer model whose zigzag polygon is $\Delta(a,b)$, which will be constructed in Section~\ref{sec_dimer_ab}.
For simplicity, we fix the lower left vertex of $\Delta(a,b)$ as the origin and consider the cone $\sigma_{a,b}\coloneqq \sum_{i=1}^4\RR_{\ge0} v_i$ over $\Delta(a,b)$, where
\[
v_1\coloneqq (0,0,1), \qquad v_2\coloneqq (a,0,1), \qquad v_3\coloneqq (b,1,1), \qquad v_4\coloneqq (0,1,1).
\]
Then we see that
\[
\CC[\sigma_{a,b}^\vee\cap\sfM]=\CC\bigr[t_1, t_2, t_2^{-1}t_3, t_1^{-1}t_2^{b-a}t_3^a\bigl].
\]
by computing the Hilbert basis (see \cite[Proposition~1.2.23]{CLS}).
We easily show that $\CC\bigl[\sigma_{a,b}^\vee\cap\sfM\bigr]$ is isomorphic to the toric cDV singularity
$R_{a,b}\coloneqq \CC[x,y,z,w]/\bigl(xy-z^aw^b\bigr)$ of type $cA_{a+b-1}$ given in Section~\ref{subsec_cDVtoric}.
Note that $R_{a,b}$ is not an isolated singularity except the case $a=b=1$.
\end{Example}

\begin{Example}\label{ex_toric2}
Let $\Delta$ be the triangle shown in the right of Figure~\ref{fig_hypersurf_lattice}.
A consistent dimer model giving rise to $\Delta$ as the zigzag polygon will be given in Section~\ref{sec_D4}.
As in the previous example, we consider the cone $\sigma\coloneqq \sum_{i=1}^3\RR_{\ge0} v_i$ over $\Delta$, where
\[
v_1\coloneqq (0,0,1), \qquad v_2\coloneqq (2,0,1), \qquad v_3\coloneqq (0,2,1),
\]
and we have
\[
\CC[\sigma^\vee\cap\sfM]=\CC\bigl[t_1, t_2, t_3, t_1^{-1}t_2^{-1}t_3^2\bigr].
\]
We easily show that $\CC\bigl[\sigma^\vee\cap\sfM\bigr]$ is isomorphic to the toric cDV singularity
$\CC[x,y,z,w]/\bigl(xyz-w^2\bigr)$ of type $cD_4$ given in Section~\ref{subsec_cDVtoric}.
\end{Example}

\section{Preliminaries on moduli spaces of quiver representations}\label{sec_quiver_rep}

In this section, we review moduli spaces of quiver representations arising from consistent dimer models.
For this purpose, we introduce quivers associated to dimer models and their representations.

\subsection{Quivers associated to dimer models}\label{subsec_quiver}

Let $\Gamma$ be a dimer model.
As the dual of $\Gamma$, we obtain the quiver $Q_\Gamma$ associated to $\Gamma$, which is embedded in $\TT$, as follows.
The quiver $Q=Q_\Gamma$ consists of
\begin{itemize}\itemsep=0pt
\item the set $Q_0$ of vertices defined as the dual to the set $\Gamma_2$ of faces of $\Gamma$,
\item the set $Q_1$ of arrows defined as the dual to the set $\Gamma_1$ of edges of $\Gamma$, where the orientation of any arrow is defined so that the white node is on the right of the arrow.
\end{itemize}
For example, Figure~\ref{ex_dimer1_basic} shows a dimer model and the associated quiver.
We consider two maps
\[
\hd \colon \ Q_1\to Q_0 \qquad \text{and}\qquad \tl \colon \ Q_1\to Q_0
\]
that assign to each arrow $a\in Q_1$ its head and tail, respectively.

A \emph{path} of length $r\ge 1$ is a finite sequence of arrows $\gamma=a_1\cdots a_r$ with $\hd(a_i)=\tl(a_{i+1})$ for $i=1, \dots, r-1$.
We define $\tl(a)=\tl(a_1)$, $\hd(a)=\hd(a_r)$ for a path $\gamma=a_1\cdots a_r$.
A path $\gamma$ is called a \emph{cycle} if $\hd(\gamma)=\tl(\gamma)$, and a \emph{loop} if it is a cycle of length one.

A \emph{relation} in $Q$ is a $\CC$-linear combination of paths of length at least two having the same head and tail.
We consider relations in $Q$ defined as follows.
For each arrow $a\in Q_1$, there exist two paths $\gamma_a^+$, $\gamma_a^-$ such that $\hd(\gamma_a^\pm)=\tl(a)$, $\tl(\gamma_a^\pm)=\hd(a)$ and
$\gamma_a^+$ (resp.\ $\gamma_a^-$) goes around the white (resp.\ black) node incident to the edge dual to $a$ clockwise (resp.\ counterclockwise) as shown in Figure~\ref{fig_relation_quiver}.
We define the set of relations $\calJ_Q\coloneqq \{\gamma_a^+-\gamma_a^- \mid a\in Q_1\}$ and
call the pair $(Q,\calJ_Q)$ the \emph{quiver with relations} associated to~$\Gamma$.

\begin{figure}[!ht]\centering
\scalebox{0.7}{
\begin{tikzpicture}[myarrow/.style={black, -latex}]
\newcommand{\noderad}{0.2cm} 
\newcommand{\edgewidth}{0.05cm} 
\newcommand{\arrowwidth}{0.06cm} 
\newcommand{\nodewidthw}{0.06cm} 
\newcommand{\nodewidthb}{0.04cm} 
\coordinate (B1) at (-1,0); \coordinate (W1) at (1,0);
\foreach \r/\n in {45/a,15/b,-45/c} {\path (W1) ++(\r:2) coordinate (W1\n); }
\foreach \r/\n in {135/a,165/b,225/c} {\path (B1) ++(\r:2) coordinate (B1\n); }
\foreach \n in {a,b,c} {\draw[line width=\edgewidth] (W1)--(W1\n); \draw[line width=\edgewidth] (B1)--(B1\n); }
\draw[line width=\edgewidth] (B1)--(W1);
\draw[line width=\edgewidth, loosely dotted] ([shift={(1,0)}]-35:1) arc(-30:10:1);
\draw[line width=\edgewidth, loosely dotted] ([shift={(-1,0)}]215:1) arc(210:170:1);
\filldraw [fill=black, line width=\nodewidthb] (B1) circle [radius=\noderad] ;
\filldraw [fill=white, line width=\nodewidthw] (W1) circle [radius=\noderad] ;
\foreach \n/\r/\m in {1/-90/1.5,2/90/1.5,3/15/2.5,4/-3/2.8,5/-20/2.5} {\coordinate (RH\n) at (\r:\m); }
\foreach \n/\r/\m in {1/-90/1.5,2/90/1.5,3/165/2.5,4/183/2.8,5/200/2.5} {\coordinate (LH\n) at (\r:\m); }
\foreach \s/\t in {1/2,2/3,3/4,5/1} {\draw[->, line width=\arrowwidth, shorten >=0.08cm, shorten <=0.08cm, myarrow, blue] (RH\s)--(RH\t);
\draw[->, line width=\arrowwidth, shorten >=0.08cm, shorten <=0.05cm, myarrow, blue] (LH\s)--(LH\t); }
\draw[line width=\arrowwidth, shorten >=0.15cm, shorten <=0.15cm, blue, loosely dotted] (RH4)--(RH5);
\draw[line width=\arrowwidth, shorten >=0.15cm, shorten <=0.15cm, blue, loosely dotted] (LH4)--(LH5);
\node[blue] at (0.3,0.4) {\Large$a$}; \node[blue] at (3.3,0) {\Large$\gamma_a^+$}; \node[blue] at (-3.3,0) {\Large$\gamma_a^-$};
\end{tikzpicture}}

\vspace{-2mm}

\caption{An example of $\gamma_a^+$ and $\gamma_a^-$.}\label{fig_relation_quiver}
\end{figure}
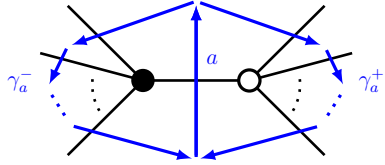

Since we consider moduli spaces of quiver representations satisfying a stability condition in this paper,
we now introduce representations of quivers.
A \emph{representation} of $(Q,\calJ_Q)$ consists of
\begin{itemize}\itemsep=0pt
\item a set of $\CC$-vector spaces $\{M_v \mid v\in Q_0\}$,
\item a set of $\CC$-linear maps $\varphi_a\colon M_{\tl(a)}\rightarrow M_{\hd(a)}$ satisfying the relations in $\calJ_Q$,
i.e., $\varphi_{\gamma_a^+}=\varphi_{\gamma_a^-}$ for any $a\in Q_1$.
\end{itemize}
\noindent
Here, for a path $\gamma=a_1\cdots a_r$, the map $\varphi_\gamma$ is defined as the composite $\varphi_{a_1}\cdots\varphi_{a_r}$ of $\CC$-linear maps
(note that in this paper, a composition of morphism $fg$ means we first apply $f$ then $g$).
We assume that a representation $M=((M_v)_{v\in Q_0}, (\varphi_a)_{a\in Q_1})$ of $(Q,\calJ_Q)$ is finite-dimensional
(i.e., $\dim_\CC M_v$ is finite for all $v\in Q_0$) and call ${\dd}\coloneqq(\dim_\CC M_v)_{v\in Q_0}$ the \emph{dimension vector} of $M$.

For example, if we label each arrow of the quiver $Q$ given in Figure~\ref{ex_dimer1_basic} as shown on the left of Figure~\ref{fig_quiver_rep},
then a representation of $(Q,\calJ_Q)$ takes the form shown on the right of Figure~\ref{fig_quiver_rep}.
Here, each $M_i$ $(i\in Q_0)$ is a $\CC$-vector space, and the arrows between them are $\CC$-linear maps satisfying the relations in $\calJ_Q$.
More precisely, for example, due to the relation $\gamma_{a_1}^+-\gamma_{a_1}^-=a_1^*b_1-b_2a_1^*$,
the corresponding $\CC$-linear maps satisfy $\varphi_{a_1^*}\varphi_{b_1}=\varphi_{b_2}\varphi_{a_1^*}$.

\begin{figure}[!ht]
\centering

\begin{tikzpicture}[sarrow/.style={black, -latex},tarrow/.style={black, latex-}]
\newcommand{\noderad}{0.18cm} 
\newcommand{\edgewidth}{0.05cm} 
\newcommand{\nodewidthw}{0.05cm} 
\newcommand{\nodewidthb}{0.04cm} 
\newcommand{\arrowwidth}{0.06cm} 
\newcommand{\arrowwidthrep}{0.03cm} 

\node at (0,0) {\scalebox{0.6}{
\begin{tikzpicture}
\foreach \n/\a/\b in {1/1.5/3, 2/3.5/3, 3/5.5/3, 4/7.5/3} {\coordinate (B\n) at (\a,\b);}
\foreach \n/\a/\b in {1/0.5/1,2/2.5/1, 3/4.5/1, 4/6.5/1} {\coordinate (W\n) at (\a,\b);}
\draw[line width=\edgewidth] (0,0) rectangle (8,4);
\foreach \w/\b in {1/1,2/2,3/3,4/3} {\draw[line width=\edgewidth] (W\w)--(B\b); };
\foreach \w/\s/\t in {1/0/2,1/0/0,1/1/0,2/2/0,2/3/0,3/4/0,3/5/0,4/6/0,4/7/0} {\draw[line width=\edgewidth] (W\w)--(\s,\t); };
\foreach \b/\s/\t in {1/1/4,1/2/4,2/3/4,2/4/4,3/5/4,3/6/4,4/7/4,4/8/4,4/8/2} {\draw[line width=\edgewidth] (B\b)--(\s,\t); };
\foreach \x in {1,2,3,4} {\filldraw [fill=black, line width=\nodewidthb] (B\x) circle [radius=\noderad] ;};
\foreach \x in {1,2,3,4} {\filldraw [fill=white, line width=\nodewidthw] (W\x) circle [radius=\noderad] ;};

\node[blue] (V0) at (0.5,3) {\Large$0$}; \node[blue] (V1) at (2,2) {\Large$1$};
\node[blue] (V2) at (4,2) {\Large$2$}; \node[blue] (V3) at (5.5,1) {\Large$3$}; \node[blue] (V4) at (7,2) {\Large$4$};

\draw[shorten >=0.0cm, shorten <=0.0cm, sarrow, blue, line width=\arrowwidth] (V0)--(V1) node[midway,xshift=-0.15cm,yshift=0.35cm] {\large$a_0$};
\draw[shorten >=0.0cm, shorten <=0.0cm, sarrow, blue, line width=\arrowwidth] (V1)--(V2) node[midway,xshift=-0.15cm,yshift=0.25cm] {\large$a_1$}; 
\draw[shorten >=0.0cm, shorten <=0.0cm, sarrow, blue, line width=\arrowwidth] (V2)--(V3) node[midway,xshift=-0.1cm,yshift=0.35cm] {\large$a_2$};
\draw[shorten >=0.0cm, shorten <=0.0cm, sarrow, blue, line width=\arrowwidth] (V3)--(V4) node[midway,xshift=-0.3cm,yshift=0.1cm] {\large$a_3$};
\draw[shorten >=0.0cm, shorten <=0.0cm, sarrow, blue, line width=\arrowwidth] (V0)--(0,4) ;
\draw[shorten >=0.0cm, shorten <=0.0cm, sarrow, blue, line width=\arrowwidth] (1,4)--(V0);
\draw[shorten >=-0.1cm, shorten <=0.0cm, sarrow, blue, line width=\arrowwidth] (0,2.666)--(V0);
\draw[shorten >=0.0cm, shorten <=0.0cm, sarrow, blue, line width=\arrowwidth] (V1)--(2,4);
\draw[shorten >=0.0cm, shorten <=0.0cm, sarrow, blue, line width=\arrowwidth] (2,0)--(V1) node[midway,xshift=-0.2cm,yshift=-0.2cm] {\large$b_1$};
\draw[shorten >=0.0cm, shorten <=0.0cm, sarrow, blue, line width=\arrowwidth] (V1)--(1,0) node[midway,xshift=-0.25cm,yshift=0.2cm] {\large$a_0^*$};
\draw[shorten >=0.0cm, shorten <=0.0cm, sarrow, blue, line width=\arrowwidth] (3,4)--(V1);
\draw[shorten >=0.0cm, shorten <=0.0cm, sarrow, blue, line width=\arrowwidth] (V2)--(4,4);
\draw[shorten >=0.0cm, shorten <=0.0cm, sarrow, blue, line width=\arrowwidth] (4,0)--(V2) node[midway,xshift=-0.2cm,yshift=-0.2cm] {\large$b_2$};
\draw[shorten >=0.0cm, shorten <=0.0cm, sarrow, blue, line width=\arrowwidth] (V2)--(3,0) node[midway,xshift=-0.25cm,yshift=0.2cm] {\large$a_1^*$};
\draw[shorten >=0.0cm, shorten <=0.0cm, sarrow, blue, line width=\arrowwidth] (V3)--(5,0);
\draw[shorten >=0.0cm, shorten <=0.0cm, sarrow, blue, line width=\arrowwidth] (5,4)--(V2) node[midway,xshift=0.25cm,yshift=-0.1cm] {\large$a_2^*$};
\draw[shorten >=0.0cm, shorten <=0.0cm, sarrow, blue, line width=\arrowwidth] (V4)--(6,4) node[midway,xshift=-0.25cm,yshift=-0.1cm] {\large$a_3^*$};
\draw[shorten >=0.0cm, shorten <=0.0cm, sarrow, blue, line width=\arrowwidth] (6,0)--(V3);
\draw[shorten >=0.0cm, shorten <=0.0cm, sarrow, blue, line width=\arrowwidth] (V4)--(7,0) node[midway,xshift=0.25cm,yshift=-0.2cm] {\large$b_4$};
\draw[shorten >=0.0cm, shorten <=0.0cm, sarrow, blue, line width=\arrowwidth] (7,4)--(V4);
\draw[shorten >=0.0cm, shorten <=0.0cm, sarrow, blue, line width=\arrowwidth] (8,0)--(V4) node[midway,xshift=0.2cm,yshift=0.2cm] {\large$a_4^*$};
\draw[shorten >=0.0cm, shorten <=0.0cm, sarrow, blue, line width=\arrowwidth] (V4)--(8,2.666) node[midway,xshift=-0.1cm,yshift=-0.35cm] {\large$a_4$};
\end{tikzpicture}}};

\node at (7,0) {\scalebox{0.9}{
\begin{tikzpicture}[sarrow/.style={black, -latex},tarrow/.style={black, latex-}]

\foreach \a/\b/\n in {0/0/M_1, 2/0/M_2, 4/0/M_3, 6/0/M_4, 3/2/M_0}{\node (V\a\b) at (\a,\b) {$\n$};};

\begin{scope}[transform canvas={yshift=-0.08cm}]
\foreach \a/\b/\q in {00/20/a_1,20/40/a_2,40/60/a_3}
{\draw[sarrow, line width=\arrowwidthrep] (V\a)--(V\b)node[midway,xshift=0cm,yshift=-0.15cm] {\scriptsize$\varphi_{\q}$}; };
\draw[sarrow, line width=\arrowwidthrep] (V32)--(V00)node[midway,xshift=0.2cm,yshift=-0.15cm] {\scriptsize$\varphi_{a_0}$};
\draw[sarrow, line width=\arrowwidthrep] (V60)--(V32)node[midway,xshift=-0.1cm,yshift=-0.15cm] {\scriptsize$\varphi_{a_4}$};
\end{scope}

\begin{scope}[transform canvas={yshift=0.08cm}]
\foreach \a/\b/\q in {00/20/a_1^*,20/40/a_2^*,40/60/a_3^*}
{\draw[tarrow, line width=\arrowwidthrep] (V\a)--(V\b)node[midway,xshift=0cm,yshift=0.18cm] {\scriptsize $\varphi_{\q}$}; };
\draw[tarrow, line width=\arrowwidthrep] (V32)--(V00)node[midway,xshift=-0.15cm,yshift=0.18cm] {\scriptsize $\varphi_{a_0^*}$};
\draw[tarrow, line width=\arrowwidthrep] (V60)--(V32)node[midway,xshift=0.1cm,yshift=0.18cm] {\scriptsize $\varphi_{a_4^*}$};
\end{scope}

\foreach \a in {00,20,60} {\draw [line width=\arrowwidthrep,sarrow] (V\a) to [in=300,out=240,loop] ();};
\foreach \a/\q in {00/b_1,20/b_2,60/b_4} {
\path (V\a) ++(-50:1cm) coordinate (V\a L); \node at (V\a L) {\scriptsize $\varphi_{\q}$}; }
\end{tikzpicture}
}};

\end{tikzpicture}

\vspace{-2mm}

\caption{A representation of $(Q,\calJ_Q)$ associated to a dimer model given in Figure~\ref{ex_dimer1_basic}.}\label{fig_quiver_rep}
\end{figure}

Let $M$, $M^\prime$ be representations of $(Q,\calJ_Q)$.
A morphism from $M$ to $M^\prime$ is a family of $\CC$-linear maps $\{f_v\colon M_v\rightarrow M_v^\prime\}_{v\in Q_0}$ such that
$\varphi_af_{\hd(a)}=f_{\tl(a)}\varphi_a^\prime$ for any arrow $a\in Q_1$.
We say that representations $M$ and $M^\prime$ are \emph{isomorphic}, if $f_v$ is an isomorphism of vector spaces for all $v\in Q_0$.
A representation $N$ of $(Q,\calJ_Q)$ is called a \emph{subrepresentation} of $M$ if there is an injective morphism $N\rightarrow M$.

In the rest of this paper, we consider representations of the quiver with relations $(Q,\calJ_Q)$ associated to a consistent dimer model $\Gamma$
and assume that the dimension vector of any representation is $\barone\coloneqq(1,\dots,1)$.

\subsection{Stability parameters and crepant resolutions}\label{subsec_crepant_resolution}

In this subsection, we introduce moduli spaces parametrizing quiver representations satisfying a certain stability condition.

First, we consider the weight space
\[
\Theta(Q)\coloneqq\biggl\{\theta=(\theta_v)_{v\in Q_0}\in\ZZ^{Q_0} \mid \sum_{v\in Q_0}\theta_v=0\biggr\}
\]
and let $\Theta(Q)_\RR\coloneqq\Theta(Q)\otimes_\ZZ\RR$.
We call an element $\theta\in\Theta(Q)_\RR$ a \emph{stability parameter}.

Let $M$ be a representation of $(Q,\calJ_Q)$ of dimension vector $\barone$.
For a subrepresentation~$N$ of~$M$, we define $\theta(N)\coloneqq\sum_{v\in Q_0}\theta_v(\dim_\CC N_v)$, and hence $\theta(M)=0$ in particular.
For a stability parameter $\theta\in\Theta(Q)_\RR$, we introduce $\theta$-stable representations as follows.

\begin{Definition}[{see \cite{King_moduli}}]
Let $\theta\in\Theta(Q)_\RR$.
We say that a representation $M$ is \emph{$\theta$-semistable} if $\theta(N)\ge0$ for any subrepresentation $N$ of $M$, and
$M$ is \emph{$\theta$-stable} if $\theta(N)>0$ for any non-zero proper subrepresentation $N$ of $M$.
Then, we say that $\theta$ is \emph{generic} if every $\theta$-semistable representation is $\theta$-stable.
\end{Definition}

By \cite[Proposition~5.2]{King_moduli}, for any $\theta\in\Theta(Q)_\RR$, one can construct the coarse moduli space $\overline{\calM}_\theta(Q,\calJ_Q, \barone)$ of $S$-equivalence (Seshadri equivalence) classes of $\theta$-semistable representations of dimension vector $\barone$
(i.e., $\theta$-semistable representations whose Jordan--H\"{o}lder filtrations have the same composition factors).
By \cite[Proposition~5.3]{King_moduli}, for a generic parameter $\theta\in\Theta(Q)_\RR$,
one can construct the fine module space $\calM_\theta(Q,\calJ_Q, \barone)$ parametrizing isomorphism classes of $\theta$-stable representations of dimension vector $\barone$ as a GIT (geometric invariant theory) quotient.
Moreover, this moduli space gives a crepant resolution as follows.

\begin{Theorem}[{see \cite[Theorems~6.3 and 6.4]{IU_moduli}, \cite[Corollary~1.2]{IU_anycrepant}}]
\label{thm_crepant_dimer}
Let $\Gamma$ be a consistent dimer model, and $Q$ be the associated quiver.
Let $R$ be the three-dimensional Gorenstein toric ring associated to $\Gamma$.
Then, for a generic parameter $\theta\in\Theta(Q)_\RR$, the moduli space $\calM_\theta(Q,\calJ_Q, \barone)$ is a~smooth toric Calabi--Yau threefold
and a projective crepant resolution of $\Spec R$.

Moreover, any projective crepant resolution of $\Spec R$ can be obtained as the moduli space $\calM_\theta(Q,\calJ_Q, \barone)$ for some generic parameter $\theta\in\Theta(Q)_\RR$.
\end{Theorem}

In the following, we let $\calM_\theta=\calM_\theta(Q,\calJ_Q, \barone)$ and $\overline{\calM}_\theta=\overline{\calM}_\theta(Q,\calJ_Q, \barone)$ for simplicity.
Let $G$ be the subset of isomorphism classes of representations of $(Q,\calJ_Q)$ defined as follows:
\[
G\coloneqq \bigl\{ [((M_v)_{v\in Q_0}, (\varphi_a)_{a\in Q_1}) ] \mid \varphi_a\in \CC^\times \ \text{for any $a\in Q_1$} \bigr\}.
\]
This has the structure of an algebraic torus with a multiplication defined as
\begin{equation}
\label{eq_torus_multi}
 [((M_v)_{v\in Q_0}, (\varphi_a)_{a\in Q_1}) ]\cdot [((M_v)_{v\in Q_0}, (\varphi^\prime_a)_{a\in Q_1}) ]
= [((M_v)_{v\in Q_0}, (\varphi_a\varphi^\prime_a)_{a\in Q_1}) ].
\end{equation}
Since any representation in $G$ has no proper subrepresentation, it is $\theta$-stable for any $\theta$,
and hence~$G$ is contained in~$\overline{\calM}_\theta$ for any $\theta$.
If $\theta$ is generic, then~$G$ is the open dense torus contained in the toric variety $\calM_\theta$
and~$G$ acts on $\calM_\theta$ by the multiplication~\eqref{eq_torus_multi}, see~\cite{IU_moduli} for more details.

Since $\calM_\theta$ is a fine moduli space for a generic parameter $\theta\in\Theta(Q)_\RR$,
it carries a universal family $\calT_\theta\coloneqq \bigoplus_{v\in Q_0}\calL_v$
of $\theta$-stable representations of $(Q,\calJ_Q)$ of dimension vector $\barone$, called a~\emph{tautological bundle} of $\calM_\theta$,
where $\calL_v$ is a line bundle on $\calM_\theta$ for any $v\in Q_0$.
In general, there is an ambiguity of a choice of $\calT_\theta$, that is, by tensoring a line bundle to $\calT_\theta$,
we have a vector bundle having the same properties as $\calT_\theta$.
Thus, we fix a vertex of $Q$, which we denote by $0\in Q_0$, as a~specific one, and normalize the tautological bundle so that $\calL_0\cong\calO_{\calM_\theta}$.
On the other hand, in our setting, $\calT_\theta$ is a tilting bundle (see \cite[Theorem~1.4]{IU_special}), and hence it induces an equivalence $\calD^\mathrm{b}(\mathsf{coh}\,\calM_\theta)\cong\calD^\mathrm{b}(\mathsf{mod} \End_{\calM_\theta}(\calT_\theta))$, see~\cite{Bondal_SEC,Rickard_morita}.
By \cite[Corollary~4.15]{IW_singular}, we see that $\End_{\calM_\theta}(\calT_\theta)$
is a \emph{non-commutative crepant resolution} (\emph{NCCR}) in the sense of~\cite{VdB_NCCR}.
This was also proved in~\cite{Bro_dimer} using another method.

Since $\calM_\theta$ is a smooth toric variety for a generic parameter $\theta$, it can be described by using a~smooth toric fan.
Namely, there is a certain smooth subdivision $\Sigma_\theta$ of the cone $\sigma_\Gamma$
such that the toric variety $X_{\Sigma_\theta}$ associated to $\Sigma_\theta$ is isomorphic to $\calM_\theta$ (see, e.g., \cite[Chapter~11]{CLS}).
We denote the set of $r$-dimensional cones in $\Sigma_\theta$ by $\Sigma_\theta(r)$ where $r=1, 2, 3$.
By the orbit-cone correspondence (see, e.g., \cite[Chapter~3]{CLS}), a cone $\sigma\in\Sigma_\theta(r)$ corresponds to a~$(3-r)$-dimensional
torus orbit in $X_{\Sigma_\theta}\cong\calM_\theta$, which we will denote by $\calO_\sigma$.
The intersection of cones in $\Sigma_\theta$ and the hyperplane at height one induces the triangulation of $\Delta_\Gamma$ into elementary triangles,
and hence we can identify
\begin{itemize}\itemsep=0pt
\item one-dimensional cones (= \emph{rays}) in $\Sigma_\theta$ with lattice points in the triangulation of~$\Delta_\Gamma$,
\item two-dimensional cones in $\Sigma_\theta$ with line segments in the triangulation of~$\Delta_\Gamma$,
\item three-dimensional cones in $\Sigma_\theta$ with triangles in the triangulation of~$\Delta_\Gamma$.
\end{itemize}
We denote the triangulation of $\Delta_\Gamma$ induced from $\Sigma_\theta$ by $\Delta_{\Gamma, \theta}$ (or $\Delta_\theta$ for simplicity).
It is known that a crepant resolution $X_{\Sigma_\theta}\to \Spec R$ is projective if and only if a triangulation $\Delta_{\Gamma, \theta}$ of $\Delta_\Gamma$ is \emph{regular} (or \emph{coherent}), see \cite[Proposition~2.4]{dais2001all}.
Since $\calM_\theta$ is a projective crepant resolution of $\Spec R$, the triangulation $\Delta_\theta$ is regular.

On the other hand, since $\calM_\theta$ parametrizes $\theta$-stable representations,
each point $y\in\calM_\theta$ corresponds to a $\theta$-stable representation, which we denote by $M_y$.
Under the isomorphism $\calM_\theta\cong X_{\Sigma_\theta}$,
we can assign $\theta$-stable representations to cones in $\Sigma_\theta$ (and hence to torus orbits).
As we will see in Proposition~\ref{corresp_pm}, a $\theta$-stable representation corresponding to a ray (and hence a~lattice point) in~$\Sigma_\theta$
can be obtained from a perfect matching.

\subsection{Perfect matchings corresponding to torus orbits}\label{subsec_cone_pm}

We then introduce another ingredient in dimer theory called perfect matchings.

\begin{Definition}\label{def_pm}
A \emph{perfect matching} (or \emph{dimer configuration}) of a dimer model $\Gamma$ is a subset $\sfP$ of $\Gamma_1$ such that for any node $n\in\Gamma_0$ there exists a unique edge in $\sfP$ containing $n$ as an endpoint.
We denote the set of perfect matchings of $\Gamma$ by $\PM(\Gamma)$.
\end{Definition}

Note that any dimer model does not necessarily have a perfect matching, but any consistent dimer model has a perfect matching
(see \cite[Proposition~8.1]{IU_special}).

\begin{Example}
\label{ex_PM_corner}
We consider the dimer model $\Gamma$ in Figure~\ref{ex_dimer1_basic}.
The following figures show some perfect matchings of $\Gamma$, where the edges contained in perfect matchings are colored red.

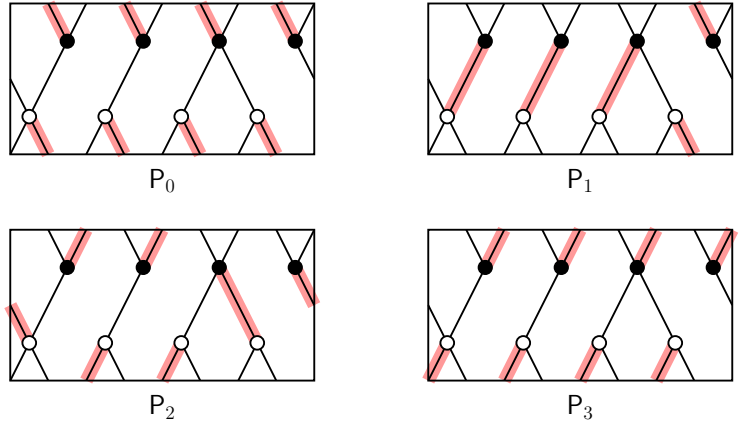
\begin{center}
\scalebox{0.5}{
\begin{tikzpicture}
\newcommand{\noderad}{0.18cm} 
\newcommand{\edgewidth}{0.05cm} 
\newcommand{\nodewidthw}{0.05cm} 
\newcommand{\nodewidthb}{0.04cm} 
\newcommand{\pmwidth}{0.35cm} 
\newcommand{\pmcolor}{red!40} 

\node at (0, -2.7) {\huge $\sfP_0$}; \node at (11, -2.7) {\huge $\sfP_1$};
\node at (0, -8.7) {\huge $\sfP_2$}; \node at (11, -8.7) {\huge $\sfP_3$};

\foreach \n/\a/\b in {1/1.5/3, 2/3.5/3, 3/5.5/3, 4/7.5/3} {\coordinate (B\n) at (\a,\b);}
\foreach \n/\a/\b in {1/0.5/1,2/2.5/1, 3/4.5/1, 4/6.5/1} {\coordinate (W\n) at (\a,\b);}

\node at (0,0){
\begin{tikzpicture}
\draw[line width=\pmwidth, \pmcolor] (W1)--(1,0); \draw[line width=\pmwidth, \pmcolor] (W2)--(3,0);
\draw[line width=\pmwidth, \pmcolor] (W3)--(5,0); \draw[line width=\pmwidth, \pmcolor] (W4)--(7,0);
\draw[line width=\pmwidth, \pmcolor] (1,4)--(B1); \draw[line width=\pmwidth, \pmcolor] (3,4)--(B2);
\draw[line width=\pmwidth, \pmcolor] (5,4)--(B3); \draw[line width=\pmwidth, \pmcolor] (7,4)--(B4);
\basicdimerA
\end{tikzpicture}};

\node at (11,0){
\begin{tikzpicture}
\foreach \w/\b in {1/1,2/2, 3/3} {\draw[line width=\pmwidth, \pmcolor] (W\w)--(B\b); };
\draw[line width=\pmwidth, \pmcolor] (W4)--(7,0); \draw[line width=\pmwidth, \pmcolor] (7,4)--(B4);
\basicdimerA
\end{tikzpicture}};

\node at (0,-6){
\begin{tikzpicture}
\draw[line width=\pmwidth, \pmcolor] (W1)--(0,2); \draw[line width=\pmwidth, \pmcolor] (2,4)--(B1);
\draw[line width=\pmwidth, \pmcolor] (W2)--(2,0); \draw[line width=\pmwidth, \pmcolor] (4,4)--(B2);
\draw[line width=\pmwidth, \pmcolor] (W3)--(4,0); \draw[line width=\pmwidth, \pmcolor] (W4)--(B3);
\draw[line width=\pmwidth, \pmcolor] (8,2)--(B4);
\basicdimerA
\end{tikzpicture}};

\node at (11,-6){
\begin{tikzpicture}
\draw[line width=\pmwidth, \pmcolor] (W1)--(0,0); \draw[line width=\pmwidth, \pmcolor] (W2)--(2,0);
\draw[line width=\pmwidth, \pmcolor] (W3)--(4,0); \draw[line width=\pmwidth, \pmcolor] (W4)--(6,0);
\draw[line width=\pmwidth, \pmcolor] (2,4)--(B1); \draw[line width=\pmwidth, \pmcolor] (4,4)--(B2);
\draw[line width=\pmwidth, \pmcolor] (6,4)--(B3); \draw[line width=\pmwidth, \pmcolor] (8,4)--(B4);
\basicdimerA
\end{tikzpicture}};

\end{tikzpicture}}
\end{center}
\end{Example}

Let $M=((M_v)_{v\in Q_0}, (\varphi_a)_{a\in Q_1})$ be a representation of $(Q,\calJ_Q)$.
We define the \emph{support} of $M$, denoted as $\Supp M$, as the set of arrows whose corresponding linear maps are not zero, that is,
\[
\Supp M\coloneqq\{a\in Q_1 \mid \varphi_a\neq0\}.
\]
We also define the \emph{cosupport} of $M$ as the complement of $\Supp M$.
For $\theta\in\Theta(Q)_\RR$, we say that a~perfect matching $\sfP$ is \emph{$\theta$-stable} if the set of arrows dual to edges contained in $\sfP$
is the cosupport of a $\theta$-stable representation.
Any perfect matching of $\Gamma$ can be obtained from a certain $\theta$-stable representation as follows.

\begin{Proposition}[{see \cite[Section~6]{IU_moduli}, \cite[Proposition~4.15]{Moz}}]
\label{corresp_pm}
Let $\Gamma$ be a consistent dimer model and $Q$ be the associated quiver.
\begin{itemize}\itemsep=0pt
\item[$(1)$] For a generic parameter $\theta\in\Theta(Q)_\RR$,
let $Z$ be a two-dimensional torus orbit of $\calM_\theta$, which is denoted by $Z=\calO_\rho$ for some ray $\rho\in\Sigma_\theta(1)$.
For any $y\in Z$, the cosupport of the $\theta$-stable representation~$M_y$ is the set of arrows dual to edges in
a certain perfect matching~$\sfP$ of~$\Gamma$.
This perfect matching $\sfP$ does not depend on a choice of $y\in Z=\calO_\rho$, thus we denote it by~$\sfP_\rho$.
\item[$(2)$]
For any perfect matching $\sfP$ of~$\Gamma$, there exists a generic parameter $\theta\in\Theta(Q)_\RR$ such that $\sfP$ is $\theta$-stable.
\end{itemize}
\end{Proposition}

By Proposition~\ref{corresp_pm}\,(1), for a generic parameter $\theta$,
we can assign a unique $\theta$-stable perfect matching to each lattice point of $\Delta_\Gamma$.
Thus, we have a bijection between lattice points of $\Delta_\Gamma$ and $\theta$-stable perfect matchings.
We say that a perfect matching $\sfP$ corresponds to a lattice point $q\in\Delta_\Gamma$
if for some generic parameter $\theta$ there exists a ray $\rho\in\Sigma_\theta(1)$ such that $\sfP=\sfP_\rho$ and $q=\rho\cap \Delta_\Gamma$.
We denote by $\PM_\theta(\Gamma)$ the set of $\theta$-stable perfect matchings.
By Proposition~\ref{corresp_pm}\,(2), we see that any perfect matching is contained in $\PM_\theta(\Gamma)$ for some generic parameter $\theta$.

\begin{Definition}
Let $\Delta_\Gamma$ be the zigzag polygon of a consistent dimer model~$\Gamma$.
We say that~$\sfP$ is
\begin{itemize}\itemsep=0pt
\item a \emph{corner} (or \emph{extremal}) \emph{perfect matching} if $\sfP$ corresponds to a vertex of~$\Delta_\Gamma$,
\item a \emph{boundary} (or \emph{external}) \emph{perfect matching} if $\sfP$ corresponds to a lattice point on the boundary of~$\Delta_\Gamma$,
and hence a corner perfect matching is a boundary perfect matching in particular.
\item an \emph{internal} \emph{perfect matching} if $\sfP$ corresponds to an interior lattice point of~$\Delta_\Gamma$.
\end{itemize}
\end{Definition}

We here note that corner perfect matchings have typical properties as follows.

\begin{Proposition}[{\cite[Corollary~4.27]{Bro_dimer}, \cite[Proposition~9.2]{IU_special}}]
\label{prop_cPM_unique}
Let $\Gamma$ be a consistent dimer model.
Then there is a unique corner perfect matching corresponding to each vertex of~$\Delta_\Gamma$,
and hence any corner perfect matching is $\theta$-stable for any generic parameter $\theta\in\Theta(Q)_\RR$.
Moreover, any corner perfect matching can be obtained from zigzag paths as in {\rm \cite[\emph{Section}~8]{IU_special}}.
\end{Proposition}

Thus, we can give a cyclic order to corner perfect matchings along the corresponding vertices of $\Delta_\Gamma$ in the anti-clockwise direction.
We say that two corner perfect matchings are \emph{adjacent} if they are adjacent with respect to the above cyclic order.

Next, we discuss the relationship between perfect matchings and zigzag paths.
We define the \emph{symmetric difference} $\sfP\ominus\sfP^\prime$ of perfect matchings $\sfP, \sfP^\prime\in\PM(\Gamma)$ as
$\sfP\ominus\sfP^\prime\coloneqq\sfP\cup\sfP^\prime{\setminus}\sfP\cap\sfP^\prime$.
Then, $\sfP\ominus\sfP^\prime$ can be considered as a $1$-cycle on $\TT$.
We fix the orientation of $\sfP\ominus\sfP^\prime$ so that an edge $e\in \sfP\ominus\sfP^\prime$ is directed
from a white (resp.\ black) node to a black (resp.\ white) node if $e\in\sfP$ (resp.\ $e\in\sfP^\prime$).

\begin{Proposition} [{see \cite[Corollary~3.8]{Gul}, \cite[Step~1 of the proof of Proposition~9.2, Corollary~9.3]{IU_special}}]\label{zigzag_sidepolygon}
Let $\Gamma$ be a consistent dimer model and $\Delta_\Gamma$ be the zigzag polygon.
Let~$E$ be a side of~$\Delta_\Gamma$.
Then, all zigzag paths whose slopes coincide with the outer normal vector of~$E$ arise as~$\sfP\ominus\sfP^\prime$
for the adjacent corner perfect matchings $\sfP,\sfP^\prime$ corresponding to the endpoints of~$E$.
\end{Proposition}

For example, the perfect matchings $\sfP_0$, $\sfP_1$, $\sfP_2$, $\sfP_3$ in Example~\ref{ex_PM_corner} are corner perfect matchings and
the zigzag paths shown in Figure~\ref{ex_dimer1_zigzag} can be obtained as the symmetric differences $\sfP_i\ominus\sfP_j$ for some $i, j=0, 1, 2, 3$,
see also Example~\ref{ex_boundaryPM}.

\begin{Observation}[{see \cite[Proposition~4.15 and the last part of Section~4]{Moz}}] \label{method_make_triangulation}
For a generic parameter $\theta\in\Theta(Q)_\RR$,
there is a certain method to detect a smooth toric fan $\Sigma_\theta$ such that ${X_{\Sigma_\theta}\cong\calM_\theta}$ when we know perfect matchings in~$\PM_\theta(\Gamma)$.
To do so, we detect a~triangulation~$\Delta_\theta$ of $\Delta_\Gamma$ which is identical to~$\Sigma_\theta$.
First, we assign each perfect matching in $\PM_\theta(\Gamma)$ to the corresponding lattice point of $\Delta_\Gamma$.
Then, for any pair of perfect matchings~$(\sfP,\sfP^\prime)$ in~$\PM_\theta(\Gamma)$, we check whether the set of arrows dual to $\sfP\cup \sfP^\prime$
is the cosupport of a $\theta$-stable representation or not.
If so, then we draw a line segment that connects lattice points corresponding to~$\sfP$ and~$\sfP^\prime$.
Repeating these arguments, we have a desired triangulation~$\Delta_\theta$.
\end{Observation}

For a generic parameter $\theta\in\Theta(Q)_\RR$ and an $r$-dimensional cone $\sigma\in\Sigma_\theta(r)$,
we have a $(3-r)$-dimensional torus orbit $\calO_\sigma$ in~$\calM_\theta$.
For $y\in\calO_\sigma$, we have the corresponding $\theta$-stable representation $M_y$ of $(Q, \calJ_Q)$.
Since the action of the open dense torus $G$ on $\calM_\theta$ is defined as in~\eqref{eq_torus_multi},
we see that the support of all $\theta$-stable representations corresponding to points in $\calO_\sigma$ are the same,
and hence we denote a representative of such $\theta$-stable representations by $M_\sigma$.
In particular, we have the following proposition by Observation~\ref{method_make_triangulation}.

\begin{Proposition}\label{prop_stablerep_cosupport}
Let $\theta\in\Theta(Q)_\RR$ be a generic parameter and $\sigma\in\Sigma_\theta(r)$ be an $r$-dimensional cone where $r=1,2,3$.
The cosupport of the $\theta$-stable representation $M_\sigma$
consists of the arrows dual to $\bigcup_{i=1}^r\sfP_i$, where $\sfP_1, \dots, \sfP_r$ are $\theta$-stable perfect matchings corresponding to the rays of $\sigma$.
\end{Proposition}

For a generic parameter $\theta\in\Theta(Q)_\RR$, the precise description of the tautological bundle \smash{$\calT_\theta=\bigoplus_{v\in Q_0}\calL_v$}
can be obtained by using perfect matchings in $\PM_\theta(\Gamma)$.
Here, we note how to compute~$\calT_\theta$ from $\Gamma$ following \cite[Section~2.5]{BCQV_recipe}.
Let $\overline{Q}$ be the \emph{double quiver} of $Q$, that is, $\overline{Q}$ can be obtained by adding an extra arrow $a^\ast\in (Q^{\rm op})_1$
in the opposite direction to $Q$ for any arrow $a\in Q_1$. We~call a path in~$\overline{Q}$ a \emph{weak path}.
For a perfect matching $\sfP$ of $\Gamma$, we define the degree function $\deg_\sfP$ on $Q_1$ associated to $\sfP$ as
\begin{equation*}
\deg_\sfP(a)=
\begin{cases}
1, & \text{the edge dual to $a\in Q_1$ is in $\sfP$},\\
0, & \text{otherwise},
\end{cases}
\end{equation*}
for any $a\in Q_1$.
We extend this degree function to the arrows in the double quiver $\overline{Q}$ as $\deg_\sfP(a^\ast)=-\deg_\sfP(a)$.
Then for a weak path $\gamma=a_1a_2\cdots a_r$ in $\overline{Q}$, we define
\[
\deg_\sfP(\gamma)=\sum_{i=1}^r\deg_\sfP(a_i).
\]

For $\rho\in\Sigma_\theta(1)$, there is a unique $\theta$-stable perfect matching corresponding to $\rho$ (see Proposition~\ref{corresp_pm}),
which we denote by $\sfP_\rho$.
Let $D_\rho$ be the torus-invariant prime divisor of $\calM_\theta$ corresponding to $\rho\in\Sigma_\theta(1)$.
For a weak path $\gamma$ in $\overline{Q}$, we define the divisor $D_\gamma$ as follows:
\[
D_\gamma=\sum_{\rho\in \Sigma_\theta(1)}(\deg_{\sfP_\rho} \gamma)D_\rho.
\]
Note that for weak paths $\gamma$, $\gamma^\prime$ such that $\hd(\gamma)=\hd(\gamma^\prime)$ and $\tl(\gamma)=\tl(\gamma^\prime)$,
we have $D_\gamma=D_{\gamma^\prime}$ in $\Pic\calM_\theta$.

\begin{Proposition}[{cf.\ \cite[Theorem~4.2]{BM_crepant}, \cite[Lemma~2.10]{BCQV_recipe}}]\label{prop_compute_tautological}
Consider the moduli space $\calM_\theta$ for a generic parameter $\theta\in\Theta(Q)_\RR$.
For the tautological bundle $\calT_\theta= \bigoplus_{v\in Q_0}\calL_v$ $($see Section~{\rm \ref{subsec_crepant_resolution}}$)$,
we see that $\calL_v \cong \calO_{\calM_\theta}(D_{\gamma_v})$ for any $v\in Q_0$, where~$\gamma_v$ is a weak path in $\overline{Q}$ from a~vertex~$0$ to a~vertex~$v$.
\end{Proposition}

\subsection{Wall-and-chamber structures}
\label{subsec_wall_chamber}

It is known that the space $\Theta(Q)_\RR$ of stability parameters has a \emph{wall-and-chamber structure}.
Namely, we define an equivalence relation on the set of generic parameters so that $\theta\sim \theta^\prime$
if and only if any $\theta$-stable representation of $(Q,\calJ_Q)$ is also $\theta^\prime$-stable and vice versa,
and this relation gives rise to the decomposition of stability parameters into finitely many chambers which are separated by walls (cf.\ \cite{DH_GIT, thadd_GIT}).
Here, a \emph{chamber} is an open cone in $\Theta(Q)_\RR$ consisting of equivalent generic parameters
and a \emph{wall} is a codimension one face of the closure of a chamber.
Note that any generic parameter lies in some chamber (see \cite[Lemma~6.1]{IU_anycrepant}).
The moduli space~$\calM_\theta$ is unchanged if a parameter $\theta$ moves in a chamber~$C$ of $\Theta(Q)_\RR$ by definition,
thus we sometimes use the notation~$\calM_C$ instead of $\calM_\theta$ for $\theta\in C$.

Let $C$, $C^\prime$ be adjacent chambers of $\Theta(Q)_\RR$ separated by a wall $W$, that is, $W=\overline{C}\cap \overline{C^\prime}$.
We choose generic parameters $\theta\in C$ and $\theta^\prime\in C^\prime$.
We also choose a stability parameter $\tilde{\theta}\in W$ such that $\tilde{\theta}$ does not lie on any other walls.
Note that $\tilde{\theta}$ is not generic since it is not contained in any chamber.
Let $X_{\tilde{\theta}}$ be the normalization of an irreducible component of $\overline{\calM}_{\tilde{\theta}}$ containing
the algebraic torus $G\subset \overline{\calM}_{\tilde{\theta}}$.
Then, there exists a projective morphism from $\calM_\theta$ to $\overline{\calM}_{\tilde{\theta}}$ factoring through $X_{\tilde{\theta}}$ :
\smash{$
\calM_\theta\xrightarrow{f}X_{\tilde{\theta}}\rightarrow\overline{\calM}_{\tilde{\theta}}$},
see \cite[Section~6]{IU_anycrepant}, \cite[Section~4.2]{BCQV_recipe}.
Similarly, we have a~projective morphism \smash{$\calM_{\theta^\prime}\xrightarrow{f^\prime}X_{\tilde{\theta}}\rightarrow\overline{\calM}_{\tilde{\theta}}$},
and we obtain a \emph{wall-crossing diagram}:
\begin{equation}\label{eq_wall_crossing_diagram}
\begin{aligned}&
\begin{tikzpicture}
\node(moduli1) at (0,0)
{\scalebox{1}{\begin{tikzpicture}$\calM_\theta$\end{tikzpicture}}};
\node(moduli2) at (3,0)
{\scalebox{1}{\begin{tikzpicture}$\calM_{\theta^\prime}$\end{tikzpicture}}};
\node(cont) at (1.5,-1.3)
{\scalebox{1}{\begin{tikzpicture}$X_{\tilde{\theta}}$\end{tikzpicture}}};
\path (moduli1) ++(-35:0.33cm) coordinate (moduli1+); \path (moduli2) ++(-35:0.33cm) coordinate (moduli2+);
\path (cont) ++(120:0.26cm) coordinate (cont-l); \path (cont) ++(25:0.55cm) coordinate (cont-r);
\draw[->] (moduli1+) to node[midway,xshift=-0.1cm,yshift=-0.25cm] {\footnotesize $f$} (cont-l) ;
\draw[->] (moduli2+) to node[midway,xshift=0.1cm,yshift=-0.25cm] {\footnotesize $f^\prime$} (cont-r) ;
\end{tikzpicture}
\end{aligned}
\end{equation}

The morphism $f$ is a \emph{primitive birational contraction} which can be classified into several types as in~\cite{Wil_walltypes}.
In our situation, by \cite[Section~11]{IU_anycrepant}, it is one of the following types:
\begin{itemize}\itemsep=0pt
\item Type $0$: $f\colon \calM_\theta\rightarrow X_{\tilde{\theta}}$ is an isomorphism.
\item Type I: $f\colon \calM_\theta\rightarrow X_{\tilde{\theta}}$ contracts a torus-invariant curve to a point.
\item Type $\typeIII$: $f\colon \calM_\theta\rightarrow X_{\tilde{\theta}}$ contracts a torus-invariant surface to a torus-invariant curve.
\end{itemize}
Note that a morphism contracting a surface to a point, which is called type \typeII, does not appear in our situation (see \cite[Lemma~10.5]{IU_anycrepant}).
The wall-crossing diagram \eqref{eq_wall_crossing_diagram} is the Atiyah flop if $f$ is of type~I (see \cite[Lemma~11.26]{IU_anycrepant}),
in which case the contracted curve corresponds to a diagonal of a parallelogram appearing in the triangulation $\Delta_\theta$ and
the Atiyah flop corresponds to the \emph{flip} of the diagonal.
If $f\colon \calM_\theta\rightarrow X_{\tilde{\theta}}$ is of type $\typeIII$, then it contracts a~toric divisor $D\subset\calM_\theta$ to a~torus-invariant curve $\ell_0\subset X_{\tilde{\theta}}$ and $f^\prime$ also contracts a toric divisor $D^\prime\subset\calM_{\theta^\prime}$ to $\ell_0$.
Moreover, we have an isomorphism $\calM_\theta\cong\calM_{\theta^\prime}$ (see \cite[Lemma~11.29]{IU_anycrepant}).
The case of type $0$ appears if the polygon $\Delta_\Gamma$ contains an interior lattice point (see \cite[Section~11.1]{IU_anycrepant}),
but in Sections~\ref{sec_type1and3_wallcrossing}--\ref{sec_D4} which are the main parts of this paper, we do not encounter such a situation.
Thus, we focus on the cases of types I and~$\typeIII$.

For each wall $W$, we have a primitive birational contraction $f$ (and a wall-crossing diagram) as above, thus we also classify walls in $\Theta(Q)_\RR$
according to the corresponding type of primitive birational contractions.
The precise description of a wall is determined by the degree of a~contracted curve on the tautological bundle $\calT_\theta=\bigoplus_{v\in Q_0}\calL_v$
by the argument in~\cite{IU_anycrepant} which was originally discussed in \cite{CrawIshii}.
More precisely, for a chamber~$C\subset\Theta(Q)_\RR$ and the tautological bundle~$\calT_\theta$ with $\theta\in C$, we have the linearisation map $L_C\colon \Theta(Q)_\RR\to \Pic(\calM_\theta)_\RR$ (see \cite[Section~3.2]{CrawIshii}, \cite[(4.1)]{IU_anycrepant}), and the line bundle $L_C(\theta)$ is ample for~$\theta\in C$. For a~curve~$\ell$ contained in~$\calM_\theta$, we can write
\begin{equation}\label{eq_linearisation}
\deg(L_C(\theta) |_\ell)=\sum_{v\in Q_0}\deg(\calL_v |_\ell)\theta_v
\end{equation}
(see \cite[equation~(5.3)]{CrawIshii}, \cite[equation~(9.6)]{IU_anycrepant}). As the parameter~$\theta$ varies within the chamber $C$, the degree in \eqref{eq_linearisation} is positive.
Moreover, if a wall contracting $\ell$ lies on the boundary of $C$, then this degree becomes~$0$ on the wall. In particular, we have the following proposition.

\begin{Proposition}[{\cite[Lemmas~11.21 and 11.30]{IU_anycrepant}}]\label{prop_equation_wall}
Let the notation be as above. We suppose that a wall $W$ is either
\begin{itemize}\itemsep=0pt
\item of type {\rm I} corresponding to $f\colon \calM_\theta\rightarrow X_{\tilde{\theta}}$ that contracts a torus-invariant curve $\ell\subset\calM_\theta$ to a point, or
\item of type $\typeIII$ corresponding to $f\colon \calM_\theta\rightarrow X_{\tilde{\theta}}$ that contracts a toric divisor $D\subset\calM_\theta$ to
a torus-invariant curve $\ell_0\subset X_{\tilde{\theta}}$,
and let $\ell$ be a torus-invariant curve in $D$ which is contracted to a point in $\ell_0$ via $f$.
\end{itemize}
Then, we have that
\begin{equation}
\sum_{v\in Q_0}\deg(\calL_v |_\ell)\tilde{\theta}_v=0 \qquad \text{for any $\tilde{\theta}=\bigl(\tilde{\theta}_v\bigr)_{v\in Q_0}\in W$.}
\label{eq_wall_equation}
\end{equation}
\end{Proposition}

\begin{Remark}
As we will see in Theorem~\ref{thm_main_wall} if $\calM_\theta$ is a projective crepant resolution of a~toric cDV singularity,
then any equation with the form~\eqref{eq_wall_equation} certainly
determines a wall of some chambers, which is a typical property for a toric cDV singularity.
For a certain three-dimensional Gorenstein toric singularity whose toric diagram contains an interior lattice point,
we encounter the situation that the equation~\eqref{eq_wall_equation} derived from a torus-invariant curve in $\calM_\theta$ with $\theta\in C$
does not determine a wall intersected with~$\overline{C}$, see \cite[Example~9.13]{CrawIshii} and \cite[Example~12.6]{IU_anycrepant}.
\end{Remark}

\section{Observations on boundary perfect matchings}\label{section_boundaryPM}

In the latter half of this paper, we mainly use boundary perfect matchings, thus we show some properties of boundary perfect matchings in this section.
Concerning properties of internal perfect matchings, see, e.g., \cite[Sections~3 and 5]{Nak_2rep}.

\subsection{Descriptions of boundary perfect matchings}\label{subsec_boundaryPM}

Let $\cPM$, $\cPM^\prime$ be adjacent corner perfect matchings of a consistent dimer model $\Gamma$.
By Proposition~\ref{zigzag_sidepolygon}, zigzag paths $z_1,\dots, z_r$ having the same slope (i.e., $[z_1]=\cdots=[z_r]$)
arise as $\cPM\ominus\cPM^\prime$, in which case we denote $\cPM\ominus\cPM^\prime=\{z_1,\dots, z_r\}$.
We suppose that $\cPM\cap z_i=\Zig(z_i)$ and $\cPM^\prime\cap z_i=\Zag(z_i)$ for any $i=1,\dots,r$.
The slope $[z_i]$ is the outer normal vector of the side of $\Delta_\Gamma$ whose endpoints are the vertices of $\Delta_\Gamma$
corresponding to $\cPM$, $\cPM^\prime$. We denote such a side by $E=E(\cPM, \cPM^\prime)$.
Then, we observe the description of boundary perfect matchings using the corner ones.

\begin{Proposition}[{e.g., \cite[Proposition~4.35]{Bro_dimer}, \cite[Corollary~3.8]{Gul}}]
\label{char_bound}
Let $\cPM$, $\cPM^\prime$, and $E$ be as above.
Let $q$ be a lattice point on~$E$ and~$m$ be the number of primitive side segments of $E$ between $q$ and the lattice point corresponding to~$\cPM$.
Then, any perfect matching of the following form corresponds to~$q$, and hence it is a boundary perfect matching:
\begin{align*}
\sfP_I\coloneqq\biggl(\cPM{\setminus}\bigcup_{i\in I}\Zig(z_i)\biggr)\cup\bigcup_{i\in I}\Zag(z_i)
&=\biggl(\cPM^\prime{\setminus}\bigcup_{i\in I^\rmc}\Zag(z_i)\biggr)\cup\bigcup_{i\in I^\rmc}\Zig(z_i)\\
&=\bigcup_{i\in I^\rmc}\Zig(z_i)\cup\bigcup_{i\in I}\Zag(z_i)\cup (\cPM\cap\cPM^\prime ),
\end{align*}
where $I$ is a subset of $[r]\coloneqq\{1, \dots, r\}$ with $m=|I|$ and $I^\rmc=[r]{\setminus}I$.
Note that $\sfP_\varnothing=\cPM$ and $\sfP_{[r]}=\cPM^\prime$.
Moreover, any boundary perfect matching takes this form.
Thus, the number of boundary perfect matchings corresponding to the lattice point $q$ is $\binom{r}{m}$ in particular.
\end{Proposition}

\begin{Example}\label{ex_boundaryPM}
We consider the perfect matchings $\sfP_0$ and $\sfP_1$ in Example~\ref{ex_PM_corner}.
These are corner perfect matchings and the symmetric difference $\sfP_0\ominus\sfP_1$ is equal to the zigzag paths shown in the following figure.

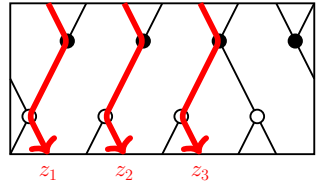
\begin{center}
\scalebox{0.5}{
\begin{tikzpicture}
\newcommand{\noderad}{0.18cm} 
\newcommand{\edgewidth}{0.05cm} 
\newcommand{\nodewidthw}{0.05cm} 
\newcommand{\nodewidthb}{0.04cm} 
\newcommand{\zigzagwidth}{0.15cm} 
\newcommand{\zigzagcolor}{red} 

\basicdimerA
\draw[->, line width=\zigzagwidth, rounded corners, color=\zigzagcolor] (1,4)--(B1)--(W1)--(1,0);
\draw[->, line width=\zigzagwidth, rounded corners, color=\zigzagcolor] (3,4)--(B2)--(W2)--(3,0);
\draw[->, line width=\zigzagwidth, rounded corners, color=\zigzagcolor] (5,4)--(B3)--(W3)--(5,0);

\node[red] at (1,-0.5) {\huge$z_1$}; \node[red] at (3,-0.5) {\huge$z_2$}; \node[red] at (5,-0.5) {\huge$z_3$};
\end{tikzpicture}}
\end{center}

The slopes of these zigzag paths are the outer normal vectors of the lower base of the zigzag polygon $\Delta(3,2)$ shown in Figure~\ref{ex_zigzagpolygon1}.
In particular, $\sfP_0$ (resp.\ $\sfP_1$) corresponds to the lower left (resp.\ right) vertex of $\Delta(3,2)$.
We fix the lower left vertex as the origin $(0, 0)$.

For the above zigzag paths $z_1$, $z_2$, $z_3$, we have $\Zig(z_i)=\sfP_0\cap z_i$ and $\Zag(z_i)=\sfP_1\cap z_i$.
Applying Proposition~\ref{char_bound} to subsets $\{1\}$ and $\{1, 3\}$, we have the perfect matchings as in Figure~\ref{fig_ex_PMfromzigzag},
which respectively corresponds to the lattice points $(1,0)$ and $(2,0)$ in $\Delta(3,2)$.

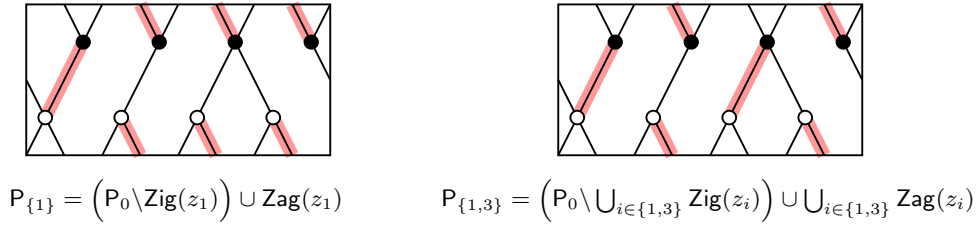
\begin{figure}[!ht]\centering
\begin{tikzpicture}

\node at (0,0) {\scalebox{0.5}{
\begin{tikzpicture}
\newcommand{\noderad}{0.18cm} 
\newcommand{\edgewidth}{0.05cm} 
\newcommand{\nodewidthw}{0.05cm} 
\newcommand{\nodewidthb}{0.04cm} 
\newcommand{\pmwidth}{0.35cm} 
\newcommand{\pmcolor}{red!40} 

\draw[line width=\pmwidth, \pmcolor] (W1)--(B1);
\draw[line width=\pmwidth, \pmcolor] (W2)--(3,0);
\draw[line width=\pmwidth, \pmcolor] (W3)--(5,0); \draw[line width=\pmwidth, \pmcolor] (W4)--(7,0);
\draw[line width=\pmwidth, \pmcolor] (3,4)--(B2);
\draw[line width=\pmwidth, \pmcolor] (5,4)--(B3); \draw[line width=\pmwidth, \pmcolor] (7,4)--(B4);
\basicdimerA
\end{tikzpicture}}};

\node at (0,-1.6){\small
$\sfP_{\{1\}}=\bigl(\sfP_0{\setminus}\Zig(z_1)\bigr)\cup\Zag(z_1)$};

\node at (7,0) {\scalebox{0.5}{
\begin{tikzpicture}
\newcommand{\noderad}{0.18cm} 
\newcommand{\edgewidth}{0.05cm} 
\newcommand{\nodewidthw}{0.05cm} 
\newcommand{\nodewidthb}{0.04cm} 
\newcommand{\pmwidth}{0.35cm} 
\newcommand{\pmcolor}{red!40} 

\draw[line width=\pmwidth, \pmcolor] (W1)--(B1); \draw[line width=\pmwidth, \pmcolor] (W3)--(B3);
\draw[line width=\pmwidth, \pmcolor] (W2)--(3,0); \draw[line width=\pmwidth, \pmcolor] (W4)--(7,0);
\draw[line width=\pmwidth, \pmcolor] (3,4)--(B2); \draw[line width=\pmwidth, \pmcolor] (7,4)--(B4);
\basicdimerA
\end{tikzpicture}}};

\node at (7,-1.6){\small
$\sfP_{\{1,3\}}=\bigl(\sfP_0{\setminus}\bigcup_{i\in \{1,3\}}\Zig(z_i)\bigr)\cup\bigcup_{i\in \{1,3\}}\Zag(z_i)$};
\end{tikzpicture}
\vspace{-2mm}

\caption{Examples of boundary perfect matchings determined by the zigzag paths $z_1$, $z_2$, $z_3$.}
\label{fig_ex_PMfromzigzag}
\end{figure}
\end{Example}

\subsection{Zigzag switchings}
\label{subsec_switchings}

In order to handle boundary perfect matchings, we introduce a new operation, which we will use in Section~\ref{sec_variation_representation}.

\begin{Definition}
\label{def_switching}
Suppose that $\sfP_I$ is a boundary perfect matching as in Proposition~\ref{char_bound} for a~subset~$I$ of~$[r]$.
Let $j\in [r]$, and hence $j\in I$ or $j\in I^\rmc$.
We define the \emph{zigzag switching} of $\sfP_I$ with respect to $z_j$ (or $j$), denoted by $\BS_{z_j}(\sfP_I)$, as follows:
\begin{gather*}
\BS_{z_j}(\sfP_I)=
\begin{cases}
\displaystyle\bigcup_{i\in I^\rmc\cup\{j\}}\Zig(z_i)\cup\bigcup_{i\in I{\setminus}\{j\}}\Zag(z_j)\cup\big(\cPM\cap\cPM^\prime\big)
\\
\qquad \text{if $j\in I$, equivalently $\Zag(z_i)\subset\sfP_I$}, \vspace{1mm} \\
\displaystyle\bigcup_{i\in I^\rmc{\setminus}\{j\}}\Zig(z_i)\cup\bigcup_{i\in I\cup\{j\}}\Zag(z_j)\cup\big(\cPM\cap\cPM^\prime\big)
\\
\qquad \text{if $j\in I^\rmc$, equivalently $\Zag(z_i)\subset\sfP_I$}.
\end{cases}
\end{gather*}
\end{Definition}

By definition and Proposition~\ref{char_bound}, $\BS_{z_j}(\sfP_I)$ remains a boundary perfect matching and satisfies the following properties.

\begin{Lemma}
Let $\sfP_I$ be a boundary perfect matching as in Proposition~{\rm \ref{char_bound}}.
For $j, k\in [r]$, we have $\BS_{z_j}\BS_{z_j}(\sfP_I)=\sfP_I$ and $\BS_{z_j}\BS_{z_k}(\sfP_I)=\BS_{z_k}\BS_{z_j}(\sfP_I)$.
\end{Lemma}

\begin{Lemma}
Let $\Gamma$ be a consistent dimer model, and $\sfP_I$ be a boundary perfect matching as in Proposition~{\rm \ref{char_bound}}.
Let $p$, $p^\prime$ be vertices of $\Delta_\Gamma$ corresponding to $\cPM$, $\cPM^\prime$, respectively.
Let $q$ be a~lattice point of $\Delta_\Gamma$ corresponding to $\sfP_I$, and $q^+$, $q^-$ be the lattice points on $E(\cPM,\cPM^\prime)$
next to~$q$.
We assume that $q^+$ $($resp.\ $q^-)$ is located between $p$ and $q$ $($resp.\ $p^\prime$ and $q)$.

\begin{center}
{\scalebox{0.6}{
\begin{tikzpicture}
\newcommand{\edgewidth}{0.08cm}
\newcommand{\nodewidth}{0.07cm}
\newcommand{\noderad}{0.1} 

\node at (-2,0) {\LARGE$E(\cPM,\cPM^\prime)\colon$};

\foreach \n/\a in {1/0,2/2,3/4,4/6,5/8,6/10,7/12} {
\coordinate (V\n) at (\a,0);
};

\foreach \s/\t in {1/2,3/4,4/5,6/7} {
\draw[line width=\edgewidth] (V\s)--(V\t); };
\draw[line width=\edgewidth, loosely dotted] (2.3,0)--(3.7,0);
\draw[line width=\edgewidth, loosely dotted] (8.3,0)--(9.7,0);

\foreach \n in {1,2,3,4,5,6,7} {
\draw [line width=\nodewidth, fill=black] (V\n) circle [radius=\noderad] ;
};
\newcommand{\nodeshift}{0.7cm}
\node[xshift=0cm,yshift=-\nodeshift] at (V1) {\LARGE$\cPM$};
\node[xshift=0cm,yshift=-\nodeshift] at (V4) {\LARGE$\sfP$};
\node[xshift=0cm,yshift=-\nodeshift] at (V7) {\LARGE$\cPM^\prime$};
\node[xshift=0cm,yshift=\nodeshift] at (V1) {\LARGE$p$};
\node[xshift=0cm,yshift=\nodeshift] at (V3) {\LARGE$q^-$};
\node[xshift=0cm,yshift=\nodeshift-0.1cm] at (V4) {\LARGE$q$};
\node[xshift=0cm,yshift=\nodeshift] at (V5) {\LARGE$q^+$};
\node[xshift=0cm,yshift=\nodeshift] at (V7) {\LARGE$p^\prime$};
\end{tikzpicture}
} }
\end{center}

Then we have the following:
\begin{itemize}\itemsep=0pt
\item[$(1)$] If $j\in I$, then the boundary perfect matching $\BS_{z_j}(\sfP_I)$ corresponds to $q^-$.
\item[$(2)$] If $j\in I^\rmc$, then the boundary perfect matching $\BS_{z_j}(\sfP_I)$ corresponds to $q^+$.
\item[$(3)$] If $j\in I$ and $k\in I^\rmc\cup\{j\}$, then the boundary perfect matching $\BS_{z_k}\BS_{z_j}(\sfP_I)$ corresponds to $q$.
\item[$(4)$] If $j\in I^\rmc$ and $k\in I\cup\{j\}$, then the boundary perfect matching $\BS_{z_k}\BS_{z_j}(\sfP_I)$ corresponds to $q$.
\end{itemize}
\end{Lemma}

\begin{Example}
Let the notation be as in Example~\ref{ex_boundaryPM}.
For the perfect matching
\[
\sfP_{\{1,3\}}=\biggl(\sfP_0{\setminus}\bigcup_{i\in \{1,3\}}\Zig(z_i)\biggr)\cup\bigcup_{i\in \{1,3\}}\Zag(z_i)
=\Zig(z_2)\cup\bigcup_{i\in \{1,3\}}\Zag(z_i)\cup (\sfP_0\cap\sfP_1 ),
\]
we see that
\begin{gather*}
\BS_{z_3}(\sfP_{\{1,3\}})=\sfP_{\{1\}}, \qquad \BS_{z_2}(\sfP_{\{1,3\}})=\sfP_{\{1,2,3\}}=\sfP_1, \\
\BS_{z_1}\BS_{z_3}(\sfP_{\{1,3\}})=\BS_{z_1}(\sfP_{\{1\}})=\sfP_{\varnothing}=\sfP_0.
\end{gather*}
\end{Example}

\subsection{Stable boundary perfect matchings}

As we saw in Section~\ref{subsec_cone_pm}, for a given generic parameter $\theta\in\Theta(Q)_\RR$,
we have the collection of $\theta$-stable perfect matchings $\PM_\theta(\Gamma)$ whose elements correspond bijectively to lattice points on $\Delta_\Gamma$.
In what follows, we will identify $\theta$-stable perfect matchings with corresponding lattice points on $\Delta_\Gamma$.

\begin{Setting}
\label{set_stable_boundaryPM}
Let $\Gamma$, $\cPM$, $\cPM^{\prime}$ be the same as Section~\ref{subsec_boundaryPM}.
For a generic parameter $\theta\in\Theta(Q)_\RR$, let $\sfP_1,\dots,\sfP_{r-1}\in\PM_\theta(\Gamma)$ be $\theta$-stable boundary perfect matchings
that correspond bijectively to the $r-1$ strict interior lattice points on $E=E(\cPM, \cPM^\prime)$.
In this setting, we can choose $\sfP_j$ so that the lattice length from $\cPM$ is $j$ (and hence the lattice length from $\cPM^\prime$ is $r-j$),
which means that $\sfP_j$ takes the form
\begin{equation}\label{eq_stable_boundaryPM}
\sfP_j=\bigcup_{i\in I^\rmc_{j,\theta}}\Zig(z_i)\cup\bigcup_{i\in I_{j,\theta}}\Zag(z_i)\cup\big(\cPM\cap\cPM^\prime\big)
\end{equation}
by Proposition~\ref{char_bound}, where $I_{j,\theta}$ is a subset of $[r]$ with $|I_{j,\theta}|=j$ and $I^\rmc_{j,\theta}=[r]{\setminus}I_{j,\theta}$.

\begin{center}
{\scalebox{0.6}{
\begin{tikzpicture}
\newcommand{\edgewidth}{0.08cm}
\newcommand{\nodewidth}{0.07cm}
\newcommand{\noderad}{0.1} 

\node at (-2,0) {\LARGE$E(\cPM,\cPM^\prime)\colon$};

\foreach \n/\a in {1/0,2/2,3/4,4/6,5/8,6/10,7/12} {
\coordinate (V\n) at (\a,0);
};

\foreach \s/\t in {1/2,3/4,4/5,6/7} {
\draw[line width=\edgewidth] (V\s)--(V\t); };
\draw[line width=\edgewidth, loosely dotted] (2.3,0)--(3.7,0);
\draw[line width=\edgewidth, loosely dotted] (8.3,0)--(9.7,0);

\foreach \n in {1,2,3,4,5,6,7} {
\draw [line width=\nodewidth, fill=black] (V\n) circle [radius=\noderad] ;
};
\newcommand{\nodeshift}{0.7cm}
\node[xshift=0cm,yshift=-\nodeshift] at (V1) {\LARGE$\cPM$};
\node[xshift=0cm,yshift=-\nodeshift] at (V2) {\LARGE$\sfP_1$};
\node[xshift=0cm,yshift=-\nodeshift] at (V3) {\LARGE$\sfP_{j-1}$};
\node[xshift=0cm,yshift=-\nodeshift] at (V4) {\LARGE$\sfP_j$};
\node[xshift=0cm,yshift=-\nodeshift] at (V5) {\LARGE$\sfP_{j+1}$};
\node[xshift=0cm,yshift=-\nodeshift] at (V6) {\LARGE$\sfP_{r-1}$};
\node[xshift=0cm,yshift=-\nodeshift] at (V7) {\LARGE$\cPM^\prime$};
\end{tikzpicture}
} }
\end{center}
Note that for any $j=1,\dots,r-1$ a subset $I_{j,\theta}$ is determined uniquely for a given $\theta$.
\end{Setting}

\begin{Lemma}
\label{lem_condition_stable_boundaryPM}
Let the notation be as in Setting~{\rm \ref{set_stable_boundaryPM}}.
For any $i=1, \dots, r$ and $j=1,\dots,r-2$, we see that if $\sfP_j\cap z_i=\Zag(z_i)$, then $\sfP_{j+1}\cap z_i=\Zag(z_i)$.
\end{Lemma}

\begin{proof}
We assume that $\sfP_j\cap z_i=\Zag(z_i)$ and $\sfP_{j+1}\cap z_i=\Zig(z_i)$.
Since $|I_{j,\theta}|+1=|I_{j+1,\theta}|$, there exists a zigzag path $z_k$ such that
$\sfP_j\cap z_k=\Zig(z_k)$ and $\sfP_{j+1}\cap z_k=\Zag(z_k)$, and hence $z_i\neq z_k$.
Since $\sfP_j$, $\sfP_{j+1}$ are $\theta$-stable, there is a $\theta$-stable representation $M=((M_v)_{v\in Q_0}, (\varphi_a)_{a\in Q_1})$
of dimension vector $\barone$
such that the cosupport of $M$ contains all arrows dual to edges composing the zigzag paths $z_i$ or $z_k$ (see Propositions~\ref{prop_stablerep_cosupport} and \ref{char_bound}).
Since the slopes of $z_i$ and $z_k$ are the same, these zigzag paths divide the two-torus $\TT$ into two parts.
Thus the quiver supporting $M$ is divided into two connected parts which we will denote by $Q_-$, $Q_+$.
We note that $\sum_{v\in (Q_-)_0}\theta_v+\sum_{v\in (Q_+)_0}\theta_v=0$, and may assume that
$\sum_{v\in (Q_-)_0}\theta_v=-\sum_{v\in (Q_+)_0}\theta_v<0$.
Then, a subrepresentation $N=((N_v)_{v\in Q_0}, (\varphi_a)_{a\in Q_1})$ of $M$ such that
\[
\begin{cases}
\dim N_v=1& \text{for any $v\in (Q_-)_0$}, \\
\dim N_v=0& \text{for any $v\in (Q_+)_0$}
\end{cases}
\]
satisfies $\theta(N)<0$, which is a contradiction.
\end{proof}

\begin{Proposition}
\label{prop_description_boundaryPM}
Let the notation be as in Setting~{\rm \ref{set_stable_boundaryPM}}.
For any $\theta$-stable non-corner boundary perfect matching $\sfP_j$ $(j=1, \dots, r-1)$,
there exists a unique sequence $(z_{i_1},\dots,z_{i_r})$ of zigzag paths such that
$\{i_1,\dots,i_r\}=\{1,\dots,r\}$ and
\begin{equation*}
\sfP_j=\bigcup_{k=j+1}^r\Zig(z_{i_k})\cup\bigcup_{k=1}^j\Zag(z_{i_k})\cup (\cPM\cap\cPM^\prime ).
\end{equation*}
\end{Proposition}

\begin{proof}
By \eqref{eq_stable_boundaryPM}, any boundary perfect matching $\sfP_j$ is determined by the set $I_{j,\theta}$.
Thus, we will detect zigzag paths whose intersections with $\sfP_j$ are their zags for identifying $\sfP_j$.

For the perfect matching $\sfP_1$, since $|I_{1,\theta}|=1$, we have a zigzag path $z_{i_1}\in\{z_1,\dots,z_r\}$ such that $\sfP_1\cap z_{i_1}=\Zag(z_{i_1})$,
that is, $I_{1,\theta}=\{i_1\}$.
Next, for the perfect matching $\sfP_2$,
since $|I_{2,\theta}|=2$ and $\sfP_2\cap z_{i_1}=\Zag(z_{i_1})$ by Lemma~\ref{lem_condition_stable_boundaryPM},
we have a zigzag path $z_{i_2}\in\{z_1,\dots,z_r\}{\setminus}\{z_{i_1}\}$ such that $\sfP_2\cap z_{i_2}=\Zag(z_{i_2})$,
thus $I_{2,\theta}=\{i_1,i_2\}$.
By repeating these arguments, we have the assertion.
\end{proof}

\section[Dimer models associated to R\_{a,b}]{Dimer models associated to $\boldsymbol{R_{a,b}}$}
\label{sec_dimer_ab}

Let $a$, $b$ be integers with $a\ge 1$ and $a\ge b\ge0$.
In what follows, we consider toric cDV singularities discussed in Section~\ref{subsec_cDVtoric}.
In particular, we focus on a toric $cA_{a+b-1}$ singularity:
\[
R_{a,b}\coloneqq \CC[x,y,z,w]/\bigl(xy-z^aw^b\bigr).
\]
Recall that the toric diagram of $R_{a,b}$ is the trapezoid $\Delta(a,b)$ shown in Figure~\ref{fig_hypersurf_lattice},
see also Example~\ref{ex_toric1}.

By Theorem~\ref{exist_dimer}, there exists a consistent dimer model $\Gamma$ whose zigzag polygon is the trapezoid $\Delta(a,b)$,
although it is not unique in general.
By the arguments in \cite[Section~1.2]{Nag}, such a~consistent dimer model takes the form of a tiling of the real two-torus $\TT=\RR^2/\ZZ^2$
by rhombi and hexagons.
We here recall the precise construction.
First, we place an infinite number of rhombi (resp.\ hexagons) in a line as shown in Figure~\ref{fig_basictile_SH},
and we denote the union of such rhombi (resp.\ hexagons) by $S$ (resp.\ $H$). We assume that all sides of rhombi and hexagons have the same length.

\begin{figure}[!ht]
\centering
\scalebox{0.7}{
\begin{tikzpicture}
\newcommand{\edgewidth}{0.035cm} 
\newcommand{\boundaryrad}{1cm}
\node at (0,0) {
\begin{tikzpicture}
\foreach \n/\a/\d in {1/0/cos(60), 2/90/cos(30), 3/180/cos(60), 4/270/cos(30)}
{\coordinate (S1\n) at (\a:{\d}); \path (S1\n) ++(0,{2*cos(30)}) coordinate (S2\n); \path (S2\n) ++(0,{2*cos(30)}) coordinate (S3\n);}
\foreach \m in {1,2,3} {\draw[line width=\edgewidth] (S\m1)--(S\m2)--(S\m3)--(S\m4)--(S\m1);}

\node at (-1.2,{2*cos(30)}) {\Large$S$:};
\draw[line width=\edgewidth, dotted] (0,-1.1)--(0,-1.5); \draw[line width=\edgewidth, dotted] (0,{4*cos(30)+1.1})--(0,{4*cos(30)+1.5});
\end{tikzpicture}};

\node at (6,0) {
\begin{tikzpicture}
\foreach \n/\a in {1/0, 2/60, 3/120, 4/180, 5/240, 6/300}
{\coordinate (H1\n) at (\a:\boundaryrad); \path (H1\n) ++(0,{2*cos(30)}) coordinate (H2\n); \path (H2\n) ++(0,{2*cos(30)}) coordinate (H3\n); }
\foreach \m in {1,2,3} {\draw[line width=\edgewidth] (H\m1)--(H\m2)--(H\m3)--(H\m4)--(H\m5)--(H\m6)--(H\m1);}

\node at (-1.7,{2*cos(30)}) {\Large$H$:};
\draw[line width=\edgewidth, dotted] (0,-1.1)--(0,-1.5); \draw[line width=\edgewidth, dotted] (0,{4*cos(30)+1.1})--(0,{4*cos(30)+1.5});
\end{tikzpicture}};
\end{tikzpicture}}
\vspace{-2mm}

\caption{Infinite number of rhombi and hexagons lied in a line.}
\label{fig_basictile_SH}
\end{figure}
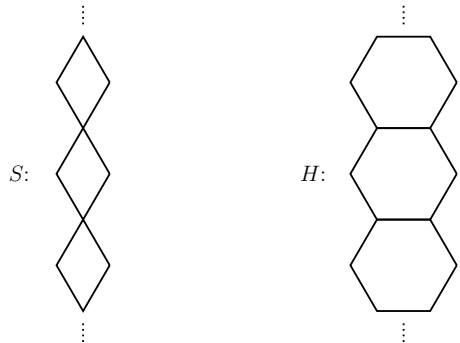

Let $n\coloneqq a+b$. We consider the tuple $(i_1, \dots, i_n)$ defined as
\begin{equation*}
i_k\coloneqq\begin{cases}
-1,& k=1, 2, \dots, a, \\
+1,& k=a+1, a+2, \dots, n.
\end{cases}
\end{equation*}
Then, for $\perm\in\fkS_n$ and the tuple $(i_1, \dots, i_n)$, we define the map $t_\perm\colon [n]=\{1,2,\dots,n\}\rightarrow \{S, H\}$ as
\begin{equation}\label{eq_map_SH}
t_\perm(k)=
\begin{cases}
S & \text{if $i_{\perm(k)}\neq i_{\perm(k+1)}$},\\
H & \text{if $i_{\perm(k)}= i_{\perm(k+1)}$}
\end{cases}
\end{equation}
for any $k=1,\dots, n$.
We extend the map $t_\perm$ to $\widetilde{t}_\perm\colon \ZZ\rightarrow \{S, H\}$ by setting
$\widetilde{t}_\perm(l)=t_\perm(k)$ for any $l\in \ZZ$ with $l \equiv k \pmod n$.
We label all rhombi in $S$ with $l \pmod n$ if $\widetilde{t}_\perm(l)=S$,
and label all hexagons in $H$ with $l \pmod n$ if $\widetilde{t}_\perm(l)=H$.
Then, we arrange $S$ and $H$ labeled with $l \pmod n$ along the cyclic order determined by $l \pmod n$ so that they tile the plane $\RR^2$.
By taking a~minimum-area parallelogram such that each vertex lies on the center of rhombi or hexagon labeled by~$0$,
we can cut out a fundamental domain of~$\TT$ from the tiling of~$\RR^2$.
This induces a~cell decomposition of~$\TT$ by rhombi and hexagons, which can be considered as a graph on~$\TT$.
We color the vertices of this graph with either black or white so that the resulting graph is bipartite, see Figure~\ref{fig_dimer_initial}.
Note that there are several choices of a fundamental domain of $\TT$ and there are two choices of a coloring of the graph, but in any case
the graph is a consistent dimer model and its zigzag polygon is unimodular equivalent to~$\Delta(a,b)$.
For simplicity, we always choose a~fundamental domain of $\TT$ (and a $\ZZ$-basis of~$\rmH_1(\TT)$)
such that the resulting dimer model, which we will denote by $\Gamma_\perm$, satisfies $\Delta_{\Gamma_\perm}=\Delta(a,b)$.
Also, in the following, we reuse the labels of faces of $\Gamma_\perm$ as the labels of vertices of the associated quiver $Q_\perm$.

\begin{Example}\label{ex_tiling_SH}
Let $a=3$, $b=2$. We consider the tuple $(i_1, i_2, i_3, i_4, i_5)=(-1,-1,-1,+1,+1)$ and the identity element $\id\in\fkS_5$.
Then we have
\[
\big(t_{\id}(1), t_{\id}(2), t_{\id}(3), t_{\id}(4), t_{\id}(5)\big)=(H,H,S,H,S).
\]
Then we consider the tiling of $\RR^2$ by labeled rhombi and hexagons determined by $\bigl(\widetilde{t}_{\id}(l)\bigr)_{l\in\ZZ}$,
and take a fundamental domain of $\TT$.
We color the vertices with either black or white, and obtain the dimer model $\Gamma_{\id}$ whose zigzag polygon is $\Delta(3,2)$ as shown in the right of Figure~\ref{fig_dimer_initial}.
This dimer model coincides with the dimer model given in Figure~\ref{ex_dimer1_basic} up to homotopy equivalence.

\begin{figure}[!ht]\centering
\scalebox{0.7}{
\begin{tikzpicture}
\newcommand{\edgewidth}{0.035cm} 
\newcommand{\boundaryrad}{1cm}
\newcommand{\noderad}{0.1cm} 
\newcommand{\nodewidthw}{0.035cm} 
\newcommand{\nodewidthb}{0.025cm} 

\node at (0,0) {
\begin{tikzpicture}
\node at (0,0) {
\begin{tikzpicture}
\foreach \n/\a/\d in {1/0/cos(60), 2/90/cos(30), 3/180/cos(60), 4/270/cos(30)}
{\coordinate (S1\n) at (\a:{\d}); \path (S1\n) ++(0,{2*cos(30)}) coordinate (S2\n); \path (S2\n) ++(0,{2*cos(30)}) coordinate (S3\n); \path (S3\n) ++(0,{2*cos(30)}) coordinate (S4\n);}
\foreach \m in {1,2,3,4} {\draw[line width=\edgewidth] (S\m1)--(S\m2)--(S\m3)--(S\m4)--(S\m1);}
\foreach \x/\y in {0/0, 0/{2*cos(30)}, 0/{4*cos(30)}, 0/{6*cos(30)}} {\node[blue] at (\x,{\y}) {\large$0$};}
\end{tikzpicture}};

\node at (1,0) {
\begin{tikzpicture}
\foreach \n/\a in {1/0, 2/60, 3/120, 4/180, 5/240, 6/300}
{\coordinate (H1\n) at (\a:\boundaryrad); \path (H1\n) ++(0,{2*cos(30)}) coordinate (H2\n); \path (H2\n) ++(0,{2*cos(30)}) coordinate (H3\n); }
\foreach \m in {1,2,3} {\draw[line width=\edgewidth] (H\m1)--(H\m2)--(H\m3)--(H\m4)--(H\m5)--(H\m6)--(H\m1);}
\foreach \x/\y in {0/0, 0/{2*cos(30)}, 0/{4*cos(30)}} {\node[blue] at (\x,{\y}) {\large$1$};}
\end{tikzpicture}};

\node at (2.5,0) {
\begin{tikzpicture}
\foreach \n/\a in {1/0, 2/60, 3/120, 4/180, 5/240, 6/300}
{\coordinate (H1\n) at (\a:\boundaryrad); \path (H1\n) ++(0,{2*cos(30)}) coordinate (H2\n); \path (H2\n) ++(0,{2*cos(30)}) coordinate (H3\n); \path (H3\n) ++(0,{2*cos(30)}) coordinate (H4\n);}
\foreach \m in {1,2,3,4} {\draw[line width=\edgewidth] (H\m1)--(H\m2)--(H\m3)--(H\m4)--(H\m5)--(H\m6)--(H\m1);}
\foreach \x/\y in {0/0, 0/{2*cos(30)}, 0/{4*cos(30)}, 0/{6*cos(30)}} {\node[blue] at (\x,{\y}) {\large$2$};}
\end{tikzpicture}};

\node at (3.5,0) {
\begin{tikzpicture}
\foreach \n/\a/\d in {1/0/cos(60), 2/90/cos(30), 3/180/cos(60), 4/270/cos(30)}
{\coordinate (S1\n) at (\a:{\d}); \path (S1\n) ++(0,{2*cos(30)}) coordinate (S2\n); \path (S2\n) ++(0,{2*cos(30)}) coordinate (S3\n);}
\foreach \m in {1,2,3} {\draw[line width=\edgewidth] (S\m1)--(S\m2)--(S\m3)--(S\m4)--(S\m1);}
\foreach \x/\y in {0/0, 0/{2*cos(30)}, 0/{4*cos(30)}} {\node[blue] at (\x,{\y}) {\large$3$};}
\end{tikzpicture}};

\node at (4.5,0) {
\begin{tikzpicture}
\foreach \n/\a in {1/0, 2/60, 3/120, 4/180, 5/240, 6/300}
{\coordinate (H1\n) at (\a:\boundaryrad); \path (H1\n) ++(0,{2*cos(30)}) coordinate (H2\n); \path (H2\n) ++(0,{2*cos(30)}) coordinate (H3\n); \path (H3\n) ++(0,{2*cos(30)}) coordinate (H4\n);}
\foreach \m in {1,2,3,4} {\draw[line width=\edgewidth] (H\m1)--(H\m2)--(H\m3)--(H\m4)--(H\m5)--(H\m6)--(H\m1);}
\foreach \x/\y in {0/0, 0/{2*cos(30)}, 0/{4*cos(30)}, 0/{6*cos(30)}} {\node[blue] at (\x,{\y}) {\large$4$};}
\end{tikzpicture}};

\node at (5.5,0) {
\begin{tikzpicture}
\foreach \n/\a/\d in {1/0/cos(60), 2/90/cos(30), 3/180/cos(60), 4/270/cos(30)}
{\coordinate (S1\n) at (\a:{\d}); \path (S1\n) ++(0,{2*cos(30)}) coordinate (S2\n); \path (S2\n) ++(0,{2*cos(30)}) coordinate (S3\n);}
\foreach \m in {1,2,3} {\draw[line width=\edgewidth] (S\m1)--(S\m2)--(S\m3)--(S\m4)--(S\m1);}
\foreach \x/\y in {0/0, 0/{2*cos(30)}, 0/{4*cos(30)}} {\node[blue] at (\x,{\y}) {\large$0$};}
\end{tikzpicture}};
\end{tikzpicture}};

\node at (10,0){
\begin{tikzpicture}
\draw[gray, line width=\edgewidth, dashed] (0,{-cos(30)})--(0,{cos(30)})--(5.5,0)--(5.5,{-2*cos(30)})--(0,{-cos(30)});

\node at (0,0) {
\begin{tikzpicture}
\foreach \n/\a/\d in {1/0/cos(60), 2/90/cos(30), 3/180/cos(60), 4/270/cos(30)}
{\coordinate (S1\n) at (\a:{\d}); \path (S1\n) ++(0,{2*cos(30)}) coordinate (S2\n); \path (S2\n) ++(0,{2*cos(30)}) coordinate (S3\n); \path (S3\n) ++(0,{2*cos(30)}) coordinate (S4\n);}
\foreach \m in {1,2,3,4} {\draw[line width=\edgewidth] (S\m1)--(S\m2)--(S\m3)--(S\m4)--(S\m1);}
\foreach \m/\n in {1/1,1/3, 2/1, 2/3, 3/1,3/3, 4/1, 4/3} {\filldraw [fill=black, line width=\nodewidthb] (S\m\n) circle [radius=\noderad] ;};
\foreach \m/\n in {1/2,1/4, 2/2, 2/4, 3/2,3/4, 4/2, 4/4} {\filldraw [fill=white, line width=\nodewidthw] (S\m\n) circle [radius=\noderad] ;};
\foreach \x/\y in {0/0, 0/{2*cos(30)}, 0/{4*cos(30)}, 0/{6*cos(30)}} {\node[blue] at (\x,{\y}) {\large$0$};}
\end{tikzpicture}};

\node at (1,0) {
\begin{tikzpicture}
\foreach \n/\a in {1/0, 2/60, 3/120, 4/180, 5/240, 6/300}
{\coordinate (H1\n) at (\a:\boundaryrad); \path (H1\n) ++(0,{2*cos(30)}) coordinate (H2\n); \path (H2\n) ++(0,{2*cos(30)}) coordinate (H3\n); }
\foreach \m in {1,2,3} {\draw[line width=\edgewidth] (H\m1)--(H\m2)--(H\m3)--(H\m4)--(H\m5)--(H\m6)--(H\m1);}
\foreach \m/\n in {1/1,1/3, 1/5, 2/1, 2/3, 2/5, 3/1, 3/3, 3/5} {\filldraw [fill=black, line width=\nodewidthb] (H\m\n) circle [radius=\noderad] ;};
\foreach \m/\n in {1/2,1/4, 1/6, 2/2, 2/4, 2/6, 3/2, 3/4, 3/6} {\filldraw [fill=white, line width=\nodewidthw] (H\m\n) circle [radius=\noderad] ;};
\foreach \x/\y in {0/0, 0/{2*cos(30)}, 0/{4*cos(30)}} {\node[blue] at (\x,{\y}) {\large$1$};}
\end{tikzpicture}};

\node at (2.5,0) {
\begin{tikzpicture}
\foreach \n/\a in {1/0, 2/60, 3/120, 4/180, 5/240, 6/300}
{\coordinate (H1\n) at (\a:\boundaryrad); \path (H1\n) ++(0,{2*cos(30)}) coordinate (H2\n); \path (H2\n) ++(0,{2*cos(30)}) coordinate (H3\n); \path (H3\n) ++(0,{2*cos(30)}) coordinate (H4\n);}
\foreach \m in {1,2,3,4} {\draw[line width=\edgewidth] (H\m1)--(H\m2)--(H\m3)--(H\m4)--(H\m5)--(H\m6)--(H\m1);}
\foreach \m/\n in {1/1,1/3, 1/5, 2/1, 2/3, 2/5, 3/1, 3/3, 3/5, 4/1, 4/3, 4/5} {\filldraw [fill=black, line width=\nodewidthb] (H\m\n) circle [radius=\noderad] ;};
\foreach \m/\n in {1/2,1/4, 1/6, 2/2, 2/4, 2/6, 3/2, 3/4, 3/6, 4/2, 4/4, 4/6} {\filldraw [fill=white, line width=\nodewidthw] (H\m\n) circle [radius=\noderad] ;};
\foreach \x/\y in {0/0, 0/{2*cos(30)}, 0/{4*cos(30)}, 0/{6*cos(30)}} {\node[blue] at (\x,{\y}) {\large$2$};}
\end{tikzpicture}};

\node at (3.5,0) {
\begin{tikzpicture}
\foreach \n/\a/\d in {1/0/cos(60), 2/90/cos(30), 3/180/cos(60), 4/270/cos(30)}
{\coordinate (S1\n) at (\a:{\d}); \path (S1\n) ++(0,{2*cos(30)}) coordinate (S2\n); \path (S2\n) ++(0,{2*cos(30)}) coordinate (S3\n);}
\foreach \m in {1,2,3} {\draw[line width=\edgewidth] (S\m1)--(S\m2)--(S\m3)--(S\m4)--(S\m1);}
\foreach \m/\n in {1/2,1/4, 2/2, 2/4, 3/2,3/4} {\filldraw [fill=black, line width=\nodewidthb] (S\m\n) circle [radius=\noderad] ;};
\foreach \m/\n in {1/1,1/3, 2/1, 2/3, 3/1,3/3} {\filldraw [fill=white, line width=\nodewidthw] (S\m\n) circle [radius=\noderad] ;};
\foreach \x/\y in {0/0, 0/{2*cos(30)}, 0/{4*cos(30)}} {\node[blue] at (\x,{\y}) {\large$3$};}
\end{tikzpicture}};

\node at (4.5,0) {
\begin{tikzpicture}
\foreach \n/\a in {1/0, 2/60, 3/120, 4/180, 5/240, 6/300}
{\coordinate (H1\n) at (\a:\boundaryrad); \path (H1\n) ++(0,{2*cos(30)}) coordinate (H2\n); \path (H2\n) ++(0,{2*cos(30)}) coordinate (H3\n); \path (H3\n) ++(0,{2*cos(30)}) coordinate (H4\n);}
\foreach \m in {1,2,3,4} {\draw[line width=\edgewidth] (H\m1)--(H\m2)--(H\m3)--(H\m4)--(H\m5)--(H\m6)--(H\m1);}
\foreach \m/\n in {1/2,1/4, 1/6, 2/2, 2/4, 2/6, 3/2, 3/4, 3/6, 4/2, 4/4, 4/6} {\filldraw [fill=black, line width=\nodewidthb] (H\m\n) circle [radius=\noderad] ;};
\foreach \m/\n in {1/1,1/3, 1/5, 2/1, 2/3, 2/5, 3/1, 3/3, 3/5, 4/1, 4/3, 4/5} {\filldraw [fill=white, line width=\nodewidthw] (H\m\n) circle [radius=\noderad] ;};
\foreach \x/\y in {0/0, 0/{2*cos(30)}, 0/{4*cos(30)}, 0/{6*cos(30)}} {\node[blue] at (\x,{\y}) {\large$4$};}
\end{tikzpicture}};

\node at (5.5,0) {
\begin{tikzpicture}
\foreach \n/\a/\d in {1/0/cos(60), 2/90/cos(30), 3/180/cos(60), 4/270/cos(30)}
{\coordinate (S1\n) at (\a:{\d}); \path (S1\n) ++(0,{2*cos(30)}) coordinate (S2\n); \path (S2\n) ++(0,{2*cos(30)}) coordinate (S3\n);}
\foreach \m in {1,2,3} {\draw[line width=\edgewidth] (S\m1)--(S\m2)--(S\m3)--(S\m4)--(S\m1);}
\foreach \m/\n in {1/1,1/3, 2/1, 2/3, 3/1,3/3} {\filldraw [fill=black, line width=\nodewidthb] (S\m\n) circle [radius=\noderad] ;};
\foreach \m/\n in {1/2,1/4, 2/2, 2/4, 3/2,3/4} {\filldraw [fill=white, line width=\nodewidthw] (S\m\n) circle [radius=\noderad] ;};
\foreach \x/\y in {0/0, 0/{2*cos(30)}, 0/{4*cos(30)}} {\node[blue] at (\x,{\y}) {\large$0$};}
\end{tikzpicture}};
\end{tikzpicture}};

\end{tikzpicture}}
\vspace{-2mm}

\caption{}\label{fig_dimer_initial}
\end{figure}
\end{Example}

We now observe zigzag paths of the dimer model $\Gamma_\perm$.
Zigzag paths appearing in the parts originating from $S$ and $H$ in the universal cover $\widetilde{\Gamma}_\perm$ take the form shown in Figure~\ref{fig_part_zigzag}.
Note that two zigzag paths appearing in each figure share no edges.
In the universal cover $\widetilde{\Gamma}_\perm$, rhombi and hexagons are glued together along these zigzag paths.
Since the number of faces of $\Gamma_\perm$ on the two-torus $\TT$ is $n=a+b$,
if we choose a fundamental domain such that the zigzag paths in Figure~\ref{fig_part_zigzag} restricted to $\Gamma_\perm$ have slopes $(0,1)$ or $(0,-1)$,
we obtain $\Delta_{\Gamma_\perm}=\Delta(a,b)$. In particular, these zigzag paths satisfy the following properties.

\begin{Lemma}\label{lem_zigzag_nointersection}
Let $\Gamma_\perm$ be the consistent dimer model for some $\perm\in\fkS_n$.
Let $z_1, \dots, z_n$ be zigzag paths on $\Gamma_\perm$ such that for any $i=1, \dots, n$ the slope $[z_i]$ is either $(0,-1)$ or $(0,1)$,
in which these zigzag paths take the form shown in Figure~{\rm \ref{fig_part_zigzag}}.
Then we have the following:
\begin{itemize}\itemsep=0pt
\item[$(1)$] Any pair of zigzag paths in $\{z_1, \dots, z_n\}$ do not intersect each other, that is, they do not share any edge.
\item[$(2)$] For any $i=1, \dots, n$, the zigzag path $z_i$ consists of the edges shared by faces labeled $k-1$ and $k$$\pmod n$ for some $k=1, \dots, n$.
\item[$(3)$] For any $i=1, \dots, n$, the zigzag path $z_i$ consists of two edges, and hence both $\Zig(z_i)$ and $\Zag(z_i)$ consist of a single edge.
\item[$(4)$] An edge dual to a loop arrow in $Q_\perm$, which appears on a joint of hexagons in $H$, is not contained in any zigzag path $z_i$.
\end{itemize}
\end{Lemma}

\begin{proof}
(1) and (2) follow from the explanation preceding the lemma.
Furthermore, by definition, a rhombus or a hexagon with the same label is unique in $\Gamma_\perm$, and hence we have~(3).
Regarding~(4), since each zigzag path~$z_i$ separates faces with different labels, an arrow appearing as the dual of an edge on $z_i$ can not be a loop.
\end{proof}

\vspace{-8mm}

\begin{figure}[!ht]
\centering
\scalebox{0.8}{
\begin{tikzpicture}
\newcommand{\edgewidth}{0.035cm} 
\newcommand{\boundaryrad}{1cm}
\newcommand{\noderad}{0.1cm} 
\newcommand{\nodewidthw}{0.035cm} 
\newcommand{\nodewidthb}{0.025cm} 
\newcommand{\zigzagwidth}{0.08cm} 
\newcommand{\zigzagcolor}{red} 
\newcommand{\zigzagcolorII}{blue} 

\node at (0,0) {
\begin{tikzpicture}
\foreach \n/\a/\d in {1/0/cos(60), 2/90/cos(30), 3/180/cos(60), 4/270/cos(30)}
{\coordinate (S1\n) at (\a:{\d}); \path (S1\n) ++(0,{2*cos(30)}) coordinate (S2\n); \path (S2\n) ++(0,{2*cos(30)}) coordinate (S3\n);}
\foreach \m in {1,2} {\draw[line width=\edgewidth] (S\m1)--(S\m2)--(S\m3)--(S\m4)--(S\m1);}
\foreach \m/\n in {1/2,1/4, 2/2, 2/4} {\filldraw [fill=white, line width=\nodewidthw] (S\m\n) circle [radius=\noderad] ;};
\foreach \m/\n in {1/1,1/3, 2/1, 2/3} {\filldraw [fill=black, line width=\nodewidthb] (S\m\n) circle [radius=\noderad] ;};

\begin{scope}[transform canvas={xshift=-0.2cm}]
\draw[->, line width=\zigzagwidth, rounded corners, color=\zigzagcolorII] (S14)--(S13)--(S24)--(S23)--(S34);
\end{scope}
\begin{scope}[transform canvas={xshift=0.2cm}]
\draw[<-, line width=\zigzagwidth, rounded corners, color=\zigzagcolor] (S14)--(S11)--(S24)--(S21)--(S34);
\end{scope}

\draw[line width=\edgewidth, dotted] (0,-1.1)--(0,-1.5); \draw[line width=\edgewidth, dotted] (0,{2*cos(30)+1.1})--(0,{2*cos(30)+1.5});
\end{tikzpicture}};

\node at (4,0) {
\begin{tikzpicture}
\foreach \n/\a/\d in {1/0/cos(60), 2/90/cos(30), 3/180/cos(60), 4/270/cos(30)}
{\coordinate (S1\n) at (\a:{\d}); \path (S1\n) ++(0,{2*cos(30)}) coordinate (S2\n); \path (S2\n) ++(0,{2*cos(30)}) coordinate (S3\n);}
\foreach \m in {1,2} {\draw[line width=\edgewidth] (S\m1)--(S\m2)--(S\m3)--(S\m4)--(S\m1);}
\foreach \m/\n in {1/2,1/4, 2/2, 2/4} {\filldraw [fill=black, line width=\nodewidthb] (S\m\n) circle [radius=\noderad] ;};
\foreach \m/\n in {1/1,1/3, 2/1, 2/3} {\filldraw [fill=white, line width=\nodewidthw] (S\m\n) circle [radius=\noderad] ;};

\begin{scope}[transform canvas={xshift=-0.2cm}]
\draw[<-, line width=\zigzagwidth, rounded corners, color=\zigzagcolor] (S14)--(S13)--(S24)--(S23)--(S34);
\end{scope}
\begin{scope}[transform canvas={xshift=0.2cm}]
\draw[->, line width=\zigzagwidth, rounded corners, color=\zigzagcolorII] (S14)--(S11)--(S24)--(S21)--(S34);
\end{scope}

\draw[line width=\edgewidth, dotted] (0,-1.1)--(0,-1.5); \draw[line width=\edgewidth, dotted] (0,{2*cos(30)+1.1})--(0,{2*cos(30)+1.5});
\end{tikzpicture}};

\node at (8,0) {
\begin{tikzpicture}
\foreach \n/\a in {1/0, 2/60, 3/120, 4/180, 5/240, 6/300}
{\coordinate (H1\n) at (\a:\boundaryrad); \path (H1\n) ++(0,{2*cos(30)}) coordinate (H2\n); \path (H2\n) ++(0,{2*cos(30)}) coordinate (H3\n); }
\foreach \m in {1,2} {\draw[line width=\edgewidth] (H\m1)--(H\m2)--(H\m3)--(H\m4)--(H\m5)--(H\m6)--(H\m1);}
\foreach \m/\n in {1/1,1/3, 1/5, 2/1, 2/3, 2/5} {\filldraw [fill=white, line width=\nodewidthw] (H\m\n) circle [radius=\noderad] ;};
\foreach \m/\n in {1/2,1/4, 1/6, 2/2, 2/4, 2/6} {\filldraw [fill=black, line width=\nodewidthb] (H\m\n) circle [radius=\noderad] ;};

\begin{scope}[transform canvas={xshift=-0.2cm}]
\draw[->, line width=\zigzagwidth, rounded corners, color=\zigzagcolorII] (H15)--(H14)--(H25)--(H24)--(H35);
\end{scope}
\begin{scope}[transform canvas={xshift=0.2cm}]
\draw[->, line width=\zigzagwidth, rounded corners, color=\zigzagcolorII] (H16)--(H11)--(H26)--(H21)--(H36);
\end{scope}

\draw[line width=\edgewidth, dotted] (0,-1.1)--(0,-1.5); \draw[line width=\edgewidth, dotted] (0,{2*cos(30)+1.1})--(0,{2*cos(30)+1.5});
\end{tikzpicture}};

\node at (12,0) {
\begin{tikzpicture}
\foreach \n/\a in {1/0, 2/60, 3/120, 4/180, 5/240, 6/300}
{\coordinate (H1\n) at (\a:\boundaryrad); \path (H1\n) ++(0,{2*cos(30)}) coordinate (H2\n); \path (H2\n) ++(0,{2*cos(30)}) coordinate (H3\n); }
\foreach \m in {1,2} {\draw[line width=\edgewidth] (H\m1)--(H\m2)--(H\m3)--(H\m4)--(H\m5)--(H\m6)--(H\m1);}
\foreach \m/\n in {1/1,1/3, 1/5, 2/1, 2/3, 2/5} {\filldraw [fill=black, line width=\nodewidthb] (H\m\n) circle [radius=\noderad] ;};
\foreach \m/\n in {1/2,1/4, 1/6, 2/2, 2/4, 2/6} {\filldraw [fill=white, line width=\nodewidthw] (H\m\n) circle [radius=\noderad] ;};

\begin{scope}[transform canvas={xshift=-0.2cm}]
\draw[<-, line width=\zigzagwidth, rounded corners, color=\zigzagcolor] (H15)--(H14)--(H25)--(H24)--(H35);
\end{scope}
\begin{scope}[transform canvas={xshift=0.2cm}]
\draw[<-, line width=\zigzagwidth, rounded corners, color=\zigzagcolor] (H16)--(H11)--(H26)--(H21)--(H36);
\end{scope}

\draw[line width=\edgewidth, dotted] (0,-1.1)--(0,-1.5); \draw[line width=\edgewidth, dotted] (0,{2*cos(30)+1.1})--(0,{2*cos(30)+1.5});
\end{tikzpicture}};
\end{tikzpicture}}
\vspace{-2mm}

\caption{}
\label{fig_part_zigzag}
\end{figure}

By the above construction, we can obtain a consistent dimer model $\Gamma_\perm$ for any $\perm\in\fkS_n$.
For the adjacent transposition $s_k$ $(k=1, \dots, n-1)$, we have the consistent dimer model $\Gamma_{\perm s_k}$.
Then we see that the dimer models $\Gamma_\perm$ and $\Gamma_{\perm s_k}$ are transformed into each other by ``mutations”.
First, as we mentioned in Section~\ref{sec_quiver_rep}, we have the quiver with potential associated to a dimer model.
In many cases, a \emph{mutation of a dimer model}, which produces a new dimer model from a given one,
can be defined as the dual of the mutation of a quiver with potential in the sense of \cite{DWZ1} at a vertex corresponding to a quadrilateral face (see \cite[Section~7.2]{Boc_toricNCCR}, \cite[Section~4]{Nak_reflexive}, \cite[Section~6.2]{Nak_2rep}).
However, this mutation can be applied only to a vertex not lying on $2$-cycles and not having loops.
Since any vertex of the quiver associated to the dimer model $\Gamma_\perm$ lies on $2$-cycles and may have a loop,
we can not apply this mutation to our dimer model.
On the other hand, there is a certain way to relate the associated Jacobian algebras $A_\perm$ and $A_{\perm s_k}$
as shown in \cite[Section~3.1]{Nag}, and it can be considered as the \emph{mutation of tilting modules} in the sense of \cite[Section~5]{IR_FZ}.
Since any two elements in $\fkS_n$ can be transformed into each other by the action of adjacent transpositions,
the associated dimer models can also be related by the mutations of the associated Jacobian algebras.

\section{Wall-and-chamber structures and zigzag paths}
\label{sec_type1and3_wallcrossing}

\begin{Setting}
\label{setting_zigzag_for_a+b}
For $\perm\in\fkS_n$, let $\Gamma_\perm$ be the consistent dimer model whose zigzag polygon is the trapezoid $\Delta(a,b)$
and $Q_\perm$ be the associated quiver as in Section~\ref{sec_dimer_ab}.
Let $z_1,\dots, z_a$, $w_1,\dots, w_b$ be zigzag paths satisfying $[z_1]=\cdots=[z_a]=(0,-1)$ and $[w_1]=\cdots=[w_b]=(0,1)$.
Note that these zigzag paths satisfy the properties as in Lemma~\ref{lem_zigzag_nointersection}.
We fix the lower left vertex of $\Delta(a,b)$ as the origin.

For the space $\Theta(Q_\perm)_\RR$ of stability parameters, any $\theta\in \Theta(Q_\perm)_\RR$ satisfies $\theta_0=-\sum_{v\neq 0}\theta_v$.
Thus, in what follows, when we consider $\Theta(Q_\perm)_\RR$, we employ the coordinates $\theta_v$ with $v\neq 0$.
For a generic parameter $\theta\in C$ in a chamber $C\subset\Theta(Q_\perm)_\RR$,
let $\sfP_{(i,j)}^\theta$ be the $\theta$-stable boundary perfect matching corresponding to the lattice point $(i,j)$ on $\Delta(a,b)$.
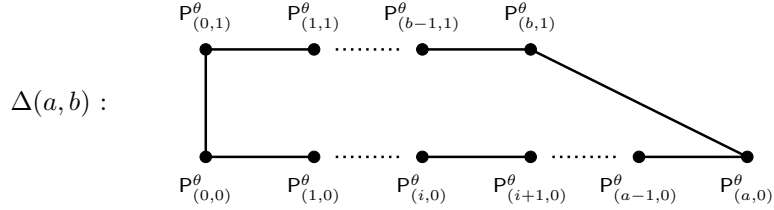
\begin{center}
\begin{tikzpicture}
\newcommand{\edgewidth}{0.035cm}
\newcommand{\noderad}{0.08} 

\node at (-5,0) {\small $\Delta(a,b)$:};

\node at (0,0)
{\scalebox{0.9}{
\begin{tikzpicture}
\foreach \n/\a/\b in {00/0/0,10/1.5/0,20/3/0, 30/4.5/0, 40/6/0, 50/7.5/0, 01/0/1.5, 11/1.5/1.5, 21/3/1.5, 31/4.5/1.5} {
\coordinate (V\n) at (\a,\b);
};

\foreach \s/\t in {00/10,00/01,01/11,20/30,21/31,40/50,50/31} {
\draw [line width=\edgewidth] (V\s)--(V\t) ;};
\draw [line width=\edgewidth, dotted] (1.8,0)-- ++(0.9,0);
\draw [line width=\edgewidth, dotted] (4.8,0)-- ++(0.9,0);
\draw [line width=\edgewidth, dotted] (1.8,1.5)-- ++(0.9,0);

\foreach \n in {00,10,20,30,40,50,01,11,21,31}{
\draw [fill=black] (V\n) circle [radius=\noderad] ; };

\newcommand{\nodeshift}{0.45cm}
\node[xshift=0cm,yshift=-\nodeshift] at (V00) {\footnotesize$\sfP_{(0,0)}^\theta$};
\node[xshift=0cm,yshift=-\nodeshift] at (V10) {\footnotesize$\sfP_{(1,0)}^\theta$};
\node[xshift=0cm,yshift=-\nodeshift] at (V20) {\footnotesize$\sfP_{(i,0)}^\theta$};
\node[xshift=0cm,yshift=-\nodeshift] at (V30) {\footnotesize$\sfP_{(i+1,0)}^\theta$};
\node[xshift=0cm,yshift=-\nodeshift] at (V40) {\footnotesize$\sfP_{(a-1,0)}^\theta$};
\node[xshift=0cm,yshift=-\nodeshift] at (V50) {\footnotesize$\sfP_{(a,0)}^\theta$};
\node[xshift=0cm,yshift=\nodeshift] at (V01) {\footnotesize$\sfP_{(0,1)}^\theta$};
\node[xshift=0cm,yshift=\nodeshift] at (V11) {\footnotesize$\sfP_{(1,1)}^\theta$};
\node[xshift=0cm,yshift=\nodeshift] at (V21) {\footnotesize$\sfP_{(b-1,1)}^\theta$};
\node[xshift=0cm,yshift=\nodeshift] at (V31) {\footnotesize$\sfP_{(b,1)}^\theta$};
\end{tikzpicture}
}};
\end{tikzpicture}
\end{center}
Thus,
\smash{$\PM_\theta(\Gamma_\perm)=\bigl\{\sfP_{(i,0)}^\theta \mid 0\le i\le a\bigr\}\! \cup\! \bigl\{\sfP_{(j,1)}^\theta \mid 0\le j\le b\bigr\}$}.
We recall that corner perfect matchings \smash{$\sfP^\theta_{(0,0)}$}, \smash{$\sfP^\theta_{(a,0)}$}, \smash{$\sfP^\theta_{(0,1)}$}, \smash{$\sfP^\theta_{(b,1)}$} are the same
for any generic parameter $\theta\in\Theta(Q_\perm)_\RR$ (see Proposition~\ref{prop_cPM_unique}).
Thus, we simply denote them by $\sfP_{(0,0)}$, $\sfP_{(a,0)}$, $\sfP_{(0,1)}$, $\sfP_{(b,1)}$, respectively.
By Proposition~\ref{char_bound}, corner perfect matchings satisfy
\begin{alignat}{3}
&\bigcup_{k=1}^a\Zig(z_k)\subset\sfP_{(0,0)}^\theta=\sfP_{(0,0)}, \qquad && \bigcup_{k=1}^a\Zag(z_k)\subset\sfP_{(a,0)}^\theta=\sfP_{(a,0)}, &\nonumber\\
&\bigcup_{k=1}^b\Zag(w_k)\subset\sfP_{(0,1)}^\theta=\sfP_{(0,1)}, \qquad && \bigcup_{k=1}^b\Zig(w_k)\subset\sfP_{(b,1)}^\theta=\sfP_{(b,1)}.&\label{allzigzag_in_cornerPM}
\end{alignat}
\end{Setting}

Applying Proposition~\ref{prop_description_boundaryPM} to our situation as in Setting~\ref{setting_zigzag_for_a+b},
we have the following.

\begin{Proposition}\label{prop_bPM_zzsequence}\samepage
Let the notation be as in Setting~{\rm \ref{setting_zigzag_for_a+b}}.
For any generic parameter $\theta\in C$ in a~chamber $C\subset\Theta(Q_\perm)_\RR$,
there exist unique sequences $(z_{k_1},\dots, z_{k_a})$, $(w_{k^\prime_1}, \dots, w_{k^\prime_b})$ of zigzag paths
with $\{k_1,\dots, k_a\}=\{1,\dots,a\}$ and $\{k^\prime_1,\dots, k^\prime_b\}=\{1,\dots, b\}$ such that
\begin{gather}
\label{eq_bPM_from_zz1}
\sfP^\theta_{(i,0)}=\bigcup_{k=k_1}^{k_i}\Zag(z_k)\cup\bigcup_{k=k_{i+1}}^{k_a}\Zig(z_k) \cup (\sfP_{(0,0)}\cap\sfP_{(a,0)} ) \quad \text{for any $i=1,\dots, a-1$},
\\
\label{eq_bPM_from_zz2}
\sfP^\theta_{(j,1)}=\bigcup_{k=k^\prime_1}^{k^\prime_j}\Zig(w_k)\cup\bigcup_{k=k^\prime_{j+1}}^{k^\prime_b}\Zag(w_k) \cup (\sfP_{(0,1)}\cap\sfP_{(b,1)} ) \quad \text{for any $j=1,\dots, b-1$}.
\end{gather}
\end{Proposition}

\subsection{Zigzag paths associated to chambers}
\label{subsec_zigzag_ass_chamber}

Rearranging the zigzag paths $z_1, \dots, z_a$, $w_1, \dots, w_b$ as in Setting~\ref{setting_zigzag_for_a+b},
we consider the sequence $(u_1, \dots, u_n)$ of zigzag paths such that
\begin{align}
\label{eq_label_z_w}
&\bullet\, \text{$\{u_1, \dots, u_n\}=\{z_1,\dots, z_a, w_1,\dots, w_b\}$ as sets,} \\
\label{eq_label_z_w2}
&\bullet\, \text{$u_k$ consists of the edges shared by the faces labeled by $k-1$ and $k$ \hspace{-8pt}$\pmod n$.}
\end{align}
Recall that the label of each face used in \eqref{eq_label_z_w2} is determined by the map $t_\perm\colon [n]\rightarrow \{S, H\}$,
see the sentences following \eqref{eq_map_SH} and Figure~\ref{fig_dimer_initial}.
In the following, we consider the set
\[
\{\calZ_\permII\coloneqq (u_{\permII(1)}, \dots, u_{\permII(n)}) \mid \permII\in\fkS_n\}
\]
of sequences of zigzag paths obtained as permutations of $(u_1, \dots, u_n)$.
We assign such a sequence to each chamber in $\Theta(Q_\perm)_\RR$ as follows.

First, for a chamber $C$ in $\Theta(Q_\perm)_\RR$, we have the fine moduli space $\calM_C$
and the triangulation of $\Delta(a,b)$ as explained in Section~\ref{subsec_crepant_resolution}. We will denote such a triangulation by $\Delta_C$.
Note that the argument in \cite[Section~3]{dais2001all} asserts that any triangulation of $\Delta(a,b)$ into elementary triangles is regular,
and thus the crepant resolution $\calM_C\to \Spec R_{a,b}$ induced by a triangulation $\Delta_C$ is projective.
We draw the line $L$ from $\bigl(0,\frac{1}{2}\bigr)$ to $\bigl(a,\frac{1}{2}\bigr)$ which passes through $\Delta_C$.
Let $\{\Delta_{C,k}\}_{k=1}^n$ be the set of elementary triangles in $\Delta_C$,
and we fix the index $k$ so that the line $L$ passes through~$\Delta_{C,k}$ first, then it passes through $\Delta_{C,k+1}$ for any $k=1, \dots, n-1$.
Then, we define the \emph{sign} of~$\Delta_{C,k}$ as
\[
\sgn(\Delta_{C,k})=
\begin{cases}
+1 & \text{if $\Delta_{C,k}$ shares a side with the upper base of $\Delta(a,b)$},\\
-1 & \text{if $\Delta_{C,k}$ shares a side with the lower base of $\Delta(a,b)$},
\end{cases}
\]
and let $\sgn(\Delta_C)\coloneqq \big(\sgn(\Delta_{C,1}), \dots, \sgn(\Delta_{C,n})\big)$.
For example, if a triangulation $\Delta_C$ of $\Delta(3,2)$ takes the form as in Figure~\ref{fig_sgn_triangulation},
then we have $\sgn(\Delta_C)=(+1, -1, -1, +1, -1)$.

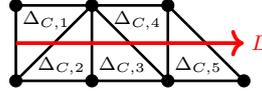
\begin{figure}[!ht]
\centering
\scalebox{1.1}{
\begin{tikzpicture}
\newcommand{\edgewidth}{0.035cm}
\newcommand{\noderad}{0.08} 
\foreach \n/\a/\b in {00/0/0,10/1/0,01/0/1,11/1/1,20/2/0,21/2/1,30/3/0} {
\coordinate (V\n) at (\a,\b); };
\draw [line width=\edgewidth] (V00)--(V30)--(V21)--(V01)--(V00);
\foreach \s/\t in {00/11,10/11, 20/11, 20/21} {\draw [line width=\edgewidth] (V\s)--(V\t);};
\foreach \n in {00,10,01,11,20,21,30}{
\draw [fill=black] (V\n) circle [radius=\noderad]; };
\draw [line width=0.05cm, red, ->] (0,0.5)--(3,0.5);
\node[red] at (3.2,0.5) {\scriptsize $L$};
\node at (0.4,0.8) {\tiny $\Delta_{C,1}$}; \node at (0.6,0.2) {\tiny $\Delta_{C,2}$};
\node at (1.4,0.2) {\tiny $\Delta_{C,3}$}; \node at (1.6,0.8) {\tiny $\Delta_{C,4}$}; \node at (2.4,0.2) {\tiny $\Delta_{C,5}$};
\end{tikzpicture}}
\vspace{-2mm}

\caption{A triangulation of $\Delta(3,2)$ and the labeling of elementary triangles.}
\label{fig_sgn_triangulation}
\end{figure}

Then we define the \emph{sign} of a zigzag path in $\{u_1, \dots, u_n\}$ as
\[
\sgn(u_k)=
\begin{cases}
+1 & \text{if $[u_k]=(0,1)$},\\
-1 & \text{if $[u_k]=(0,-1)$}.
\end{cases}
\]
Note that the number of elementary triangles satisfying $\sgn(\Delta_{C,k})=+1$ (resp.\ $\sgn(\Delta_{C,k})=-1$) coincides with
that of zigzag paths satisfying $\sgn(u_k)=+1$ (resp.\ $\sgn(u_k)=-1$) by the definition of the zigzag polygon.
For example, the zigzag paths $u_1, \dots, u_5$ as in Figure~\ref{fig_naming_zigzag} satisfy
$\sgn(u_1)=\sgn(u_2)=\sgn(u_3)=-1$ and $\sgn(u_4)=\sgn(u_5)=+1$.
For a sequence $\calZ_\permII\coloneqq (u_{\permII(1)}, \dots, u_{\permII(n)})$ with $\permII\in\fkS_n$,
we let $\sgn(\calZ_\permII)\coloneqq (\sgn(u_{\permII(1)}), \dots, \sgn(u_{\permII(n)}))$.

\begin{figure}[!ht]\centering
\scalebox{0.5}{
\begin{tikzpicture}
\newcommand{\noderad}{0.18cm} 
\newcommand{\edgewidth}{0.05cm} 
\newcommand{\nodewidthw}{0.05cm} 
\newcommand{\nodewidthb}{0.04cm} 
\newcommand{\zigzagwidth}{0.15cm} 
\newcommand{\zigzagcolor}{red} 
\newcommand{\zigzagcolortwo}{blue} 

\basicdimerA
\draw[->, line width=\zigzagwidth, rounded corners, color=\zigzagcolor] (1,4)--(B1)--(W1)--(1,0);
\draw[->, line width=\zigzagwidth, rounded corners, color=\zigzagcolor] (3,4)--(B2)--(W2)--(3,0);
\draw[->, line width=\zigzagwidth, rounded corners, color=\zigzagcolor] (5,4)--(B3)--(W3)--(5,0);
\draw[->, line width=\zigzagwidth, rounded corners, color=\zigzagcolortwo] (6,0)--(W4)--(B3)--(6,4);
\draw[->, line width=\zigzagwidth, rounded corners, color=\zigzagcolortwo] (0,0)--(W1)--(0,2);
\draw[->, line width=\zigzagwidth, rounded corners, color=\zigzagcolortwo] (8,2)--(B4)--(8,4);

\foreach \n/\x/\y in {0/0.5/3, 1/2/2, 2/4/2, 3/5.5/1, 4/7/2}{
\node[blue] (V\n) at (\x,\y) {\LARGE$\n$}; };

\node[red] at (1,-0.5) {\LARGE$u_1$}; \node[red] at (3,-0.5) {\LARGE$u_2$}; \node[red] at (5,-0.5) {\LARGE$u_3$};
\node[blue] at (6,4.5) {\LARGE$u_4$}; \node[blue] at (8,4.5) {\LARGE$u_5$};
\end{tikzpicture}}
\vspace{-2mm}

\caption{The zigzag paths $u_1, \dots, u_5$ whose slopes are either $(0,1)$ or $(0,-1)$.}
\label{fig_naming_zigzag}
\end{figure}
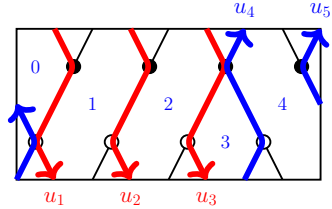

By definition, combining the sequences $(z_{k_1}, \dots, z_{k_a})$ and $(w_{k^\prime_1}, \dots, w_{k^\prime_b})$,
we can obtain the sequence $\calZ_\permII$ satisfying $\sgn(\Delta_C)=\sgn(\calZ_\permII)$.
We record these facts as follows.

\begin{Proposition}
\label{prop_chamber_correspond_zigzag}
Let the notation be as in Setting~{\rm \ref{setting_zigzag_for_a+b}}.
Let $u_1, \dots, u_n$ be the zigzag paths satisfying the conditions \eqref{eq_label_z_w} and \eqref{eq_label_z_w2},
in particular the slope $[u_i]$ is either $(0,1)$ or $(0,-1)$.
Then, for any chamber $C\subset\Theta(Q_\perm)_\RR$, there exists a unique sequence
$\calZ_\permII=(u_{\permII(1)}, \dots, u_{\permII(n)})$ with $\permII\in\fkS_n$ such that
\begin{itemize}\itemsep=0pt
\item[\rm (a)] $\sgn(\Delta_C)=\sgn(\calZ_\permII)$,
\item[\rm (b)] the subsequence of $(u_{\permII(1)}, \dots, u_{\permII(n)})$ consisting of zigzag paths with $\sgn(u_{\permII(i)})=-1$
coincides with $(z_{k_1},\dots, z_{k_a})$,
\item[\rm (c)] the subsequence of $(u_{\permII(1)}, \dots, u_{\permII(n)})$ consisting of zigzag paths with $\sgn(u_{\permII(i)})=+1$
coincides with $(w_{k^\prime_1}, \dots, w_{k^\prime_b})$,
\end{itemize}
where $(z_{k_1},\dots, z_{k_a})$ and $(w_{k^\prime_1}, \dots, w_{k^\prime_b})$ are sequences of zigzag paths respectively
associated to the upper base and the lower base of $\Delta(a,b)$ as in Proposition~{\rm \ref{prop_bPM_zzsequence}}.
\end{Proposition}

We will show that the sequence $\calZ_\permII$ assigned to a chamber $C$ as in Proposition~\ref{prop_chamber_correspond_zigzag}
determines the walls of $C$ and reveal the wall-and-chamber structure of $\Theta(Q)_\RR$ in Section~\ref{subsec_main_wallcrossing}.

\subsection[Combinatorics of dimer models associated to Delta(a,b)]{Combinatorics of dimer models associated to $\boldsymbol{\Delta(a,b)}$}
\label{subsec_jigsaw}

We keep Setting~\ref{setting_zigzag_for_a+b}, but we write $\Gamma=\Gamma_\perm$, $Q=Q_\perm$ for simplicity.
We recall that each wall in $\Theta(Q)_\RR$ is determined by the equation \eqref{eq_wall_equation} in Proposition~\ref{prop_equation_wall},
and the tautological bundle $\calT_\theta=\bigoplus_{v\in Q_0}\calL_v$ used in \eqref{eq_wall_equation} can be obtained by using
perfect matchings in $\PM_\theta(\Gamma)$ as shown in Proposition~\ref{prop_compute_tautological}.
The combinatorial descriptions of $\theta$-stable representations corresponding to three-dimensional cones in $\Sigma_\theta$
are important ingredients to detect the wall-and-chamber structure of $\Theta(Q)_\RR$.
Thus, we review some materials in \cite[Section~4]{IU_moduli}, \cite[Section~4]{Moz}, \cite[Section~3]{CHT} which discuss
$\theta$-stable representations in terms of perfect matchings.

Let $\Sigma_\theta$ be the toric fan of $\calM_\theta$ for a generic parameter $\theta\in\Theta(Q)_\RR$.
For a three-dimensional cone $\sigma\in\Sigma_\theta(3)$, let $\rho_0, \rho_1, \rho_2\in\Sigma_\theta(1)$ be the rays in~$\sigma$.
We denote the $\theta$-stable perfect matchings corresponding to $\rho_0$, $\rho_1$, $\rho_2$ by $\sfP_0$, $\sfP_1$, $\sfP_2$, respectively.
Let $Q^\sigma$ be the subquiver of $Q$ such that the set of vertices coincides with $Q_0$ and
the set of arrows consists of arrows dual to edges not contained in \smash{$\bigcup_{0\le i\le 2}\sfP_i$}.
Note that the arrow set of $Q^\sigma$ coincides with the cosupport of the $\theta$-stable representation $M_\sigma$,
and we see that $Q^\sigma$ is connected.
Let~\smash{$\widetilde{Q}$} be the quiver defined as the dual of the bipartite graph \smash{$\widetilde{\Gamma}$} on~$\RR^2$.
The inverse image of~$Q^\sigma$ under the universal cover $\RR^2\rightarrow\TT$ defines the subquiver of~$\widetilde{Q}$
whose any connected component is identical to $Q^\sigma$.
We choose one of such connected components, and denote it by $\widetilde{Q^\sigma}$.
We consider the subset of $\RR^2$ covered by the faces of $\widetilde{\Gamma}$ dual to the vertices of $\widetilde{Q^\sigma}$.
This subset has properties as in Proposition~\ref{prop_hex_property} below,
thus we call it the \emph{fundamental hexagon} associated to $\sigma$, and denote it by $\Hex(\sigma)$.
We denote the graph obtained as the union of all $\ZZ^2$-translates of the boundary of $\Hex(\sigma)$ by $\Graph(\sigma)$.

\begin{Proposition}[{cf.\ \cite[Proposition~3.4 and its proof]{CHT}}]\label{prop_hex_property}
Let the notation be as above.
For $\sigma\in\Sigma_\theta(3)$, the boundary of $\Hex(\sigma)$ contains precisely six $3$-valent nodes of $\Graph(\sigma)$.
Each chain of edges linking adjacent $3$-valent nodes on the boundary of $\Hex(\sigma)$ comprises an odd number of edges, in which
the edges belong alternately to either a single perfect matching $\sfP_i$ or to the intersection $\sfP_{i-1}\cap\sfP_{i+1}$ of perfect matchings,
where the indices of perfect matchings are taken modulo~$3$ $($cf.\ {\rm \cite[\emph{Figure}~4]{CHT}}$)$.
\end{Proposition}

Let $\sigma_+$, $\sigma_-\in\Sigma_\theta(3)$ be three-dimensional adjacent cones in $\Sigma_\theta$ and let $\tau\coloneqq\sigma_+\cap \sigma_-\in\Sigma_\theta(2)$.
Let $\rho_0$, $\rho_1$, $\rho_2$ and $\rho_1$, $\rho_2$, $\rho_3$ be the rays in $\sigma_+$ and $\sigma_-$, respectively.
On the hyperplane at height one, the pair of cones $\sigma_+$ and $\sigma_-$ takes one of the forms as in Figure~\ref{fig_triangulation_notation}
up to unimodular transformations. Thus we will discuss using these figures.

\begin{figure}[!ht]\centering
{\scalebox{1.1}{
\begin{tikzpicture}
\newcommand{\edgewidth}{0.035cm}
\newcommand{\noderad}{0.08} 

\node at (0,0) {
\begin{tikzpicture}
\foreach \n/\a/\b in {00/0/0,20/1.5/0, 11/0.75/1, 31/2.25/1} {
\coordinate (V\n) at (\a,\b);
};
\foreach \s/\t in {00/20,20/31,31/11,11/00} {\draw [line width=\edgewidth] (V\s)--(V\t) ;};
\draw [line width=\edgewidth] (V20)--(V11) ;
\foreach \n in {00,20,11,31}{\draw [fill=black] (V\n) circle [radius=\noderad] ; };

\newcommand{\nodeshift}{0.3cm}
\node[xshift=0cm,yshift=-\nodeshift] at (V00) {\footnotesize$\rho_0$};
\node[xshift=0cm,yshift=\nodeshift] at (V11) {\footnotesize$\rho_1$};
\node[xshift=0cm,yshift=-\nodeshift] at (V20) {\footnotesize$\rho_2$};
\node[xshift=0cm,yshift=\nodeshift] at (V31) {\footnotesize$\rho_3$};
\node at (0.7,0.25) {\footnotesize$\sigma_+$}; \node at (1.55,0.65) {\footnotesize$\sigma_-$};
\node[fill=white,inner sep=0.5pt, circle] at (1.1, 0.5) {\footnotesize$\tau$};
\end{tikzpicture}};

\node at (5,0) {
\begin{tikzpicture}
\foreach \n/\a/\b in {00/-0.2/0,10/1/0, 11/1/1, 20/2.2/0} {
\coordinate (V\n) at (\a,\b);
};
\foreach \s/\t in {00/10,10/20,20/11,11/00} {\draw [line width=\edgewidth] (V\s)--(V\t) ;};
\draw [line width=\edgewidth] (V10)--(V11) ;
\foreach \n in {00,10,11,20}{\draw [fill=black] (V\n) circle [radius=\noderad] ; };

\newcommand{\nodeshift}{0.3cm}
\node[xshift=0cm,yshift=-\nodeshift] at (V00) {\footnotesize$\rho_0$};
\node[xshift=0cm,yshift=\nodeshift] at (V11) {\footnotesize$\rho_1$};
\node[xshift=0cm,yshift=-\nodeshift] at (V10) {\footnotesize$\rho_2$};
\node[xshift=0cm,yshift=-\nodeshift] at (V20) {\footnotesize$\rho_3$};
\node at (0.6,0.25) {\footnotesize$\sigma_+$}; \node at (1.45,0.25) {\footnotesize$\sigma_-$};
\node[fill=white,inner sep=0.5pt, circle] at (1, 0.5) {\footnotesize$\tau$};
\end{tikzpicture}};

\end{tikzpicture}
}}
\vspace{-2mm}

\caption{The intersection of cones in $\Sigma_\theta(3)$ with the hyperplane at height one.}\label{fig_triangulation_notation}
\end{figure}
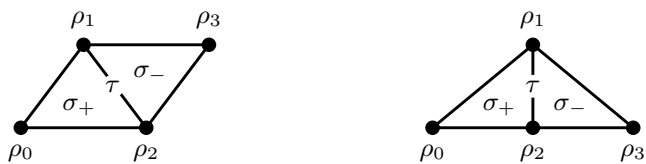

Let $\sfP_0$, $\sfP_1$, $\sfP_2$, $\sfP_3$ be $\theta$-stable perfect matchings corresponding to $\rho_0$, $\rho_1$, $\rho_2$, $\rho_3$, respectively.
Then, we call the closure of a connected component of $\TT{\setminus}\bigcup_{i=0}^3\sfP_i$ a \emph{jigsaw piece} for $\tau$ (cf.\ \cite[Definition~3.10]{CHT}).
By Proposition~\ref{prop_hex_property}, $\fkc_-\coloneqq(\sfP_1\ominus\sfP_3)\cap(\sfP_2\ominus\sfP_3)$ is a subset of the edges in the boundary of $\Hex(\sigma_-)$.
By \cite[Lemma~3.11]{CHT}, $\Hex(\sigma_+)$ is divided into two pieces along edges of~$\fkc_-$ and these pieces are precisely the jigsaw pieces for $\tau$.
Similarly, $\fkc_+\coloneqq(\sfP_0\ominus\sfP_1)\cap(\sfP_0\ominus\sfP_2)$ is a~subset of the edges in the boundary of $\Hex(\sigma_+)$,
and $\Hex(\sigma_-)$ is divided into two pieces, which are precisely the jigsaw pieces for~$\tau$, along edges of $\fkc_+$.
In particular, the following statement holds.

\begin{Proposition}[{\cite[Theorem~3.12]{CHT}, see also \cite{Nak_GHilb}}]\label{prop_jigsaw}
Let $J$, $J^\prime$ be the jigsaw pieces for $\tau$ obtained by cutting $\Hex(\sigma_+)$ in $\RR^2$ along the edges of $\fkc_-$. Then, we have the fundamental hexagon $\Hex(\sigma_-)$ by certain $\ZZ^2$-translations of $J$ and $J^\prime$.
\end{Proposition}

\begin{Example}
We consider the dimer model $\Gamma$ given in Figure~\ref{ex_dimer1_basic}, in which the zigzag polygon~$\Delta_\Gamma$ is $\Delta(3,2)$.
For the zigzag paths $u_1, \dots, u_5$ of $\Gamma$ shown in Figure~\ref{fig_naming_zigzag},
we consider the sequence
\[
\calZ_\permII=(u_{\permII(1)}, \dots, u_{\permII(5)})=(u_4,u_1,u_3,u_5,u_2)
\]
with $\permII=(1452)\in\fkS_5$. We suppose that $C\subset\Theta(Q)_\RR$ is a chamber associated to the sequence $\calZ_\permII$ as in Proposition~\ref{prop_chamber_correspond_zigzag}.
By the condition (a) in Proposition~\ref{prop_chamber_correspond_zigzag},
we see that the triangulation $\Delta_C$ takes the form as in Figure~\ref{fig_sgn_triangulation} since $\sgn(\calZ_\permII)=(+1,-1,-1,+1,-1)$.
For $\theta\in C$, let $\sigma_+\in\Sigma_\theta(3)$ be the three-dimensional cone
whose rays pass through the lattice points $(1,0)$, $(1,1)$, and $(2,0)$ in $\Delta_C$.
Similarly, let $\sigma_-\in\Sigma_\theta(3)$ be the three-dimensional cone
whose rays pass through the lattice points $(1,1)$, $(2,0)$, and $(2,1)$ in $\Delta_C$.

\begin{center}
{\scalebox{0.9}{
\begin{tikzpicture}
\newcommand{\edgewidth}{0.035cm}
\newcommand{\noderad}{0.08} 
\foreach \n/\a/\b in {10/0/0,20/1.5/0, 11/0/1.5, 21/1.5/1.5} {
\coordinate (V\n) at (\a,\b);
};
\foreach \s/\t in {10/20,20/21,21/11,11/10} {\draw [line width=\edgewidth] (V\s)--(V\t) ;};
\draw [line width=\edgewidth] (V20)--(V11) ;
\foreach \n in {10,20,11,21}{\draw [fill=black] (V\n) circle [radius=\noderad] ; };

\newcommand{\nodeshift}{0.4cm}
\node[xshift=0cm,yshift=-\nodeshift] at (V10) {$(1,0)$};
\node[xshift=0cm,yshift=\nodeshift] at (V11) {$(1,1)$};
\node[xshift=0cm,yshift=-\nodeshift] at (V20) {$(2,0)$};
\node[xshift=0cm,yshift=\nodeshift] at (V21) {$(2,1)$};
\node at (0.5,0.5) {$\sigma_+$}; \node at (1.1,1) {$\sigma_-$};
\end{tikzpicture}}}
\end{center}

The $\theta$-stable perfect matchings corresponding to these lattice points can be obtained from~$\calZ_\permII$ as Proposition~\ref{prop_bPM_zzsequence}, in particular we have the following:
\begin{itemize}\itemsep=0pt
\item $\sfP^\theta_{(1,0)}=\sfP_{\{1\}}$ and $\sfP^\theta_{(2,0)}=\sfP_{\{1,3\}}$, where $\sfP_{\{1\}}$, $\sfP_{\{1,3\}}$ are perfect matchings shown in Fig\-ure~\ref{fig_ex_PMfromzigzag},
\item $\sfP^\theta_{(2,1)}$ is the corner perfect matching $\sfP_2$ shown in Example~\ref{ex_PM_corner},
\item $\sfP^\theta_{(1,1)}=\big(\sfP^\theta_{(2,1)}{\setminus}\Zig(u_5)\big)\cup\Zag(u_5)$.
\end{itemize}

Considering the edges in $\sfP^\theta_{(1,0)}\cup\sfP^\theta_{(1,1)}\cup\sfP^\theta_{(2,0)}$ and
$\sfP^\theta_{(1,1)}\cup\sfP^\theta_{(2,0)}\cup\sfP^\theta_{(2,1)}$,
we have the fundamental hexagons $\Hex(\sigma_+)$ and $\Hex(\sigma_-)$ associated to
$\sigma_+$ and $\sigma_-$ as shown in Figures~\ref{fig_hex+} and \ref{fig_hex-}, respectively.
Let
\begin{gather*}
\fkc_-=\bigl(\sfP^\theta_{(1,1)}\ominus\sfP^\theta_{(2,1)}\bigr)\cap\bigl(\sfP^\theta_{(2,0)}\ominus\sfP^\theta_{(2,1)}\bigr) , \qquad
\fkc_+=\bigl(\sfP^\theta_{(1,0)}\ominus\sfP^\theta_{(1,1)}\bigr)\cap\bigl(\sfP^\theta_{(1,0)}\ominus\sfP^\theta_{(2,0)}\bigr).
\end{gather*}
Then, the edge $\fkc_-$ cuts $\Hex(\sigma_+)$ into two jigsaw pieces, and the edge $\fkc_+$ cuts $\Hex(\sigma_-)$ into two jigsaw pieces.
By sliding a jigsaw piece in $\Hex(\sigma_+)$, we can obtain $\Hex(\sigma_-)$, and vice versa.

\newcommand{\jigsawbasicsetting}{
\foreach \n/\a/\b in {1/1.5/3, 2/3.5/3, 3/5.5/3, 4/7.5/3} {\coordinate (00B\n) at (\a,\b);}
\foreach \n/\a/\b in {1/0.5/1,2/2.5/1, 3/4.5/1, 4/6.5/1} {\coordinate (00W\n) at (\a,\b);}
\foreach \n/\a/\b in {1/1.5/3, 2/3.5/3, 3/5.5/3, 4/7.5/3} {\coordinate (10B\n) at (\a+8,\b);}
\foreach \n/\a/\b in {1/0.5/1,2/2.5/1, 3/4.5/1, 4/6.5/1} {\coordinate (10W\n) at (\a+8,\b);}
\foreach \n/\a/\b in {1/1.5/3, 2/3.5/3, 3/5.5/3, 4/7.5/3} {\coordinate (01B\n) at (\a,\b+4);}
\foreach \n/\a/\b in {1/0.5/1,2/2.5/1, 3/4.5/1, 4/6.5/1} {\coordinate (01W\n) at (\a,\b+4);}
\foreach \n/\a/\b in {1/1.5/3, 2/3.5/3, 3/5.5/3, 4/7.5/3} {\coordinate (11B\n) at (\a+8,\b+4);}
\foreach \n/\a/\b in {1/0.5/1,2/2.5/1, 3/4.5/1, 4/6.5/1} {\coordinate (11W\n) at (\a+8,\b+4);}
\foreach \n/\a/\b in {1/1.5/3, 2/3.5/3, 3/5.5/3, 4/7.5/3} {\coordinate (02B\n) at (\a,\b+8);}
\foreach \n/\a/\b in {1/0.5/1,2/2.5/1, 3/4.5/1, 4/6.5/1} {\coordinate (02W\n) at (\a,\b+8);}
\foreach \n/\a/\b in {1/1.5/3, 2/3.5/3, 3/5.5/3, 4/7.5/3} {\coordinate (12B\n) at (\a+8,\b+8);}
\foreach \n/\a/\b in {1/0.5/1,2/2.5/1, 3/4.5/1, 4/6.5/1} {\coordinate (12W\n) at (\a+8,\b+8);}
\coordinate (exB) at (7.5,-1); \coordinate (exW) at (16.5,5);
}

\newcommand{\jigsawbasic}{
\foreach \p/\q in {0/0,1/0,0/1,1/1,0/2,1/2}{
\foreach \w/\b in {1/1,2/2,3/3,4/3} {\draw[line width=\edgewidth] (\p\q W\w)--(\p\q B\b); };
\foreach \w/\s/\t in {1/0/2,1/0/0,1/1/0,2/2/0,2/3/0,3/4/0,3/5/0,4/6/0,4/7/0} {\draw[line width=\edgewidth] (\p\q W\w)--(\s+8*\p,\t+4*\q); };
\foreach \b/\s/\t in {1/1/4,1/2/4,2/3/4,2/4/4,3/5/4,3/6/4,4/7/4,4/8/4,4/8/2} {\draw[line width=\edgewidth] (\p\q B\b)--(\s+8*\p,\t+4*\q); };
\draw[line width=\edgewidth] (exB)--(7,0); \draw[line width=\edgewidth] (exB)--(8,0);
\draw[line width=\edgewidth] (exW)--(16,4); \draw[line width=\edgewidth] (exW)--(16,6);
\foreach \x in {1,2,3,4} {\filldraw [fill=black, line width=\nodewidthb] (\p\q B\x) circle [radius=\noderad] ;};
\filldraw [fill=black, line width=\nodewidthb] (exB) circle [radius=\noderad];
\foreach \x in {1,2,3,4} {\filldraw [fill=white, line width=\nodewidthw] (\p\q W\x) circle [radius=\noderad] ;};
\filldraw [fill=white, line width=\nodewidthw] (exW) circle [radius=\noderad] ;
\foreach \n/\x/\y in {0/0.5/3, 1/2/2, 2/4/2, 3/5.5/1, 4/7/2}{
\node[blue] at (\x+8*\p,\y+4*\q) {\LARGE$\n$}; };
\draw[line width=\edgewidth, gray] (0+8*\p,0+4*\q) rectangle (8+8*\p,4+4*\q);
};}

\begin{figure}[!ht]
\centering
\scalebox{0.45}{
\begin{tikzpicture}
\newcommand{\noderad}{0.18cm} 
\newcommand{\edgewidth}{0.05cm} 
\newcommand{\nodewidthw}{0.05cm} 
\newcommand{\nodewidthb}{0.04cm} 
\newcommand{\pmwidth}{0.28cm} 
\newcommand{\pmcolor}{red!40} 
\jigsawbasicsetting

\filldraw[lightgray!50] (exB)--(11W2)--(10B2)--(11B3)--(12W3)--(11B2)--(12W2)--(00B4)--(01B3)--(01W3)--(exB);
\foreach \p/\q in {0/0,1/0,0/1,1/1,0/2,1/2}{
\draw[line width=\pmwidth, \pmcolor] (\p\q W1)--(\p\q B1);
\foreach \w/\s/\t in {2/3/0,3/5/0,4/7/0} {\draw[line width=\pmwidth, \pmcolor] (\p\q W\w)--(\s+8*\p,\t+4*\q); };
\foreach \b/\s/\t in {2/3/4,3/5/4,4/7/4} {\draw[line width=\pmwidth, \pmcolor] (\p\q B\b)--(\s+8*\p,\t+4*\q); };};
\draw[line width=\pmwidth, \pmcolor] (exB)--(7,0);
\foreach \p/\q in {0/0,1/0,0/1,1/1,0/2,1/2}{
\draw[line width=\pmwidth, \pmcolor] (\p\q W1)--(\p\q B1);
\draw[line width=\pmwidth, \pmcolor] (\p\q W3)--(\p\q B3);
\foreach \w/\s/\t in {2/3/0,4/7/0} {\draw[line width=\pmwidth, \pmcolor] (\p\q W\w)--(\s+8*\p,\t+4*\q); };
\foreach \b/\s/\t in {2/3/4,4/7/4} {\draw[line width=\pmwidth, \pmcolor] (\p\q B\b)--(\s+8*\p,\t+4*\q); };};
\draw[line width=\pmwidth, \pmcolor] (exB)--(7,0);
\foreach \p/\q in {0/0,1/0,0/1,1/1,0/2,1/2}{
\draw[line width=\pmwidth, \pmcolor] (\p\q W4)--(\p\q B3);
\foreach \w/\s/\t in {1/0/0, 2/2/0, 3/4/0} {\draw[line width=\pmwidth, \pmcolor] (\p\q W\w)--(\s+8*\p,\t+4*\q); };
\foreach \b/\s/\t in {1/2/4, 2/4/4, 4/8/4} {\draw[line width=\pmwidth, \pmcolor] (\p\q B\b)--(\s+8*\p,\t+4*\q); };};
\draw[line width=\pmwidth, \pmcolor] (exB)--(8,0); \draw[line width=\pmwidth, \pmcolor] (exW)--(16,4);
\draw[line width=0.25cm, blue!60] (00B4)--(10W1);

\jigsawbasic

\end{tikzpicture}}
\vspace{-2mm}

\caption{A connected component (e.g., the grayed area) is the fundamental hexagon $\Hex(\sigma_+)$, and the blue edge is $\fkc_-$.}
\label{fig_hex+}
\end{figure}

\begin{figure}[!ht]\centering
\scalebox{0.45}{
\begin{tikzpicture}
\newcommand{\noderad}{0.18cm} 
\newcommand{\edgewidth}{0.05cm} 
\newcommand{\nodewidthw}{0.05cm} 
\newcommand{\nodewidthb}{0.04cm} 
\newcommand{\pmwidth}{0.28cm} 
\newcommand{\pmcolor}{red!40} 
\jigsawbasicsetting

\filldraw[lightgray!50] (00B4)--(10W1)--(11W2)--(10B2)--(11B3)--(10B4)--(exW)--(12B3)--(11B2)--(12W2)--(00B4);
\foreach \p/\q in {0/0,1/0,0/1,1/1,0/2,1/2}{
\draw[line width=\pmwidth, \pmcolor] (\p\q W1)--(\p\q B1);
\draw[line width=\pmwidth, \pmcolor] (\p\q W3)--(\p\q B3);
\foreach \w/\s/\t in {2/3/0,4/7/0} {\draw[line width=\pmwidth, \pmcolor] (\p\q W\w)--(\s+8*\p,\t+4*\q); };
\foreach \b/\s/\t in {2/3/4,4/7/4} {\draw[line width=\pmwidth, \pmcolor] (\p\q B\b)--(\s+8*\p,\t+4*\q); };};
\draw[line width=\pmwidth, \pmcolor] (exB)--(7,0);
\foreach \p/\q in {0/0,1/0,0/1,1/1,0/2,1/2}{
\draw[line width=\pmwidth, \pmcolor] (\p\q W4)--(\p\q B3);
\foreach \w/\s/\t in {1/0/0, 2/2/0, 3/4/0} {\draw[line width=\pmwidth, \pmcolor] (\p\q W\w)--(\s+8*\p,\t+4*\q); };
\foreach \b/\s/\t in {1/2/4, 2/4/4, 4/8/4} {\draw[line width=\pmwidth, \pmcolor] (\p\q B\b)--(\s+8*\p,\t+4*\q); };};
\draw[line width=\pmwidth, \pmcolor] (exB)--(8,0); \draw[line width=\pmwidth, \pmcolor] (exW)--(16,4);
\foreach \p/\q in {0/0,1/0,0/1,1/1,0/2,1/2}{
\draw[line width=\pmwidth, \pmcolor] (\p\q W4)--(\p\q B3);
\foreach \w/\s/\t in {1/0/2, 2/2/0, 3/4/0} {\draw[line width=\pmwidth, \pmcolor] (\p\q W\w)--(\s+8*\p,\t+4*\q); };
\foreach \b/\s/\t in {1/2/4, 2/4/4, 4/8/2} {\draw[line width=\pmwidth, \pmcolor] (\p\q B\b)--(\s+8*\p,\t+4*\q); };};
\draw[line width=\pmwidth, \pmcolor] (exW)--(16,6);
\draw[line width=0.25cm, blue!60] (11B3)--(12W3);

\jigsawbasic
\end{tikzpicture}}
\vspace{-2mm}

\caption{A connected component (e.g., the grayed area) is the fundamental hexagon $\Hex(\sigma_-)$, and the blue edge is $\fkc_+$.}\label{fig_hex-}
\end{figure}
\end{Example}

We show some combinatorial statements concerning our consistent dimer model $\Gamma=\Gamma_\perm$.

\begin{Lemma}
\label{lem_property_c}
Let the notation be as above.
We see that $\fkc_+$ $($resp.\ $\fkc_-)$ is a single edge contained in $\sfP_0$ $($resp.\ $\sfP_3)$.
\end{Lemma}

\begin{proof}
Since $\fkc_+$ is a subset of the edges in the boundary of $\Hex(\sigma_+)$, the edges consisting of $\fkc_+$ are contained in
either $\sfP_0$ or $\sfP_{1}\cap\sfP_{2}$ by Proposition~\ref{prop_hex_property}.
In our situation, we claim that $\sfP_{1}\cap\sfP_{2}=\varnothing$.
Indeed, since $\rho_1$ (resp.\ $\rho_2$) corresponds to a lattice point on the upper (resp.\ lower) base of $\Delta(a,b)$,
if there exists an edge $e$ such that $e\in\sfP_1\cap\sfP_2$,
then we have $e\in\sfP_{(0,0)}\cap\sfP_{(a,0)}\cap\sfP_{(0,1)}\cap\sfP_{(b,1)}$
by the description of perfect matchings as in Proposition~\ref{prop_bPM_zzsequence} and Lemma~\ref{lem_zigzag_nointersection}.
By Proposition~\ref{zigzag_sidepolygon}, this implies that any zigzag path does not pass through the edge $e$, which is a contradiction.
The assertion for $\fkc_+$ follows from this claim. We have the assertion for $\fkc_-$ by a similar argument.
\end{proof}

\begin{Lemma}
\label{lem_property_interior}
Let the notation be as above.
Any edge contained in the strict interior of $\Hex(\sigma_+)$ $($resp.\ $\Hex(\sigma_-))$ does not belong to $\sfP_0\cup \sfP_1\cup \sfP_2$
$($resp.\ $\sfP_1\cup \sfP_2\cup \sfP_3)$.
Thus, the subquiver of~$\widetilde{Q}$ obtained by restricting $\widetilde{Q}$ to the strict interior of $\Hex(\sigma_+)$ $($resp.\ $\Hex(\sigma_-))$
coincides with the quiver $\widetilde{Q^{\sigma_+}}$ $($resp.\ $\widetilde{Q^{\sigma_-}})$.
\end{Lemma}

\begin{proof}
By \cite[Corollary~4.18]{Moz}, we see that if there is an edge $e\in \sfP_0\cup \sfP_1\cup \sfP_2$ contained in the strict interior of $\Hex(\sigma_+)$,
then it satisfies $e\in \sfP_0\cap \sfP_1\cap \sfP_2$.
Since $\sfP_1\cap\sfP_2=\varnothing$ (see the proof of Lemma~\ref{lem_property_c}), we have the assertion for $\Hex(\sigma_+)$.
The assertion for $\Hex(\sigma_-)$ can be shown by a~similar argument.
\end{proof}

Here, we note that by Proposition~\ref{prop_chamber_correspond_zigzag},
for any triangulation $\Delta_C$ associated to a chamber $C\subset \Theta(Q)_\RR$,
we can assign the zigzag path $z_{k_i}$ to the line segment between $(i-1, 0)$ and $(i,0)$ for all $i=1, \dots, a$,
and assign the zigzag path $w_{k^\prime_j}$ to the line segment between $(j-1, 1)$ and $(j,1)$ for all $j=1, \dots, b$.
Concerning such zigzag paths, we have the following lemma.

\begin{Lemma}\label{lem_zigzag_containPM}
Let the notation be as above.
\begin{itemize}\itemsep=0pt
\item[$(1)$] We consider the cones in $\Sigma_\theta$ as shown in the left of {\rm Figure~\ref{fig_triangulation_notation}}.
Let $z_{k_i}$ $($resp.\ $w_{k^\prime_j})$ be a~zigzag path with $[z_{k_i}]=(0,-1)$ $($resp.\ $[w_{k^\prime_j}]=(0,1))$ assigned to the line segment comprised lattice points
corresponding to $\rho_0$ and $\rho_2$ $($resp.\ $\rho_1$ and $\rho_3)$.
Then we see that $\fkc_+\in\Zig(z_{k_i})$ and $\fkc_-\in \Zig(w_{k^\prime_j})$.
\item[$(2)$] We consider the cones in $\Sigma_\theta$ as shown in the right of {\rm Figure~\ref{fig_triangulation_notation}}.
Let $z_{k_i}$ $($resp.\ $z_{k_{i+1}})$ be a zigzag path with $[z_{k_i}]=(0,-1)$ $($resp.\ $[z_{k_{i+1}}]=(0,-1))$ assigned to the line segment comprised lattice points
corresponding to $\rho_0$ and $\rho_2$ $($resp.\ $\rho_2$ and $\rho_3)$.
Then we see that $\fkc_+\in\Zig(z_{k_i})$ and $\fkc_-\in \Zag(z_{k_{i+1}})$.
\end{itemize}
\end{Lemma}

\begin{proof} (1)
By Proposition~\ref{prop_hex_property} and Lemma~\ref{lem_property_c}, the edge $\fkc_+$ belongs to a single perfect matching~$\sfP_0$,
and hence $\fkc_+\not\in\sfP_2$ in particular.
By Proposition~\ref{prop_bPM_zzsequence}, we know that $\sfP_0$ and $\sfP_2$ are the same except the edges contained in $z_{k_i}$,
and $\Zig(z_{k_i})\in\sfP_0$, $\Zag(z_{k_i})\in\sfP_2$.
Thus, we see that $\fkc_+\in\Zig(z_{k_i})$.
Similarly, the edge $\fkc_-$ belongs to a single perfect matching $\sfP_3$, and hence $\fkc_+\not\in\sfP_1$ in particular.
Since $\sfP_1$ and $\sfP_3$ are the same except the edges contained in $w_{k^\prime_j}$
and $\Zig(w_{k^\prime_j})\in\sfP_3$, $\Zag(w_{k^\prime_j})\in\sfP_1$, we see that $\fkc_-\in \Zig(w_{k^\prime_j})$.

(2) By the same argument as (1), we see that $\fkc_+\in\Zig(z_{k_i})$.
Concerning the edge $\fkc_-$, by Proposition~\ref{prop_hex_property} and Lemma~\ref{lem_property_c},
$\fkc_-$ belongs to a single perfect matching $\sfP_3$, and hence $\fkc_+\not\in\sfP_2$ in particular.
By Proposition~\ref{prop_bPM_zzsequence}, we know that $\sfP_2$ and $\sfP_3$ are the same except the edges contained in $z_{k_{i+1}}$,
and $\Zig(z_{k_{i+1}})\in\sfP_2$, $\Zag(z_{k_{i+1}})\in\sfP_3$.
Thus, we see that $\fkc_-\in\Zag(z_{k_{i+1}})$.
\end{proof}

\subsection{Wall crossings and zigzag paths}\label{subsec_main_wallcrossing}

\begin{Setting}
\label{setting_zigzag_for_u}
Let $\Gamma=\Gamma_\perm$ be a consistent dimer model for some $\perm\in\fkS_n$.
Let $u_1, \dots, u_n$ be zigzag paths on $\Gamma$ satisfying \eqref{eq_label_z_w} and \eqref{eq_label_z_w2}.
In particular, the slope $[u_k]$ is either $(0,-1)$ or $(0,1)$ for $k=1, \dots, n$.
We define a total order $<$ on $\{u_1, \dots, u_n\}$ as $u_n<u_{n-1}<\cdots< u_2< u_1$.

Suppose that a chamber $C\subset\Theta(Q)_\RR$ corresponds to a sequence $\calZ_\permII=(u_{\permII(1)}, \dots, u_{\permII(n)})$ with $\permII\in\fkS_n$
as in Proposition~\ref{prop_chamber_correspond_zigzag}.
Let $\Delta_C$ be the triangulation corresponding to $\calM_C$ and $\{\Delta_{C,k}\}_{k=1}^n$ be the set of elementary triangles in $\Delta_C$,
in which we have $\sgn(\Delta_{C,k})=\sgn(u_{\permII(k)})$ for any $k=1, \dots, n$. In particular, we can assign $u_{\permII(k)}$ to $\Delta_{C,k}$.
Also, for any $\theta\in C$, we denote the associated toric fan by $\Sigma_C=\Sigma_\theta$.
\end{Setting}

By Lemma~\ref{lem_zigzag_nointersection}, we see that any pair of zigzag paths $(u_i, u_j)$ on $\Gamma$ divide the two-torus $\TT$ into two parts (see Figure~\ref{fig_region_zigzag}).
We denote the region containing the face dual to the specific vertex $0\in Q_0$ by $\calR^-(u_i,u_j)$, and the other region by $\calR^+(u_i,u_j)$.
By abuse of notation, we also use the notation $\calR^\pm(u_i,u_j)$ for the set of vertices of $Q$ contained in
$\calR^\pm(u_i,u_j)$.
Since we essentially use one of $\calR^\pm(u_i,u_j)$, we let $\calR(u_i,u_j)\coloneqq\calR^+(u_i,u_j)$.

\begin{figure}[!ht]\centering
\scalebox{0.85}{
\begin{tikzpicture}
\newcommand{\edgewidth}{0.06cm} 
\foreach \n/\a/\b in {1/2/0, 2/1.25/0.75, 3/2/1.5, 4/1.25/2.25, 5/2/3} {\coordinate (z1\n) at (\a+0.5,\b); \coordinate (z2\n) at (\a+4,\b);}
\draw[line width=0.03cm] (0,0) rectangle (8,3);
\draw[line width=\edgewidth, red] (z11)--(z12)--(z13)--(z14)--(z15); \draw[line width=\edgewidth, red] (z21)--(z22)--(z23)--(z24)--(z25);
\node at (4,1.5) {$\calR^+(u_i,u_j)$}; \node at (1.1,1.5) {$\calR^-(u_i,u_j)$}; \node at (7,1.5) {$\calR^-(u_i,u_j)$};
\node[red] at (2.5,-0.37) {\Large $u_i$}; \node[red] at (6,-0.37) {\Large $u_j$};
\end{tikzpicture}}
\vspace{-2mm}

\caption{}\label{fig_region_zigzag}
\end{figure}

We are now ready to state our theorem.

\begin{Theorem}\label{thm_main_wall}
Let the notation be as in Setting~{\rm \ref{setting_zigzag_for_u}}.
We suppose that $\ell_k$ is an exceptional curve in $\calM_C$
and $\Delta_{C,k}$, $\Delta_{C,k+1}$ are elementary triangles in the triangulation $\Delta_C$ such that $\Delta_{C,k}\cap \Delta_{C,k+1}$
is the line segment corresponding to $\ell_k$.
\begin{itemize}\itemsep=0pt
\item[$(1)\hphantom{{}'}$]
For any $k=1, \dots, n-1$, the equation \eqref{eq_wall_equation} derived from $\ell_k$ takes the form
$\sum_{v\in \calR_k}\theta_v=0$, where $\calR_k\coloneqq\calR(u_{\permII(k)}, u_{\permII(k+1)})$, and
\[
W_{k}\coloneqq\biggl\{\theta\in\Theta(Q)_\RR \mid \sum_{v\in \calR_k}\theta_v=0\biggr\}
\]
is a wall of $C$.
\item[$(2)\hphantom{{}'}$] The wall $W_k$ is of type {\rm I} if and only if $[u_{\permII(k)}]= -[u_{\permII(k+1)}]$.
\item[$(2)'$] The wall $W_k$ is of type {\typeIII} if and only if $[u_{\permII(k)}]=[u_{\permII(k+1)}]$.
\item[$(3)\hphantom{{}'}$] Any parameter $\theta\in C$ satisfies $\sum_{v\in \calR_k}\theta_v>0$ if $u_{\permII(k)}< u_{\permII(k+1)}$.
\item[$(3)'$] Any parameter $\theta\in C$ satisfies $\sum_{v\in \calR_k}\theta_v<0$ if $u_{\permII(k)}> u_{\permII(k+1)}$.
\end{itemize}
\end{Theorem}

\begin{proof}
First, we assume that $\ell=\ell_k$ is floppable.
Then $\Delta_{C,k}$ and $\Delta_{C,k+1}$ form a parallelogram and $\ell$ corresponds to a diagonal of the parallelogram.
Suppose that the vertices of the parallelogram are $(i-1,0)$, $(i,0)$, $(j-1,1)$ and $(j,1)$.

For $\theta\in C$, let $\rho_0$, $\rho_1$, $\rho_2$, $\rho_3$ be the rays in $\Sigma_C(1)$ corresponding to the lattice points
$(i-1,0)$, $(j-1,1)$, $(i,0)$, $(j,1)$ in $\Delta_C$, respectively.
Let $\sfP_0$, $\sfP_1$, $\sfP_2$, $\sfP_3$ be $\theta$-stable perfect matchings corresponding to $\rho_0$, $\rho_1$, $\rho_2$, $\rho_3$, respectively,
that is,
$
\sfP_0=\sfP_{(i-1,0)}^\theta$, $\sfP_1=\sfP_{(j-1,1)}^\theta$, $\sfP_2=\sfP_{(i,0)}^\theta$, $\sfP_3=\sfP_{(j,1)}^\theta$.
We suppose that the diagonal connecting $(j-1,1)$ and $(i,0)$ corresponds to a cone in $\Sigma_C(2)$.
(Note that the case where the diagonal connecting $(i-1,0)$ and $(j,1)$ corresponds to a cone in $\Sigma_C(2)$ can be shown by a similar argument.)
We consider the cones $\sigma_+$, $\sigma_-\in\Sigma_C(3)$ whose rays are respectively $\rho_0$, $\rho_1$, $\rho_2$ and $\rho_1$, $\rho_2$, $\rho_3$.
Thus, $\tau=\sigma_+\cap\sigma_-$ is the cone in $\Sigma_C(2)$ corresponding to~$\ell$ (see the left of Figure~\ref{fig_triangulation_notation}).
By these settings, $\Delta_{C,k}$ (resp.\ $\Delta_{C,k+1}$) is obtained as the intersection of the cone $\sigma_+$ (resp.\ $\sigma_-$) and the hyperplane at height one,
and $u_{\permII(k)}=z_{k_i}$, $u_{\permII(k+1)}=w_{k_j^\prime}$ in the terminology of Proposition~\ref{prop_chamber_correspond_zigzag},
thus $[u_{\permII(k)}]= -[u_{\permII(k+1)}]$. We divide the arguments into two cases:
\begin{itemize}\itemsep=0pt \setlength{\leftskip}{1cm}
\item[(Case 1)] The case where $z_{k_i}=u_{\permII(k)}<u_{\permII(k+1)}=w_{k_j^\prime}$.
\item[(Case 2)] The case where $z_{k_i}=u_{\permII(k)}>u_{\permII(k+1)}=w_{k_j^\prime}$.
\end{itemize}
For these cases, we show the following.
\begin{itemize}\itemsep=0pt \setlength{\leftskip}{1cm}\samepage
\item[(Case 1)] We have
\begin{equation}
\label{eq_degree_linebdl_main}
\begin{cases}
\deg(\calL_v |_\ell)=1 & \text{for any $v\in \calR_k$},\\
\deg(\calL_v |_\ell)=0 & \text{otherwise},
\end{cases}
\end{equation}
which means that
\begin{equation}
\label{eq_degree_linebdl_sum}
\sum_{v\in Q_0}\deg(\calL_v |_\ell)\theta_v=\sum_{v\in \calR_k}\theta_v.
\end{equation}
\item[(Case 2)] We have
\begin{equation}
\label{eq_degree_linebdl_main2}
\begin{cases}
\deg(\calL_v |_\ell)=-1 & \text{for any $v\in \calR_k$},\\
\deg(\calL_v |_\ell)=0 & \text{otherwise},
\end{cases}
\end{equation}
which means that
\begin{equation}\label{eq_degree_linebdl_sum2}
\sum_{v\in Q_0}\deg(\calL_v |_\ell)\theta_v=-\sum_{v\in \calR_k}\theta_v.
\end{equation}
\end{itemize}

To show this, we compute the coordinate function on the toric chart in $\calM_C$ corresponding to~$\tau$
by the argument similar to \cite[Proposition~4.9]{CHT}.
Let $v_{\rho_0}, v_{\rho_1}, v_{\rho_2}, v_{\rho_3}\in\sfN$ be the generators of the rays $\rho_0, \rho_1, \rho_2, \rho_3\in\Sigma_C(1)$, respectively.
Let $m\in\sfM$ be the primitive vector such that $\langle m, n\rangle=0$ for any $n\in\tau$ and $\langle m, n\rangle\ge0$ for any $n\in\sigma_+$.
Thus, we have that
\[
\langle m, v_{\rho_1}\rangle=\langle m, v_{\rho_2}\rangle=0, \qquad \langle m, v_{\rho_0}\rangle>0, \qquad \langle m, v_{\rho_3}\rangle<0.
\]
Since $m$ is primitive and $(v_{\rho_0}\, v_{\rho_1}\, v_{\rho_2})$, $(v_{\rho_1}\, v_{\rho_2}\, v_{\rho_3})\in\GL(3,\ZZ)$,
we have $\langle m, v_{\rho_0}\rangle=1$, $\langle m, v_{\rho_3}\rangle=-1$.
We identify $\CC[\sfM]$ with a subring of $\CC[t^{\pm}_\rho \mid \rho\in\Sigma_C(1)]$ via the natural inclusion
$\sfM\hookrightarrow \ZZ^{\Sigma_C(1)}$. Then, we can write $t^m\in \CC[t^{\pm}_\rho \mid \rho\in\Sigma_C(1)]$ as
\[
t^m=\frac{t_{\rho_0}}{t_{\rho_3}} t^u,
\]
where $t^u$ is a Laurent monomial not containing $t_{\rho_0}$, $t_{\rho_1}$, $t_{\rho_2}$, $t_{\rho_3}$ as its factor.

Then we consider the subquivers \smash{$\widetilde{Q^{\sigma_+}}$} and \smash{$\widetilde{Q^{\sigma_-}}$} of \smash{$\widetilde{Q}$} (see Lemma~\ref{lem_property_interior}).
For each vertex $v\in Q_0$, let $\gamma_v^+$ (resp.\ $\gamma_v^-$) be a weak path in the double quiver of \smash{$\widetilde{Q^{\sigma_+}}$} (resp.\ \smash{$\widetilde{Q^{\sigma_-}}$})
from the vertex~$0$ to~$v$.
As we saw in Proposition~\ref{prop_compute_tautological}, the line bundle $\calL_v$ depends only on the target vertex $v$,
thus we may choose $\gamma_v^\pm$ so that it passes through the same vertex at most once.
We let $U_{\sigma_{\pm}}\coloneqq \Spec\CC[\sigma_{\pm}^\vee\cap \sfM]$ for the toric chart in $\calM_C$ corresponding to $\sigma_{\pm}$.
We consider the generating sections
\[
t^{\deg(\gamma^\pm_v)}\coloneqq\prod_{\rho\in\Sigma_C(1)}t_\rho^{\deg_{\sfP_\rho}(\gamma^\pm_v)}
\]
of $\rmH^0(U_{\sigma_\pm}, \calL_v)$, where $\sfP_\rho$ is the $\theta$-stable perfect matching corresponding to a ray $\rho\in\Sigma_C(1)$.
Since $\tau$ is the common face of $\sigma_+$ and $\sigma_-$, these sections can be described as either
\begin{equation}
\label{eq_degree_linebdl}
t^{\deg(\gamma^-_v)}=(t^m)^d\cdot t^{\deg(\gamma^+_v)}\qquad \text{or} \qquad t^{\deg(\gamma^-_v)}=(t^m)^{-d}\cdot t^{\deg(\gamma^+_v)},
\end{equation}
where $d$ is the minimal integer satisfying $(t^m)^{d}\cdot t^{\deg(\gamma^+_v)}\in \Spec\CC[\sigma_{-}^\vee\cap \sfM]$ or
$(t^m)^{-d}\cdot t^{\deg(\gamma^+_v)}\in \Spec\CC[\sigma_{-}^\vee\cap \sfM]$,
in which case $\deg(\calL_v |_\ell)=d$ or $-d$.

To show \eqref{eq_degree_linebdl_main} and \eqref{eq_degree_linebdl_main2},
we let $J_+$, $J_+^\prime$ (resp.\ $J_-$, $J_-^\prime$) be the jigsaw pieces for $\tau$ obtained by cutting
$\Hex(\sigma_+)$ (resp.\ $\Hex(\sigma_-)$) in $\RR^2$ along the edge $\fkc_-$ (resp.\ $\fkc_+$).
By Proposition~\ref{prop_jigsaw}, considering certain $\ZZ^2$-translations, we may assume that
$\Hex(\sigma_+)\cap \Hex(\sigma_-)=J_+^\prime=J_-^\prime$ and the face dual to the vertex $0\in Q_0$ is contained in this jigsaw piece.
Let $J_0\coloneqq J_+^\prime=J_-^\prime$.
By Lemma~\ref{lem_property_interior}, any edge contained in the strict interior of $J_0$ does not belong to $\sfP_0\cup \sfP_1\cup \sfP_2\cup \sfP_3$.
Furthermore, since the restrictions of $J_+$ and $J_-$ on the two-torus $\TT$ are identical,
any edge contained in the strict interior of $J_+$ or $J_-$ also does not belong to $\sfP_0\cup \sfP_1\cup \sfP_2\cup \sfP_3$.
\begin{itemize}\itemsep=0pt
\item If the face dual to $v$ in $\Hex(\sigma_+)$ is contained in $J_0$, then the face dual to $v$ in $\Hex(\sigma_-)$ is contained in $J_0$,
and vice versa.
In this situation, both of $\gamma_v^+$ and $\gamma_v^-$ comprise the arrows dual to edges not belonging to $\sfP_0\cup \sfP_1\cup \sfP_2\cup \sfP_3$,
and hence neither $t_{\rho_0}$ nor $t_{\rho_3}$ appears in $t^{\deg(\gamma_v^\pm)}$.
Thus we conclude that $d=0$ by~\eqref{eq_degree_linebdl}.
\item If the face dual to $v$ in $\Hex(\sigma_+)$ is contained in $J_+$, then the face dual to $v$ in $\Hex(\sigma_-)$ is contained in~$J_-$,
and vice versa. Then, we claim that
\begin{itemize}\setlength{\leftskip}{1cm}\itemsep=0pt
\item[(Case 1)] $t_{\rho_3}$ appears in $t^{\deg(\gamma^+_v)}$ with multiplicity one and $t_{\rho_0}$ appears in $t^{\deg(\gamma^-_v)}$ with multiplicity one,
\item[(Case 2)] $t_{\rho_3}^{-1}$ appears in $t^{\deg(\gamma^+_v)}$ with multiplicity one and $t_{\rho_0}^{-1}$ appears in $t^{\deg(\gamma^-_v)}$ with multiplicity one.
\end{itemize}
In fact, in this situation, the weak path $\gamma_v^+$ crosses over $\fkc_-$ in $\Hex(\sigma_+)$ and $\gamma_v^-$ crosses over~$\fkc_+$ in $\Hex(\sigma_-)$.
Let $a_{\fkc_-}$, $a_{\fkc_+}$ be the arrows dual to the edges $\fkc_-$, $\fkc_+$, respectively.
Since $[z_{k_i}]=-[w_{k^\prime_j}]$, we see that
\begin{itemize}\setlength{\leftskip}{1cm}\itemsep=0pt
\item[(Case 1)] the vertex $v$ appears on the right of $z_{k_i}$ and appears on the right of $w_{k_j^\prime}$,
\item[(Case 2)] the vertex $v$ appears on the left of $z_{k_i}$ and appears on the left of $w_{k_j^\prime}$.
\end{itemize}
Since $\fkc_+\in\Zig(z_{k_i})$ and $\fkc_-\in\Zig(w_{k^\prime_j})$ by Lemma~\ref{lem_zigzag_containPM}(1),
we see that
\begin{itemize}\setlength{\leftskip}{1cm}\itemsep=0pt
\item[(Case 1)] $a_{\fkc_-}$ is contained in $\gamma_v^+$ and $a_{\fkc_+}$ is contained in $\gamma_v^-$,
\item[(Case 2)] $a_{\fkc_-}^*$ is contained in $\gamma_v^+$ and $a_{\fkc_+}^*$ is contained in $\gamma_v^-$.
\end{itemize}
Since $\fkc_+\in\sfP_0$ and $\fkc_-\in\sfP_3$ by Lemma~\ref{lem_property_c}, this shows the claim.
It follows from the claim that
$t^{\deg(\gamma^-_v)}=t^m\cdot t^{\deg(\gamma^+_v)}$ for (Case1) and $t^{\deg(\gamma^-_v)}=(t^m)^{-1}\cdot t^{\deg(\gamma^+_v)}$ for (Case2).
\end{itemize}
Since the restrictions of $J_+$, $J_-$ and $\calR_k$ on $\TT$ are identical,
we have \eqref{eq_degree_linebdl_main} and \eqref{eq_degree_linebdl_main2}.

Next, we assume that $\ell=\ell_k$ is not floppable.
Then, $\Delta_{C,k}$ and $\Delta_{C,k+1}$ form a large triangle consisting of two elementary triangles,
and a torus-invariant curve $\ell$ corresponds to a bisector of the large triangle.
Suppose that the lattice points of the triangle formed by $\Delta_{C,k}$ and $\Delta_{C,k+1}$ are $(i-1,0)$, $(j,1)$, $(i,0)$, and $(i+1,0)$.
(Note that the case where such lattice points are $(j-1,1)$, $(i,0)$, $(j,1)$, and $(j+1,1)$ can be shown by a similar argument.)

For $\theta\in C$, let $\rho_0$, $\rho_1$, $\rho_2$, $\rho_3$ be the rays in $\Sigma_C(1)$ corresponding to the lattice points
$(i-1,0)$, $(j,1)$, $(i,0)$, $(i+1,0)$ in $\Delta_C$, respectively.
Let $\sfP_0$, $\sfP_1$, $\sfP_2$, $\sfP_3$ be $\theta$-stable perfect matchings corresponding to $\rho_0$, $\rho_1$, $\rho_2$, $\rho_3$, respectively,
that is,
\[
\sfP_0=\sfP_{(i-1,0)}^\theta, \qquad \sfP_1=\sfP_{(j,1)}^\theta, \qquad \sfP_2=\sfP_{(i,0)}^\theta, \qquad \sfP_3=\sfP_{(i+1,0)}^\theta.
\]
We consider the cones $\sigma_+, \sigma_-\in\Sigma_C(3)$ whose rays are respectively $\rho_0$, $\rho_1$, $\rho_2$ and $\rho_1$, $\rho_2$, $\rho_3$.
Thus, the cone $\tau=\sigma_+\cap\sigma_-\in\Sigma_C(2)$ corresponds to the line segment obtained by connecting $(i,0)$ and $(j,1)$
 (see the right of Figure~\ref{fig_triangulation_notation}).
By these settings, $\Delta_{C,k}$ (resp.\ $\Delta_{C,k+1}$) is obtained as the intersection of the cone $\sigma_+$ (resp.\ $\sigma_-$) and the hyperplane at height one,
and $u_{\permII(k)}=z_{k_i}$, $u_{\permII(k+1)}=z_{k_{i+1}}$ in the terminology of Proposition~\ref{prop_chamber_correspond_zigzag},
thus $[u_{\permII(k)}]= [u_{\permII(k+1)}]$. As before, we divide the arguments into two cases:
\begin{itemize}\setlength{\leftskip}{1cm}\itemsep=0pt
\item[(Case 3)] The case where $z_{k_i}=u_{\permII(k)}<u_{\permII(k+1)}=z_{k_{i+1}}$.
\item[(Case 4)] The case where $z_{k_i}=u_{\permII(k)}>u_{\permII(k+1)}=z_{k_{i+1}}$.
\end{itemize}
We show that we have \eqref{eq_degree_linebdl_main} for (Case~3) and \eqref{eq_degree_linebdl_main2} for (Case~4).
We take a weak path~$\gamma_v^+$ (resp.~$\gamma_v^-$) in the double quiver of $\widetilde{Q^{\sigma_+}}$ (resp.\ $\widetilde{Q^{\sigma_-}}$) for any $v\in Q_0$,
and we have the equation~\eqref{eq_degree_linebdl} by the same argument as above.
We define $\Hex(\sigma_\pm)$, $\fkc_\pm$, $J_\pm$, $J_0$ in the same way as above.
Then any edge contained in the strict interior of $J_0$, $J_+$ or $J_-$ does not belong to $\sfP_0\cup \sfP_1\cup \sfP_2\cup \sfP_3$
by Lemma~\ref{lem_property_interior}.
\begin{itemize}\itemsep=0pt
\item If the face dual to $v$ in $\Hex(\sigma_+)$ is contained in $J_0$, then we see that $d=\deg(\calL_v |_\ell)=0$ by the same argument as above.
\item If the face dual to $v$ in $\Hex(\sigma_+)$ is contained in $J_+$, then
we see that
$t^{\deg(\gamma^-_v)}=t^m\cdot t^{\deg(\gamma^+_v)}$ for (Case~3) and $t^{\deg(\gamma^-_v)}=(t^m)^{-1}\cdot t^{\deg(\gamma^+_v)}$ for (Case~4).
by the argument similar to the above one. The difference is that in this situation
\begin{itemize}\setlength{\leftskip}{1cm}\itemsep=0pt
\item[(Case 3)] the vertex $v$ appears on the right of $z_{k_i}$ and appears on the left of $z_{k_{i+1}}$,
\item[(Case 4)] the vertex $v$ appears on the left of $z_{k_i}$ and appears on the right of $z_{k_{i+1}}$,
\end{itemize}
since $[z_{k_i}]=[z_{k_{i+1}}]$.
Nevertheless, we have the same conclusion since $\fkc_+\in\Zig(z_{k_i})\subset\sfP_0$ and $\fkc_-\in\Zag(z_{k_{i+1}})\subset\sfP_3$
by Lemma~\ref{lem_zigzag_containPM}\,(2).
\end{itemize}
Since the restriction of $J_+$, $J_-$ and $\calR_k$ on $\TT$ are identical,
we have~\eqref{eq_degree_linebdl_main} and~\eqref{eq_degree_linebdl_main2}.

By the above arguments, the hyperplane $L_k=0$, where $L_k$ is either \eqref{eq_degree_linebdl_sum} or \eqref{eq_degree_linebdl_sum2},
would give a wall of $C$ (see Proposition~\ref{prop_equation_wall}).
We here show that $L_k=0$ certainly determine a~wall of~$C$ for $k=1, \dots, n-1$.
Let $M_{\sigma_+}$ be a representative of $\theta$-stable representations corresponding to~$\sigma_+$.
By Proposition~\ref{prop_stablerep_cosupport}, the support of $M_{\sigma_+}$ is identical with the set of arrows of \smash{$\widetilde{Q^{\sigma^+}}$} which is the arrows contained in $\Hex(\sigma_+)$.
By the same argument as above, we see that $\hd(a_{\fkc_-})\in J_+$ for (Cases~1 and~3) and $\tl(a_{\fkc_-})\in J_+$ for (Cases~2 and~4).
Thus, we see that
\begin{itemize}\setlength{\leftskip}{1.2cm}\itemsep=0pt
\item[(Cases 1, 3)] there exists a subrepresentation $N=(N_v)_{v\in Q_0}$ of $M_{\sigma_+}$
with ${\{v\in Q_0\mid \dim N_v\neq 0\}}$ coincides with the set of vertices contained in $J_+$,
\item[(Cases 2, 4)] there exists a subrepresentation $N=(N_v)_{v\in Q_0}$ of $M_{\sigma_+}$
with $\{{v\in Q_0\mid \dim N_v\neq 0}\}$ coincides with the set of vertices contained in $J_0$,
\end{itemize}
since $\Hex(\sigma_+)$ is divided into two parts $J_+$, $J_0$ by the edge $\fkc_-$.
Identifying $J_+$ with $\calR_k$, we see that such a representation $N$ must satisfy
\begin{itemize}\setlength{\leftskip}{1.2cm}\itemsep=0pt
\item[(Cases 1, 3)] $\theta(N)=\sum_{v\in\calR_k}\theta_v>0$,
\item[(Cases 2, 4)] $\theta(N)=\sum_{v\in Q_0{\setminus}\calR_k}\theta_v=-\sum_{v\in\calR_k}\theta_v>0$.
\end{itemize}

Thus, any $\theta\in C$ satisfies $L_k>0$ for any $k=1, \dots, n-1$.
Then we claim that the inequality $L_k>0$ can not be derived from other inequalities of the form $L_s>0$ $(s\neq k)$.
We prove this for the case where $u_{\permII(k)}<u_{\permII(k+1)}$, in which $L_k$ takes the form \eqref{eq_degree_linebdl_sum} and the set of vertices in the region $\calR_k$ is
$\{\permII(k+1), \permII(k+1)+1, \dots, \permII(k)-1\}$.
If $L_k=\sum_{v\in\calR_k}\theta_v>0$ is derived from other inequations, then we need at least one of the inequality $L_s>0$ $(s\neq k)$ to be of the form:
\begin{itemize}\itemsep=0pt
\item $L_s=\theta_{\permII(k+1)}+\theta_{\permII(k+1)+1}+\cdots+\theta_\alpha>0$ for some $\permII(k+1)\le \alpha<\permII(k)-1$,
\item $L_s=-(\theta_\alpha+\cdots+\theta_{\permII(k+1)-1})>0$ for some $1\leq \alpha\leq \permII(k+1)-1$.
\end{itemize}
The former (resp.\ latter) inequality can be obtained from zigzag paths $u_{\permII(s)}$, $u_{\permII(s+1)}$ adjacent in~$\calZ_\permII$ such that
$u_{\permII(s)}<u_{\permII(s+1)}=u_{\permII(k+1)}$ (resp.\ $u_{\permII(k+1)}=u_{\permII(s+1)}<u_{\permII(s)}$).
In both cases, this contradicts the condition $s\neq k$. The proof for the case $u_{\permII(k)}>u_{\permII(k+1)}$ is similar.
In conclusion, the equation $\sum_{v\in\calR_k}\theta_v=0$ determines a wall of $C$.

The assertions $(2)$, $(2)'$, $(3)$, and $(3)'$ follow from the above arguments.
\end{proof}

\begin{Theorem}\label{thm_adjacentchamber_reflection}
Let the notation be as in Setting~{\rm \ref{setting_zigzag_for_u}}.
Suppose that the sequence of zigzag paths corresponding to $C$ is $\calZ_\permII=(u_{\permII(1)}, \dots, u_{\permII(n)})$.
Let $C^\prime\subset \Theta(Q)_\RR$ be the chamber adjacent to $C$ and suppose that $C$ and $C^\prime$ are separated by the wall $W_k$ given in Theorem~{\rm \ref{thm_main_wall}\,(1)}.
Then, the sequence of zigzag paths corresponding to $C^\prime$ is $\calZ_{\permII s_k}$, where $s_k$ is the adjacent transposition swapping $k$ and $k+1$.
\end{Theorem}

\begin{proof}
Let $\theta\in C$ and $\theta^\prime\in C^\prime$.
For the chamber $C^\prime$, there exists a unique sequence $\calZ_{\permII^\prime}$
satisfying the conditions in Propositions~\ref{prop_bPM_zzsequence} and~\ref{prop_chamber_correspond_zigzag} for some $\permII^\prime\in\fkS_n$.

If $W_k$ is a wall of type I, then $\Delta_{C,k}$ and $\Delta_{C,k+1}$ form a parallelogram and a crossing of the wall $W_k$ induces a flop, which corresponds to a flip of the diagonal.
Thus, we have $\sgn(\Delta_{C,k})=\sgn(\Delta_{C^\prime,k+1})$, $\sgn(\Delta_{C,k+1})=\sgn(\Delta_{C^\prime,k})$.
Also, since a flop preserves a toric divisor, we have
\begin{equation}\label{wallcross_PM}
\sfP^\theta_{(i,0)}=\sfP^{\theta^\prime}_{(i,0)}, \qquad \sfP^\theta_{(j,1)}=\sfP^{\theta^\prime}_{(j,1)}
\end{equation}
for any $i=0, \dots, a$, $j=0, \dots, b$.
The sequence $\calZ_{\permII^\prime}$ satisfying the above conditions is $\calZ_{\permII s_k}$.

If $W_k$ is a wall of type $\typeIII$, then $\Delta_{\permII,k}$ and $\Delta_{\permII,k+1}$ form a large triangle, and $\Delta_C=\Delta_{C^\prime}$.
Since a crossing of the wall $W_k$ induce a divisor-to-curve contraction, we have \eqref{wallcross_PM}
except the perfect matching corresponding to the foot of the bisector on the triangle $\Delta_{\permII,k}\cup\Delta_{\permII,k+1}$.
The sequence~$\calZ_{\permII^\prime}$ satisfying the above conditions is $\calZ_{\permII s_k}$.
\end{proof}

Since the adjacent transpositions $s_k$ $(k=1, \dots, n-1)$ generate the symmetric group $\fkS_n$,
Theorems~\ref{thm_main_wall} and \ref{thm_adjacentchamber_reflection} show that any sequence $\calZ_\permII$ corresponds to a certain chamber in $\Theta(Q)_\RR$.
In particular, we have the following.

\begin{Corollary}
\label{cor_chamber_zigzag_correspondence}
Let the notation be as in Setting~{\rm \ref{setting_zigzag_for_u}}.
There exists a one-to-one correspondence between the following sets:
\begin{itemize}\itemsep=0pt
\item[$(a)$] the set of chambers in $\Theta(Q)_\RR$,
\item[$(b)$] the set $\big\{\calZ_\permII=(u_{\permII(1)}, \dots, u_{\permII(n)}) \mid \permII\in\fkS_n\big\}$ of sequences of zigzag paths.
\end{itemize}
Under this correspondence, a sequence $\calZ_\permII$ and the corresponding chamber $C$ satisfy $\sgn(\Delta_C)=\sgn(\calZ_\permII)$.
Furthermore, a wall-crossing in $(a)$ corresponds to the action of an adjacent transposition in $(b)$.
In particular, the chambers in $\Theta(Q)_\RR$ can be identified with the Weyl chambers of type $A_{n-1}$.
\end{Corollary}

By Corollary~\ref{cor_chamber_zigzag_correspondence}, we may write a chamber $C$ as $C_\permII$ when $C$ corresponds to $\calZ_\permII$,
and can define the action of $\fkS_n$ on the set of chambers in $\Theta(Q)_\RR$, which is compatible with the action of $\fkS_n$ on $\{\calZ_\permII \mid \permII\in\fkS_n\}$.

\begin{Remark}
\label{rem_homMMP}
Note that the identification of the chambers in $\Theta(Q)_\RR$ with the Weyl chambers was already shown in \cite[Lemma~6.8]{Wem_MMP}
for any $cA_{n-1}$ singularity.
The description of each chamber in $\Theta(Q)_\RR$ given in Theorem~\ref{thm_main_wall} can also be obtained by the tracking argument of
GIT chambers established in \cite[Section~5]{Wem_MMP}, which uses mutations of maximal modifying modules.
An advantage of our method is that we can obtain a chamber description directly from a given dimer model,
but it should be emphasized that the method in \cite{Wem_MMP} is valid for any cDV singularity.
\end{Remark}

\begin{Example}
\label{ex_compute_chamber_zigzag}
We consider our running example, that is, let $\Gamma$ be the dimer model as in Figure~\ref{ex_dimer1_basic}.
Recall that the zigzag polygon of $\Gamma$ is $\Delta(3,2)$.
Let $u_1, \dots, u_5$ be zigzag paths shown in Figure~\ref{fig_naming_zigzag}, and we fix a total order $u_5<\cdots<u_1$.

Let $\permII=(1452)\in\fkS_5$, and consider the sequence $\calZ_\permII=(u_{\permII(1)}, \dots, u_{\permII(5)})=(u_4,u_1,u_3,u_5,u_2)$ of zigzag paths,
which satisfy $\sgn(\calZ_\permII)=(+1,-1,-1,+1,-1)$.
By Corollary~\ref{cor_chamber_zigzag_correspondence},
there exists a~chamber $C_\permII$ corresponding to $\calZ_\permII$, and it is given as
\begin{equation*}
C_\permII=\{\theta\in\Theta(Q)_\RR \mid \theta_1+\theta_2+\theta_3>0, \, \theta_1+\theta_2<0, \, \theta_3+\theta_4<0, \, \theta_2+\theta_3+\theta_4>0\}
\end{equation*}
by Theorem~\ref{thm_main_wall}.
In fact, since $\calR_1=\calR(u_4,u_1)=\{1,2,3\}$ and $u_4<u_1$, we have the inequality $\theta_1+\theta_2+\theta_3>0$,
and the other inequations can be obtained from other pairs of zigzag paths.
Since the triangulation $\Delta_{C_\permII}$ satisfies $\sgn(\Delta_{C_\permII})=\sgn(\calZ_\permII)=(+1,-1,-1,+1,-1)$,
we are in the situation depicted in Figure~\ref{fig_sgn_triangulation}.
Thus, the projective crepant resolution $\calM_{C_\permII}$ is the smooth toric variety associated to the toric fan induced by the triangulation $\Delta_{C_\permII}$.

Next, considering the action of $s_3\in\fkS_5$ on $\calZ_\permII$ which swaps $u_{\permII(3)}=u_3$ and $u_{\permII(4)}=u_5$,
we have the sequence $\calZ_{\permII s_3}=(u_4,u_1,u_5,u_3,u_2)$.
By Theorem~\ref{thm_adjacentchamber_reflection}, the chamber $C_{\permII s_3}$ corresponding to $\calZ_{\permII s_3}$ is adjacent to $C_\permII$.
Since $\calR_3=\calR(u_3,u_5)=\{3,4\}$ and $[u_3]=-[u_5]$,
the chamber $C_{\permII s_3}$ is separated from $C_\permII$ by the wall $\theta_3+\theta_4=0$, which is of type I (see Theorem~\ref{thm_main_wall}).
Furthermore, we have
\[
C_{\permII s_3}=\{\theta\in\Theta(Q)_\RR \mid \theta_1+\theta_2+\theta_3>0, \, \theta_1+\theta_2+\theta_3+\theta_4<0, \,
\theta_3+\theta_4>0, \, \theta_2>0\}.
\]

We then consider the action of $s_2\in\fkS_5$ on $\calZ_\permII$ which swaps $u_{\permII(2)}=u_1$ and $u_{\permII(3)}=u_3$,
and have the sequence $\calZ_{\permII s_2}=(u_4,u_3,u_1,u_5,u_2)$.
By Theorem~\ref{thm_adjacentchamber_reflection}, the chamber $C_{\permII s_2}$ corresponding to $\calZ_{\permII s_2}$ is adjacent to $C_\permII$.
Since $\calR_2=\calR(u_1,u_3)=\{1,2\}$ and $[u_1]= [u_3]$,
the chamber $C_{\permII s_2}$ is separated from $C_\permII$ by the wall $\theta_1+\theta_2=0$, which is of type $\typeIII$.
Furthermore, we have
\[
C_{\permII s_2}=\{\theta\in\Theta(Q)_\RR \mid \theta_3>0, \, \theta_1+\theta_2>0, \, \theta_1+\theta_2+\theta_3+\theta_4<0, \,
\theta_2+\theta_3+\theta_4>0\}.
\]
\end{Example}

Then we define the action of $\fkS_a\times \fkS_b\subset \fkS_n$ on $\calZ_\permII$ (and hence on $C_\permII$) so that
$\fkS_a$ (resp.\ $\fkS_b$) acts on the subsequence $(z_{k_1},\dots, z_{k_a})$ (resp.\ $(w_{k^\prime_1}, \dots, w_{k^\prime_b})$) of $\calZ_\permII$
as discussed in Proposition~\ref{prop_chamber_correspond_zigzag}.
Since such an action does not change the sign of a sequence of zigzag paths, the sign of the corresponding triangulation is also preserved,
thus we have the following.

\begin{Corollary}
\label{cor_isomclass_chambers}
Let the notation be as above. Then we see that
$\calM_{C_\permII}\cong\calM_{C_{\permII^\prime}}$ if and only if $\permII^\prime=\permII\cdot\sigma$ for some $\sigma \in\fkS_a\times\fkS_b$.
In particular, projective crepant resolutions of $\Spec R_{a,b}$ are in a one-to-one correspondence with the cosets of $\fkS_a\times\fkS_b$ in $\fkS_n$.
\end{Corollary}

\begin{Remark}
For adjacent chambers $C, C^\prime \subset\Theta(Q)_\RR$, the moduli spaces $\calM_C$, $\calM_{C^\prime}$ are not necessarily isomorphic,
but for each wall in $\Theta(Q)_\RR$, there exists a functor $\Xi$ giving rise to a derived equivalence $\rmD^b(\mathsf{coh}\,\calM_C)\simeq \rmD^b(\mathsf{coh}\,\calM_{C^\prime})$ of the adjacent moduli spaces, see \cite[Theorem~11.1]{IU_anycrepant}.
Thus, all projective crepant resolutions of $\Spec R_{a,b}$ are derived equivalent.
We consider a~path~$\gamma$ in $\Theta(Q)_\RR$ starting from a chamber $C$ and terminating at the same chamber.
Then the composite of functors $\Xi$ associated to walls passed by $\gamma$, which is also identified with the product of some transpositions
of $\fkS_n$ via the correspondence in Corollary~\ref{cor_chamber_zigzag_correspondence}, gives an autoequivalence of $\rmD^b(\mathsf{coh}\,\calM_C)$.
In particular, the mixed braid group $B_{a,b}$ acts faithfully on $\rmD^b(\mathsf{coh}\,\calM_C)$, see \cite[Theorem~1]{DS_mixed}.
Here, the mixed braid group $B_{a,b}$ is the subgroup of the braid group $B_n$ on $n$ strands defined as \smash{$\varphi^{-1}(\fkS_a\times\fkS_b)$},
where $\varphi$ is a natural surjection \smash{$B_n\xrightarrow{\varphi} \fkS_n$}.
\end{Remark}

For a chamber $C\subset \Theta(Q)_\RR$, if the projective crepant resolution $\calM_C$ contains a floppable curve~$\ell$ (equivalently~$b\neq 0$),
then there exists a wall of type I corresponding to $\ell$ (see Theorem~\ref{thm_main_wall}).
Since all projective crepant resolutions of $\Spec R_{a,b}$ (triangulations of $\Delta(a,b)$) are connected by repetitions of flops,
collecting all chambers which can be connected to~$C$ by crossings of walls of type~I, we have all projective crepant resolutions of $\Spec R_{a,b}$
as moduli spaces.
Thus, we can identify these chambers and their walls of type~I with the \emph{flop graph} of projective crepant resolutions,
which is a graph whose vertices are projective crepant resolutions and two vertices are connected by an edge if the corresponding two crepant resolutions are connected by a flop at some curve.
If $R_{a,b}$ is isolated (equivalently $a=b=1$), then the closures of such chambers cover~$\Theta(Q)_\RR$.
However, if $a\ge 2$, in which case there exists a wall of type $\typeIII$ in $\Theta(Q)_\RR$, then different chambers would give the same
projective crepant resolution up to isomorphism (cf.\ Corollary~\ref{cor_isomclass_chambers}).
To observe this phenomenon in more detail, we consider the notion of a~GIT region introduced in~\cite{BCS_GIT}.
First, let $C$, $C^\prime$ be adjacent chambers in $\Theta(Q)_\RR$.
If the wall $\overline{C}\cap\overline{C^\prime}$ is of type I, then we delete it from $\Theta(Q)_\RR$.
After deleting all walls of type I, we have the coarse wall-and-chamber structure of $\Theta(Q)_\RR$.
Each component of the coarse wall-and-chamber decomposition of $\Theta(Q)_\RR$ is said to be a \emph{GIT region}.

\begin{Proposition}\label{prop_numberx_gallery}
Let $G$ be a GIT region of $\Theta(Q)_\RR$. Then $G$ contains $\frac{n!}{a!b!}$ chambers of $\Theta(Q)_\RR$ and
any projective crepant resolution of $\Spec R_{a,b}$ can be obtained as the moduli space $\calM_C$ for some $C\subset G$.
In particular, the number of GIT regions in $\Theta(Q)_\RR$ is $a! b!$.
\end{Proposition}

\begin{proof}
As we observed above, if $\calM_C$ contains a floppable curve~$\ell$ for some chamber $C\subset \Theta(Q)_\RR$ (equivalently $b\neq 0$),
then we can obtain the GIT region of $\Theta(Q)_\RR$ containing $C$ and any projective crepant resolution can be realized as the moduli space
associated to a chamber in this GIT region.
If there is a chamber in $\Theta(Q)_\RR$ not contained in the above GIT region, we repeat the same argument to such a chamber.
Then any chamber of $\Theta(Q)_\RR$ is eventually contained in some GIT region.
Since the number of triangulations of $\Delta(a,b)$ is ${n \choose a}=\frac{n!}{a! b!}$, we have the first assertion.
Moreover, since the chamber structure of $\Theta(Q)_\RR$ can be identified with the Weyl chambers of type $A_{n-1}$ (see Corollary~\ref{cor_chamber_zigzag_correspondence}),
the number of chambers is equal to $|\fkS_n|=n!$, and hence the number of GIT regions is $a! b!$.

We note that when $b=0$, a projective crepant resolution is unique up to isomorphism and $n=a$, thus the assertions are trivial.
\end{proof}

\begin{Example}[{the suspended pinch point (cf.\ \cite[Example~12.5]{IU_anycrepant}, \cite[Section~5]{BM_crepant})}]
Let $a=2$, $b=1$.
Using the method in Section~\ref{sec_dimer_ab} for $(i_1,i_2,i_3)=(-1,-1,+1)$ and ${\rm id}\in\fkS_3$,
we have the dimer model $\Gamma=\Gamma_{\rm id}$ shown in the left of Figure~\ref{suspended_pinch_dimer}.
In particular, the zigzag polygon of $\Gamma$ is $\Delta(2,1)$.
We consider the zigzag paths $u_1$, $u_2$, $u_3$ shown in the right of Figure~\ref{suspended_pinch_dimer}.
In particular, the slopes of these zigzag paths are $[u_1]=[u_2]=(0,-1)$, $[u_3]=(0,1)$.
We fix a total order $u_3<u_2<u_1$.

\begin{figure}[!ht]\centering
{\scalebox{0.55}{
\begin{tikzpicture}
\newcommand{\noderad}{0.18cm} 
\newcommand{\edgewidth}{0.05cm} 
\newcommand{\nodewidthw}{0.05cm} 
\newcommand{\nodewidthb}{0.04cm} 
\newcommand{\zigzagwidth}{0.15cm} 
\newcommand{\zigzagcolor}{red} 
\newcommand{\zigzagcolortwo}{blue} 

\node at (0,0) {
\begin{tikzpicture}
\foreach \n/\a/\b in {1/1.5/3, 2/3.5/3} {\coordinate (B\n) at (\a,\b);}
\foreach \n/\a/\b in {1/0.5/1,2/2.5/1} {\coordinate (W\n) at (\a,\b);}
\draw[line width=\edgewidth] (0,0) rectangle (4,4);
\foreach \w/\b in {1/1,2/2} {\draw[line width=\edgewidth] (W\w)--(B\b); };
\foreach \w/\s/\t in {1/0/2,1/0/0,1/1/0,2/2/0,2/3/0} {\draw[line width=\edgewidth] (W\w)--(\s,\t); };
\foreach \b/\s/\t in {1/1/4,1/2/4,2/3/4,2/4/4,2/4/2} {\draw[line width=\edgewidth] (B\b)--(\s,\t); };
\foreach \x in {1,2} {\filldraw [fill=black, line width=\nodewidthb] (B\x) circle [radius=\noderad] ;};
\foreach \x in {1,2} {\filldraw [fill=white, line width=\nodewidthw] (W\x) circle [radius=\noderad] ;};

\foreach \n/\x/\y in {0/0.5/3, 1/2/2, 2/3.5/1}{
\node[blue] (V\n) at (\x,\y) {\LARGE$\n$}; };
\end{tikzpicture}};

\node at (8,0) {
\begin{tikzpicture}
\draw[line width=\edgewidth] (0,0) rectangle (4,4);
\foreach \w/\b in {1/1,2/2} {\draw[line width=\edgewidth] (W\w)--(B\b); };
\foreach \w/\s/\t in {1/0/2,1/0/0,1/1/0,2/2/0,2/3/0} {\draw[line width=\edgewidth] (W\w)--(\s,\t); };
\foreach \b/\s/\t in {1/1/4,1/2/4,2/3/4,2/4/4,2/4/2} {\draw[line width=\edgewidth] (B\b)--(\s,\t); };
\foreach \x in {1,2} {\filldraw [fill=black, line width=\nodewidthb] (B\x) circle [radius=\noderad] ;};
\foreach \x in {1,2} {\filldraw [fill=white, line width=\nodewidthw] (W\x) circle [radius=\noderad] ;};

\foreach \n/\x/\y in {0/0.5/3, 1/2/2, 2/3.5/1}{
\node[blue] (V\n) at (\x,\y) {\LARGE$\n$}; };

\draw[->, line width=\zigzagwidth, rounded corners, color=\zigzagcolor] (1,4)--(B1)--(W1)--(1,0);
\draw[->, line width=\zigzagwidth, rounded corners, color=\zigzagcolor] (3,4)--(B2)--(W2)--(3,0);
\draw[->, line width=\zigzagwidth, rounded corners, color=\zigzagcolortwo] (0,0)--(W1)--(0,2);
\draw[->, line width=\zigzagwidth, rounded corners, color=\zigzagcolortwo] (4,2)--(B2)--(4,4);

\node[red] at (1,-0.5) {\LARGE$u_1$}; \node[red] at (3,-0.5) {\LARGE$u_2$};
\node[blue] at (4,4.5) {\LARGE$u_3$};
\end{tikzpicture}};

\end{tikzpicture}}
}
\vspace{-2mm}

\caption{The dimer model $\Gamma$ whose zigzag polygon is $\Delta(2,1)$ (left),
the zigzag paths $u_1$, $u_2$, $u_3$ on $\Gamma$ (right).}\label{suspended_pinch_dimer}
\end{figure}

Let $Q$ be the quiver associated to $\Gamma$. Then the space of stability parameters is
\[
\Theta(Q)_\RR=\{\theta=(\theta_0, \theta_1, \theta_2) \mid \theta_0+\theta_1+\theta_2=0\}.
\]
By Theorem~\ref{thm_main_wall} and Corollary~\ref{cor_chamber_zigzag_correspondence}, we have the wall-and-chamber decomposition of $\Theta(Q)_\RR$ as shown in Figure~\ref{suspended_pinch_chamber}.
For example, the sequence $(u_3, u_2, u_1)$ corresponds to the chamber $C$ described as
$
C=\{\theta\in\Theta(Q)_\RR \mid \theta_1>0,\, \theta_2>0\}$,
and the crepant resolution $\calM_C$ is isomorphic to the toric variety associated to the triangulation of $\Delta(2,1)$
described in the first quadrant of Figure~\ref{suspended_pinch_chamber}.
A crossing of the wall $\theta_2=0$ of $C$ corresponds to a swapping of $u_3$ and $u_2$.
Also, a crossing of the wall $\theta_1=0$ of $C$ corresponds to a swapping of $u_2$ and $u_1$.

In Figure~\ref{suspended_pinch_chamber}, the equations $\theta_2=0$ and $\theta_1+\theta_2=0$ are walls of type I,
and $\theta_1=0$ is a wall of type $\typeIII$.
Thus, three chambers satisfying $\theta_1>0$ are in the same GIT region, and also the ones $\theta_1<0$ are in the same GIT region.
Each GIT region induces the flop graph of projective crepant resolutions of~$\Spec R_{2,1}$.

\begin{figure}[!ht]\centering
{\scalebox{0.8}{
\begin{tikzpicture}
\newcommand{\edgewidth}{0.035cm}
\newcommand{\noderad}{0.08} 
\newcommand{\circlerad}{4}

\draw[->, line width=\edgewidth] (-\circlerad,0)--(\circlerad,0); \draw[->, line width=\edgewidth] (0,-\circlerad)--(0,\circlerad);
\coordinate (wall1) at (135:\circlerad+0.5); \coordinate (wall2) at (315:\circlerad+0.5); \draw[line width=\edgewidth] (wall1)--(wall2);
\node at (\circlerad+0.5,0) {\large$\theta_1$}; \node at (0,\circlerad+0.5) {\large$\theta_2$};

\node at (45:\circlerad-1.8)
{\scalebox{0.8}{
\begin{tikzpicture}
\CrepantResolb
\end{tikzpicture}}};
\node at (55:\circlerad-0.7) {\large$(u_3, u_2, u_1)$};

\node at (112:\circlerad-1.2)
{\scalebox{0.8}{
\begin{tikzpicture}
\CrepantResolb
\end{tikzpicture}}};
\node at (112:\circlerad) {\large$(u_3, u_1, u_2)$};

\node at (160:\circlerad-1.2)
{\scalebox{0.8}{
\begin{tikzpicture}
\CrepantResola
\end{tikzpicture}}};
\node at (150:\circlerad) {\large$(u_1, u_3, u_2)$};

\node at (225:\circlerad-1.8)
{\scalebox{0.8}{
\begin{tikzpicture}
\CrepantResolc
\end{tikzpicture}}};
\node at (225:\circlerad-0.3) {\large$(u_1, u_2, u_3)$};

\node at (292:\circlerad-1.2)
{\scalebox{0.8}{
\begin{tikzpicture}
\CrepantResolc
\end{tikzpicture}}};
\node at (292:\circlerad) {\large$(u_2, u_1, u_3)$};

\node at (340:\circlerad-1.2)
{\scalebox{0.8}{
\begin{tikzpicture}
\CrepantResola
\end{tikzpicture}}};
\node at (330:\circlerad) {\large$(u_2, u_3, u_1)$};
\end{tikzpicture}
}}
\vspace{-2mm}

\caption{The wall-and-chamber structure of $\Theta(Q)_\RR$}
\label{suspended_pinch_chamber}
\end{figure}
\end{Example}

\begin{Remark}
As we discussed in Section~\ref{sec_dimer_ab}, the consistent dimer models in $\{\Gamma_\perm \mid \perm\in\fkS_n\}$ are associated to
the toric diagram of $R_{a,b}$, and they are transformed into one another by the actions of adjacent transpositions.
For any consistent dimer model $\Gamma_\perm$ with $\perm\in\fkS_n$, we have the same results shown in this section.
In particular, the chambers in $\Theta(Q_\perm)_\RR$ are identified with the Weyl chambers of type $A_{n-1}$.
There are differences in projective crepant resolutions associated to some chambers, which correspond to choices of simple roots
in the theory of root systems.
For example, let $a=2$, $b=1$, and $s_2=(2\, 3)\in\fkS_3$.
Then we have the consistent dimer model $\Gamma_{s_2}$ as shown in the left of Figure~\ref{suspended_pinch_dimer2}.
According to convention~\eqref{eq_label_z_w2}, we label the zigzag paths of $\Gamma_{s_2}$ as shown in the right of Figure~\ref{suspended_pinch_dimer2}.

\begin{figure}[!ht]\centering
{\scalebox{0.55}{
\begin{tikzpicture}
\newcommand{\noderad}{0.18cm} 
\newcommand{\edgewidth}{0.05cm} 
\newcommand{\nodewidthw}{0.05cm} 
\newcommand{\nodewidthb}{0.04cm} 
\newcommand{\zigzagwidth}{0.15cm} 
\newcommand{\zigzagcolor}{red} 
\newcommand{\zigzagcolortwo}{blue} 

\node at (0,0) {
\begin{tikzpicture}
\foreach \n/\a/\b in {1/1.5/3, 2/3.5/3} {\coordinate (B\n) at (\a,\b);}
\foreach \n/\a/\b in {1/0.5/1,2/2.5/1} {\coordinate (W\n) at (\a,\b);}
\draw[line width=\edgewidth] (0,0) rectangle (4,4);
\foreach \w/\b in {1/1,2/2} {\draw[line width=\edgewidth] (W\w)--(B\b); };
\foreach \w/\s/\t in {1/0/2,1/0/0,1/1/0,2/2/0,2/3/0} {\draw[line width=\edgewidth] (W\w)--(\s,\t); };
\foreach \b/\s/\t in {1/1/4,1/2/4,2/3/4,2/4/4,2/4/2} {\draw[line width=\edgewidth] (B\b)--(\s,\t); };
\foreach \x in {1,2} {\filldraw [fill=black, line width=\nodewidthb] (B\x) circle [radius=\noderad] ;};
\foreach \x in {1,2} {\filldraw [fill=white, line width=\nodewidthw] (W\x) circle [radius=\noderad] ;};

\foreach \n/\x/\y in {2/0.5/3, 0/2/2, 1/3.5/1}{
\node[blue] (V\n) at (\x,\y) {\LARGE$\n$}; };
\end{tikzpicture}};

\node at (8,0) {
\begin{tikzpicture}
\draw[line width=\edgewidth] (0,0) rectangle (4,4);
\foreach \w/\b in {1/1,2/2} {\draw[line width=\edgewidth] (W\w)--(B\b); };
\foreach \w/\s/\t in {1/0/2,1/0/0,1/1/0,2/2/0,2/3/0} {\draw[line width=\edgewidth] (W\w)--(\s,\t); };
\foreach \b/\s/\t in {1/1/4,1/2/4,2/3/4,2/4/4,2/4/2} {\draw[line width=\edgewidth] (B\b)--(\s,\t); };
\foreach \x in {1,2} {\filldraw [fill=black, line width=\nodewidthb] (B\x) circle [radius=\noderad] ;};
\foreach \x in {1,2} {\filldraw [fill=white, line width=\nodewidthw] (W\x) circle [radius=\noderad] ;};

\foreach \n/\x/\y in {2/0.5/3, 0/2/2, 1/3.5/1}{
\node[blue] (V\n) at (\x,\y) {\LARGE$\n$}; };

\draw[->, line width=\zigzagwidth, rounded corners, color=\zigzagcolor] (1,4)--(B1)--(W1)--(1,0);
\draw[->, line width=\zigzagwidth, rounded corners, color=\zigzagcolor] (3,4)--(B2)--(W2)--(3,0);
\draw[->, line width=\zigzagwidth, rounded corners, color=\zigzagcolortwo] (0,0)--(W1)--(0,2);
\draw[->, line width=\zigzagwidth, rounded corners, color=\zigzagcolortwo] (4,2)--(B2)--(4,4);

\node[red] at (1,-0.5) {\LARGE$u_3^\prime$}; \node[red] at (3,-0.5) {\LARGE$u_1^\prime$};
\node[blue] at (4,4.5) {\LARGE$u_2^\prime$};
\end{tikzpicture}};

\end{tikzpicture}}
}
\vspace{-2mm}

\caption{The dimer model $\Gamma_{s_2}$ whose zigzag polygon is~$\Delta(2,1)$ (left),
the zigzag paths $u_1^\prime$, $u_2^\prime$, $u_3^\prime$ on~$\Gamma_{s_2}$ whose slopes are either $(0,-1)$ or $(1,0)$ (right).}\label{suspended_pinch_dimer2}
\end{figure}
Then, by Theorem~\ref{thm_main_wall} and Corollary~\ref{cor_chamber_zigzag_correspondence}, we have the wall-and-chamber structure of $\Theta(Q_{s_2})_\RR$ as shown in Figure~\ref{suspended_pinch_chamber2}.
Note that the wall-and-chamber structure is the same as the one in Figure~\ref{suspended_pinch_chamber}, but the projective crepant resolution associated to each chamber is different.

\begin{figure}[!ht]\centering
{\scalebox{0.8}{
\begin{tikzpicture}
\newcommand{\edgewidth}{0.035cm}
\newcommand{\noderad}{0.08} 
\newcommand{\circlerad}{4}

\draw[->, line width=\edgewidth] (-\circlerad,0)--(\circlerad,0); \draw[->, line width=\edgewidth] (0,-\circlerad)--(0,\circlerad);
\coordinate (wall1) at (135:\circlerad+0.5); \coordinate (wall2) at (315:\circlerad+0.5); \draw[line width=\edgewidth] (wall1)--(wall2);
\node at (\circlerad+0.5,0) {\large$\theta_1$}; \node at (0,\circlerad+0.5) {\large$\theta_2$};

\node at (45:\circlerad-1.8)
{\scalebox{0.8}{
\begin{tikzpicture}
\CrepantResola
\end{tikzpicture}}};
\node at (55:\circlerad-0.7) {\large$(u_3^\prime, u_2^\prime, u_1^\prime)$};

\node at (112:\circlerad-1.2)
{\scalebox{0.8}{
\begin{tikzpicture}
\CrepantResolc
\end{tikzpicture}}};
\node at (112:\circlerad) {\large$(u_3^\prime, u_1^\prime, u_2^\prime)$};

\node at (160:\circlerad-1.2)
{\scalebox{0.8}{
\begin{tikzpicture}
\CrepantResolc
\end{tikzpicture}}};
\node at (150:\circlerad) {\large$(u_1^\prime, u_3^\prime, u_2^\prime)$};

\node at (225:\circlerad-1.8)
{\scalebox{0.8}{
\begin{tikzpicture}
\CrepantResola
\end{tikzpicture}}};
\node at (225:\circlerad-0.3) {\large$(u_1^\prime, u_2^\prime, u_3^\prime)$};

\node at (292:\circlerad-1.2)
{\scalebox{0.8}{
\begin{tikzpicture}
\CrepantResolb
\end{tikzpicture}}};
\node at (292:\circlerad) {\large$(u_2^\prime, u_1^\prime, u_3^\prime)$};

\node at (340:\circlerad-1.2)
{\scalebox{0.8}{
\begin{tikzpicture}
\CrepantResolb
\end{tikzpicture}}};
\node at (330:\circlerad) {\large$(u_2^\prime, u_3^\prime, u_1^\prime)$};
\end{tikzpicture}
}}\vspace{-2mm}

\caption{The wall-and-chamber structure of $\Theta(Q_{s_2})_\RR$.}\label{suspended_pinch_chamber2}
\end{figure}
\end{Remark}

\section{Variations of stable representations under wall crossings}\label{sec_variation_representation}

We keep Settings~\ref{setting_zigzag_for_a+b} and \ref{setting_zigzag_for_u}.
In the previous section, we showed the correspondence between the chambers in $\Theta(Q)_\RR$ and the set $\{\calZ_\permII \mid \permII\in\fkS_n\}$
of sequences of zigzag paths (see Corollary~\ref{cor_chamber_zigzag_correspondence}).
Also, using this correspondence, we can see the variation of projective crepant resolutions of $\Spec R_{a,b}$.
In this section, we observe the variations of torus orbits in projective crepant resolutions under wall crossings.
First, we recall that for a chamber $C_\permII$ the torus orbits in the projective crepant resolution $\calM_{C_\permII}$ of $\Spec R_{a,b}$
can be determined by the set $\PM_{C_\permII}(\Gamma)$ of $\theta$-stable perfect matchings of $\Gamma$
for some (and hence any) $\theta\in C_\permII$.
Precisely, for an $r$-dimensional cone $\sigma\in\Sigma_{C_\permII}(r)$ $(r=1,2,3)$,
the cosupport of a $\theta$-stable representation $M_\sigma$, which corresponds to a $(3-r)$-dimensional torus orbit $\calO_\sigma\subset\calM_{C_\permII}$,
consists of the arrows dual to $\bigcup_{i=1}^r\sfP_i$, where $\sfP_1, \dots, \sfP_r$ are perfect matchings in $\PM_{C_\permII}(\Gamma)$ corresponding to the rays of $\sigma\in\Sigma_{C_\permII}(r)$, see Proposition~\ref{prop_stablerep_cosupport}.
Thus, in the following, we will observe the variations of stable perfect matchings to understand the variations of torus orbits under wall crossings.

\begin{Setting}
\label{setting_stablerep}
Let $C_\permII$ be a chamber in $\Theta(Q)_\RR$ corresponding to the sequence $\calZ_\permII$ for some $\permII\in\fkS_n$ (see Corollary~\ref{cor_chamber_zigzag_correspondence}).
In the following, we use the notation $\Sigma_{\permII}\coloneqq\Sigma_{C_\permII}$, $\Delta_\permII\coloneqq\Delta_{C_\permII}$,
and $\PM_\permII(\Gamma)\coloneqq\PM_{C_\permII}(\Gamma)$.

We consider the set of elementary triangles $\{\Delta_{\permII,k}\}_{k=1}^n$ in the triangulation $\Delta_{\permII}$ as in Setting~\ref{setting_zigzag_for_u}.
For any $k=1, \dots, n-1$, we denote by $\sigma_{\permII, k}$ the three-dimensional cone in the toric fan $\Sigma_{\permII}$ corresponding to $\Delta_{\permII,k}$,
and denote by $\tau_{\permII, k}$ the two-dimensional cone in $\Sigma_{\permII}$ corresponding to the line segment $\Delta_{\permII,k}\cap\Delta_{\permII,k+1}$.
\end{Setting}

Suppose that $W_k=\bigl\{\theta\in\Theta(Q)_\RR \mid \sum_{v\in\calR_k}\theta_v=0\bigr\}$ is a wall of $C_\permII$ (see Theorem~\ref{thm_main_wall})
and the chamber adjacent to $C_\permII$ by the wall $W_k$ is $C_{\permII s_k}$.
For $\theta\in C_\permII$, let $\sfP^\theta_q$ be the $\theta$-stable perfect matching corresponding to a lattice point $q\in\Delta(a,b)$.
By Theorem~\ref{thm_adjacentchamber_reflection} (and its proof),
we can observe the variations of stable perfect matchings as in Propositions~\ref{variation_stablePM} and \ref{variation_stablePM2} below.

\begin{Proposition}
\label{variation_stablePM}
Let the notation be the same as above.
If the wall $W_k$ is of type {\rm I}, then we see that $\sfP^\theta_q=\sfP^{\theta^\prime}_q$ for any $\theta\in C_\permII$, $\theta^\prime\in C_{\permII s_k}$,
and any lattice point $q\in\Delta(a,b)$.
In particular, we have $\PM_\permII(\Gamma)=\PM_{\permII s_k}(\Gamma)$.
\end{Proposition}

\begin{Proposition}
\label{variation_stablePM2}
Let the notation be the same as above.
Suppose that the wall $W_k$ is of type $\typeIII$, in which case $[u_{\permII(k)}]=[u_{\permII(k+1)}]$
and the elementary triangles $\Delta_{\permII,k}$ and $\Delta_{\permII,k+1}$ form a large triangle.
Let $m\in\Delta(a,b)$ be the foot of the bisector on the triangle $\Delta_{\permII,k}\cup\Delta_{\permII,k+1}$.
\begin{center}
{\scalebox{1}{
\begin{tikzpicture}
\newcommand{\edgewidth}{0.035cm}
\newcommand{\noderad}{0.08} 

\foreach \n/\a/\b in {00/-0.2/0,10/1.25/0, 11/1.25/1.25, 20/2.7/0} {
\coordinate (V\n) at (\a,\b);
};
\foreach \s/\t in {00/10,10/20,20/11,11/00} {\draw [line width=\edgewidth] (V\s)--(V\t) ;};
\draw [line width=\edgewidth] (V10)--(V11) ;
\foreach \n in {00,10,11,20}{\draw [fill=black] (V\n) circle [radius=\noderad] ; };

\newcommand{\nodeshift}{0.3cm}
\node[xshift=0cm,yshift=-\nodeshift] at (V10) {\small$m$};
\node at (0.7,0.25) {\small$\Delta_{\permII,k}$}; \node at (1.85,0.25) {\small$\Delta_{\permII,k+1}$};
\end{tikzpicture}
}}
\end{center}
Then, for any $\theta\in C_\permII$, $\theta^\prime\in C_{\permII s_k}$, we see that \smash{$\sfP^\theta_q=\sfP^{\theta^\prime}_q$}
if $q\in\Delta(a,b)$ is a lattice point with $q\neq m$, and \smash{$\sfP^{\theta^\prime}_m=\BS_{u_{\permII(k+1)}}\BS_{u_{\permII(k)}}\bigl(\sfP^\theta_m\bigr)$},
where $\BS_{u_{\permII(k)}}$ and $\BS_{u_{\permII(k+1)}}$ are the zigzag switchings as defined in Section~{\rm \ref{subsec_switchings}}.
In particular, we have
\[
\PM_{\permII {s_k}}(\Gamma)=\bigl(\PM_{\permII}(\Gamma){\setminus} \bigl\{\sfP_m^\theta\bigr\}\bigr)\cup \bigl\{\BS_{u_{\permII(k+1)}}\BS_{u_{\permII(k)}}\bigl(\sfP^\theta_m\bigr)\bigr\}.
\]
\end{Proposition}

By using Propositions~\ref{variation_stablePM} and \ref{variation_stablePM2}, we can track the variations of
stable perfect matchings under wall-crossings.
Thus, we can also track the variations of stable representations corresponding to torus orbits in projective crepant resolutions
by considering their support determined by stable perfect matchings.
In the following, we study stable representations corresponding to exceptional curves in a projective crepant resolution
and observe their variations under wall-crossings.

We here recall that for any zigzag paths $u_i$ as in Setting~\ref{setting_zigzag_for_u}
both $\Zig(u_i)$ and $\Zag(u_i)$ consist of a single edge (see Lemma~\ref{lem_zigzag_nointersection}).
We denote the arrows of $Q$ dual to edges in $\Zig(u_i)$ and $\Zag(u_i)$ by $\aZig(u_i)$ and $\aZag(u_i)$, respectively.

\begin{Lemma}
\label{lem_zigzag_in_support}
Let the notation be as in Setting~{\rm \ref{setting_stablerep}}.
For $\theta\in C_\permII$ and a two-dimensional cone $\tau=\tau_{\permII, k}\in\Sigma_\permII(2)$, we consider the $\theta$-stable representation $M_\tau$.
Then, for any zigzag path $u_i$ $(i=1, \dots, n)$, either $\aZig(u_i)$ or $\aZag(u_i)$ is contained in $\Supp M_\tau$.
\end{Lemma}

\begin{proof}
By Proposition~\ref{prop_stablerep_cosupport}, the cosupport of $M_\tau$ consists of the arrows dual to \smash{$\sfP^\theta_{(i,0)}\cup\sfP^\theta_{(j,1)}$}
for some $i=0, \dots, a$ and $j=0, \dots, b$.
Thus, the assertion follows from the description of $\theta$-stable perfect matchings as in~\eqref{allzigzag_in_cornerPM}, \eqref{eq_bPM_from_zz1}, and~\eqref{eq_bPM_from_zz2}.
\end{proof}

We then consider a subset $V$ of the arrow set $Q_1$ such that $V$ contains either $\aZig(u_i)$ or $\aZag(u_i)$.
In view of Lemma~\ref{lem_zigzag_in_support}, for a zigzag path $u_i$ $(i=1, \dots, n)$ we define
\[
\BS_{u_i}(V)=
\begin{cases}
\big(V{\setminus}\{\aZig(u_i)\}\big) \cup \{\aZag(u_i)\}&\text{if $\aZig(u_i)\in V$,} \\
\big(V{\setminus}\{\aZag(u_i)\}\big) \cup \{\aZig(u_i)\}&\text{if $\aZag(u_i)\in V$.}
\end{cases}
\]
Note that this can be considered as a variant of the zigzag switching, thus we use the same notation.

\begin{Theorem}\label{thm_variation_support}
Let the notation be as in Setting~{\rm \ref{setting_stablerep}}.
For $k=1, \dots, n-1$, let $\ell_k$ be an exceptional curves in $\calM_{C_\permII}$ which is the torus orbit $\calO_{\tau}$
for the two-dimensional cone $\tau=\tau_{\permII, k}\in\Sigma_\permII(2)$.
We consider the chamber $C_{\permII s_k}$ separated from $C_\permII$ by the wall $W_k$ determined by $\sum_{v\in\calR_k}\theta_v=0$.
Then, for the cone $\tau^\prime\coloneqq\tau_{\permII s_k, k}\in\Sigma_{\permII s_k}(2)$, we have that
\[
\Supp M_{\tau^\prime}=\BS_{u_{\permII(k+1)}}\BS_{u_{\permII(k)}}(\Supp M_\tau).
\]
\end{Theorem}

\begin{proof}
Let $\theta\in C_\permII$ and $\theta^\prime\in C_{\permII s_k}$.

First, we assume that the wall $W_k$ is of type I, in which $\ell_k$ is floppable.
We consider the parallelogram $\Delta_{\permII,k}\cup\Delta_{\permII,k+1}$ in the triangulation $\Delta_{\permII}$.
As in the proof of Theorem~\ref{thm_main_wall}, we suppose that the vertices of the parallelogram are $(i-1, 0)$, $(i,0)$, $(j-1,1)$ and $(j,1)$.
Also, we suppose that the diagonal connecting $(j-1, 1)$ and $(i,0)$ corresponds to $\tau$.
Then the cone~$\tau^\prime$ corresponds to the diagonal connecting $(i-1, 0)$ and $(j,1)$ in the triangulation~$\Delta_{\permII s_k}$.
By Proposition~\ref{variation_stablePM}, we have
\[
\sfP^\theta_{(i-1,0)}=\sfP^{\theta^\prime}_{(i-1,0)}, \qquad \sfP^\theta_{(j-1,1)}=\sfP^{\theta^\prime}_{(j-1,1)}, \qquad
\sfP^\theta_{(i,0)}=\sfP^{\theta^\prime}_{(i,0)}, \qquad \sfP^\theta_{(j,1)}=\sfP^{\theta^\prime}_{(j,1)}.
\]
By Proposition~\ref{prop_bPM_zzsequence}, we see that
\begin{alignat*}{3}
& \Zag(u_{\permII(k)})\subset \sfP^\theta_{(i,0)}, \qquad&& \Zig(u_{\permII(k+1)})\subset \sfP^\theta_{(j-1,1)}, &\\
& \Zig(u_{\permII(k)})\subset \sfP^{\theta^\prime}_{(i-1,0)}, \qquad&& \Zag(u_{\permII(k+1)})\subset \sfP^{\theta^\prime}_{(j,1)}&
\end{alignat*}
which means that $\aZig(u_{\permII(k)}), \aZag(u_{\permII(k+1)})\in \Supp M_\tau$ and
$\aZig(u_{\permII(k+1)}), \aZag(u_{\permII(k)})\in \Supp M_{\tau^\prime}$.
Since $\Supp M_\tau$ and $\Supp M_{\tau^\prime}$ are the same except these arrows, we have the assertion.

Next, we assume that the wall $W_k$ is of type $\typeIII$, in which case
we consider the large triangle $\Delta_{\permII,k}\cup\Delta_{\permII,k+1}$ in the triangulation $\Delta_{\permII}$.
As in the proof of Theorem~\ref{thm_main_wall}, we suppose that the vertices of the large triangle are $(i-1, 0)$, $(i,0)$, $(i+1,0)$ and $(j,1)$.
Then, the triangulations~$\Delta_\permII$ and $\Delta_{\permII s_k}$ are the same, and
the cones $\tau$ and $\tau^\prime$ correspond to the line segment connecting $(i, 0)$ and $(j,1)$.
By Proposition~\ref{variation_stablePM2}, we have
$
\sfP^{\theta^\prime}_{(i,0)}=\BS_{u_{\permII(k+1)}}\BS_{u_{\permII(k)}}\bigl(\sfP^\theta_{(i,0)}\bigr)$,
and the other stable perfect matchings are preserved. Thus we have the assertion.
\end{proof}

By Theorem~\ref{thm_variation_support}, we see that the support quivers $Q^\tau$ and $Q^{\tau^\prime}$ are
transformed into each other by ``reflections'' of quivers at vertices contained in $\calR_k$.
More precisely, let $(Q^\tau, \calJ_{Q^\tau})$ be the quiver with relations such that $(Q^\tau)_0=Q_0$, $(Q^\tau)_1=\Supp M_\tau$,
and $\calJ_{Q^\tau}$ is the restriction of the relations $\calJ_Q$ on $Q^\tau$.
Note that $\calJ_{Q^\tau}$ can be described as
\[
\calJ_{Q^\tau}=\bigl\{\gamma_a^+-\gamma_a^- \mid a\in \sfP^\theta_{(i,0)}\cap\sfP^\theta_{(j,1)} \bigr\},
\]
but the claim shown in the proof of Lemma~\ref{lem_property_c} implies that $\calJ_{Q^\tau}=\varnothing$.
Thus, we consider the quiver $Q^\tau$ with no relations.
The quiver $Q^{\tau^\prime}$ is defined in a similar way and Theorem~\ref{thm_variation_support} shows that
$Q^{\tau^\prime}$ coincides with the quiver obtained by reversing all arrows of $Q^\tau$ incident to a vertex in $\calR_k$.
Note that if $\hd(a), \tl(a)\in\calR_k$ for an arrow $a\in (Q^\tau)_1$, then we reverse this arrow twice, and hence the orientation is restored.

\section[Notes on the wall-and-chamber structure for type cD\_4]{Notes on the wall-and-chamber structure for type $\boldsymbol{cD_4}$}
\label{sec_D4}

In this section, we focus on the toric compound Du Val singularity of type $cD_4$
\[
(cD_4)\colon \ R\coloneqq \CC[x,y,z,w]/\bigl(xyz-w^2\bigr)
\]
which can be realized as the toric ring
whose toric diagram is the triangle shown in the right of Figure~\ref{fig_hypersurf_lattice}, see Example~\ref{ex_toric2}.
Note that $R$ is also isomorphic to the invariant subring $S^G$ of $S\coloneqq \CC[X, Y, Z]$ under the action of
$G\coloneqq \ZZ/2\ZZ\times \ZZ/2\ZZ\cong \langle\mathrm{diag}(-1,-1,1)\rangle\times\langle\mathrm{diag}(1,-1,-1)\rangle\subset\SL(3,\CC)$.
Since $R$ is a three-dimensional Gorenstein toric ring, we can apply results in Sections~\mbox{\ref{sec_pre_dimer}--\ref{section_boundaryPM}} to~$R$.
First, since the toric diagram $\Delta_R$ is a triangle, a consistent dimer model $\Gamma$ satisfying $\Delta_R=\Delta_\Gamma$
is a hexagonal dimer model (i.e., any face of $\Gamma$ is a hexagon and any node of $\Gamma$ is 3\nobreakdash-valent,
which means that $\Gamma$ is homotopy-equivalent to a dimer model whose faces are all regular hexagons), see \cite{IN_steady,UY_McKay}.
Precisely, $\Gamma$ is described as in Figure~\ref{fig_dimer_cD4}.

\newcommand{\dimerDtype}{
\foreach \n/\a in {1/0, 2/60, 3/120, 4/180, 5/240, 6/300}
{\coordinate (Ha1\n) at (\a:\boundaryrad); \path (Ha1\n) ++(0,{2*cos(30)}) coordinate (Ha2\n); \path (Ha2\n) ++(0,{2*cos(30)}) coordinate (Ha3\n); \path (Ha3\n) ++(0,{2*cos(30)}) coordinate (Ha4\n);}
\foreach \m in {1,2,3,4} {\draw[line width=\edgewidth] (Ha\m1)--(Ha\m2)--(Ha\m3)--(Ha\m4)--(Ha\m5)--(Ha\m6)--(Ha\m1);}
\foreach \m/\n in {1/1,1/3, 1/5, 2/1, 2/3, 2/5, 3/1, 3/3, 3/5, 4/1, 4/3, 4/5} {\filldraw [fill=black, line width=\nodewidthb] (Ha\m\n) circle [radius=\noderad] ;};
\foreach \m/\n in {1/2,1/4, 1/6, 2/2, 2/4, 2/6, 3/2, 3/4, 3/6, 4/2, 4/4, 4/6} {\filldraw [fill=white, line width=\nodewidthw] (Ha\m\n) circle [radius=\noderad] ;};
\foreach \x/\y in {0/{2*cos(30)}, 0/{6*cos(30)}} {\node[blue] at (\x,{\y}) {\large$0$};}
\foreach \x/\y in {0/0, 0/{4*cos(30)}} {\node[blue] at (\x,{\y}) {\large$1$};}

\foreach \n/\a in {1/0, 2/60, 3/120, 4/180, 5/240, 6/300}
{\path (Ha1\n) ++(1.5,{1*cos(30)}) coordinate (Hb1\n); \path (Hb1\n) ++(0,{2*cos(30)}) coordinate (Hb2\n); \path (Hb2\n) ++(0,{2*cos(30)}) coordinate (Hb3\n); }
\foreach \m in {1,2,3} {\draw[line width=\edgewidth] (Hb\m1)--(Hb\m2)--(Hb\m3)--(Hb\m4)--(Hb\m5)--(Hb\m6)--(Hb\m1);}
\foreach \m/\n in {1/1,1/3, 1/5, 2/1, 2/3, 2/5, 3/1, 3/3, 3/5} {\filldraw [fill=black, line width=\nodewidthb] (Hb\m\n) circle [radius=\noderad] ;};
\foreach \m/\n in {1/2,1/4, 1/6, 2/2, 2/4, 2/6, 3/2, 3/4, 3/6} {\filldraw [fill=white, line width=\nodewidthw] (Hb\m\n) circle [radius=\noderad] ;};
\foreach \x/\y in {1.5/{3*cos(30)}} {\node[blue] at (\x,{\y}) {\large$2$};}
\foreach \x/\y in {1.5/{1*cos(30)}, 1.5/{5*cos(30)}} {\node[blue] at (\x,{\y}) {\large$3$};}

\foreach \n/\a in {1/0, 2/60, 3/120, 4/180, 5/240, 6/300}
{\path (Hb1\n) ++(1.5,{-1*cos(30)}) coordinate (Hc1\n); \path (Hc1\n) ++(0,{2*cos(30)}) coordinate (Hc2\n); \path (Hc2\n) ++(0,{2*cos(30)}) coordinate (Hc3\n); \path (Hc3\n) ++(0,{2*cos(30)}) coordinate (Hc4\n);}
\foreach \m in {1,2,3,4} {\draw[line width=\edgewidth] (Hc\m1)--(Hc\m2)--(Hc\m3)--(Hc\m4)--(Hc\m5)--(Hc\m6)--(Hc\m1);}
\foreach \m/\n in {1/1,1/3, 1/5, 2/1, 2/3, 2/5, 3/1, 3/3, 3/5, 4/1, 4/3, 4/5} {\filldraw [fill=black, line width=\nodewidthb] (Hc\m\n) circle [radius=\noderad] ;};
\foreach \m/\n in {1/2,1/4, 1/6, 2/2, 2/4, 2/6, 3/2, 3/4, 3/6, 4/2, 4/4, 4/6} {\filldraw [fill=white, line width=\nodewidthw] (Hc\m\n) circle [radius=\noderad] ;};
\foreach \x/\y in {3/{2*cos(30)}, 3/{6*cos(30)}} {\node[blue] at (\x,{\y}) {\large$1$};}
\foreach \x/\y in {3/0, 3/{4*cos(30)}} {\node[blue] at (\x,{\y}) {\large$0$};}
}

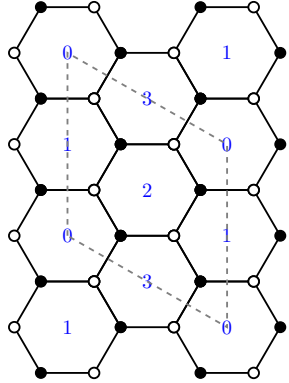
\begin{figure}[!ht]\centering
\scalebox{0.7}{
\begin{tikzpicture}
\newcommand{\edgewidth}{0.035cm} 
\newcommand{\boundaryrad}{1cm}
\newcommand{\noderad}{0.1cm} 
\newcommand{\nodewidthw}{0.035cm} 
\newcommand{\nodewidthb}{0.025cm} 

\dimerDtype
\draw[gray, line width=\edgewidth, dashed] (0,{2*cos(30)})--(0,{6*cos(30)})--(3,{4*cos(30)})--(3,0)--(0,{2*cos(30)});
\end{tikzpicture}}
\vspace{-2mm}

\caption{A dimer model $\Gamma$ associated to $\Delta_R$, where the dotted parallelogram stands for a fundamental domain of $\TT$.}
\label{fig_dimer_cD4}
\end{figure}

Let $\Delta=\Delta_\Gamma$.
There are four triangulations of the triangle $\Delta$ which are regular by the argument in \cite[Section~3]{dais2001all}.
Thus, any triangulation gives rise to a projective crepant resolution of $\Spec R$.
In particular, the flop graph of projective crepant resolutions takes the form of Figure~\ref{fig_flop_cD4}.
By Theorem~\ref{thm_crepant_dimer} (see also \cite[Theorem~1.1]{CrawIshii}), any projective crepant resolution of $\Spec R$ is obtained as
the moduli space $\calM_C$ for some chamber $C$ in $\Theta(Q_\Gamma)_\RR$.
Note that it is known that the quiver $Q_\Gamma$ coincides with the McKay quiver of $G=\ZZ/2\ZZ\times \ZZ/2\ZZ$, and for the chamber
$C_+\coloneqq\{\theta\in\Theta(Q_\Gamma)_\RR \mid \theta_v>0 \hspace{5pt} \text{for any $v\neq0$}\}$,
the moduli space $\calM_{C_+}$ is isomorphic to {\rm $G$-Hilb\,$\CC^3$}.
Furthermore, the \emph{skew group algebra} $S*G$ is derived equivalent to {\rm $G$-Hilb\,$\CC^3$} \cite[Theorem~1.2]{BKR}, and thus it is an NCCR of~$R$.

\begin{figure}[!ht]\centering
{\scalebox{1.1}{
\begin{tikzpicture}
\newcommand{\edgewidth}{0.035cm}
\newcommand{\noderad}{0.08} 
\newcommand{\cDtriangle}{
\foreach \n/\a/\b in {00/0/0,10/1/0, 20/2/0, 01/0/1, 02/0/2, 11/1/1} {
\coordinate (V\n) at (\a,\b);
};
\foreach \s/\t in {00/10,10/20,00/01,01/02,02/11,11/20} {
\draw [line width=\edgewidth] (V\s)--(V\t) ;};
\foreach \n in {00,10,20,01,02,11}{
\draw [fill=black] (V\n) circle [radius=\noderad] ; };
}

\node(CR1) at (0,0)
{\scalebox{0.6}{
\begin{tikzpicture}
\cDtriangle
\foreach \s/\t in {10/11,01/11,10/01} {
\draw [line width=\edgewidth] (V\s)--(V\t) ;};
\end{tikzpicture}
}};

\node(CR2) at (0,2.5)
{\scalebox{0.6}{
\begin{tikzpicture}
\cDtriangle
\foreach \s/\t in {10/11,10/02,10/01} {
\draw [line width=\edgewidth] (V\s)--(V\t) ;};
\end{tikzpicture}
}};

\node(CR3) at (3,-1.8)
{\scalebox{0.6}{
\begin{tikzpicture}
\cDtriangle
\foreach \s/\t in {20/01,01/11,10/01} {
\draw [line width=\edgewidth] (V\s)--(V\t) ;};
\end{tikzpicture}
}};

\node(CR4) at (-2.6,-1.8)
{\scalebox{0.6}{
\begin{tikzpicture}
\cDtriangle
\foreach \s/\t in {10/11,01/11,00/11} {
\draw [line width=\edgewidth] (V\s)--(V\t) ;};
\end{tikzpicture}
}};

\draw [line width=\edgewidth] (0,0.6)--(0,1.6); \draw [line width=\edgewidth] (0.8,-1)--(1.8,-1.7); \draw [line width=\edgewidth] (-0.8,-1)--(-1.8,-1.7);
\end{tikzpicture}
}}
\vspace{-2mm}

\caption{The flop graph of projective crepant resolutions of the toric $cD_4$ singularity $R$, where each projective crepant resolution is denoted by the associated triangulation.}
\label{fig_flop_cD4}
\end{figure}
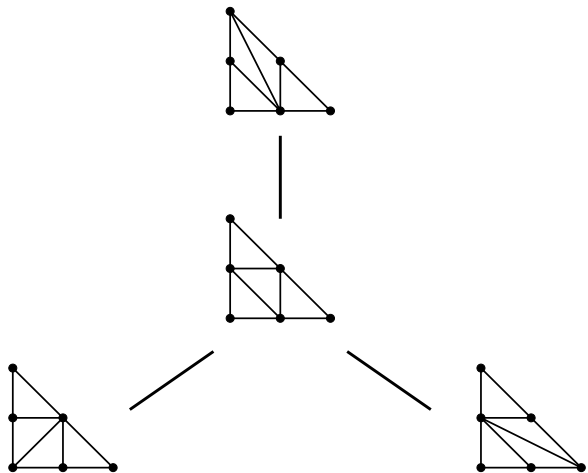

The wall-and-chamber structure of $\Theta(Q_\Gamma)_\RR$ has been studied in
\cite[Section~6]{BCS_GIT}, \cite[Section~5]{craw_thesis}, \cite{MT_stability}, \cite[Section~5]{BM_crepant}, and \cite[Remark~7.5, Example~7.6]{Wem_MMP}.
In what follows, we revisit the wall-and-chamber structure of $\Theta(Q_\Gamma)_\RR$ using the arguments similar to Theorem~\ref{thm_main_wall}.

We fix the lower left vertex of $\Delta=\Delta_\Gamma$ as the origin.
The dimer model $\Gamma$ has six zigzag paths $x_1$, $x_2$, $y_1$, $y_2$, $z_1$, $z_2$ as shown in Figure~\ref{fig_zigzag_cD4}.
The corner perfect matchings $\sfP_{(0,0)}$, $\sfP_{(2,0)}$, $\sfP_{(0,2)}$ of~$\Gamma$ corresponding to the vertices of~$\Delta$,
which are determined uniquely, can be obtained from these zigzag paths as follows:
\begin{gather*}
\sfP_{(0,0)}=\Zig(x_1)\cup\Zig(x_2)=\Zag(z_1)\cup\Zag(z_2), \\
\sfP_{(2,0)}=\Zig(y_1)\cup\Zig(y_2)=\Zag(x_1)\cup\Zag(x_2), \\
\sfP_{(0,2)}=\Zig(z_1)\cup\Zig(z_2)=\Zag(y_1)\cup\Zag(y_2).
\end{gather*}

\begin{figure}[!ht]\centering
{\scalebox{0.65}{
\begin{tikzpicture}
\newcommand{\edgewidth}{0.035cm} 
\newcommand{\boundaryrad}{1cm}
\newcommand{\noderad}{0.1cm} 
\newcommand{\nodewidthw}{0.035cm} 
\newcommand{\nodewidthb}{0.025cm} 
\newcommand{\zigzagwidth}{0.1cm} 

\node at (0,0){
\begin{tikzpicture}
\dimerDtype
\draw[gray, line width=\edgewidth, dashed] (0,{2*cos(30)})--(0,{6*cos(30)})--(3,{4*cos(30)})--(3,0)--(0,{2*cos(30)});
\draw[red, line width=\zigzagwidth, ->] (Ha42)--(Ha41)--(Ha32)--(Ha31)--(Ha22)--(Ha21)--(Ha12)--(Ha11)--(Ha16);
\draw[red, line width=\zigzagwidth, ->] (Hc43)--(Hc44)--(Hc33)--(Hc34)--(Hc23)--(Hc24)--(Hc13)--(Hc14)--(Hc15);
\node[red] at (0.6,0) {\large $x_1$}; \node[red] at (2.4,0) {\large $x_2$};
\end{tikzpicture}};

\node at (7,0){
\begin{tikzpicture}
\dimerDtype
\draw[gray, line width=\edgewidth, dashed] (0,{2*cos(30)})--(0,{6*cos(30)})--(3,{4*cos(30)})--(3,0)--(0,{2*cos(30)});
\draw[red, line width=\zigzagwidth, ->] (Ha13)--(Ha12)--(Hb13)--(Hb12)--(Hc23)--(Hc22);
\draw[red, line width=\zigzagwidth, ->] (Ha23)--(Ha22)--(Hb23)--(Hb22)--(Hc33)--(Hc32);
\draw[red, line width=\zigzagwidth, ->] (Ha33)--(Ha32)--(Hb33);
\draw[red, line width=\zigzagwidth, ->] (Hc14)--(Hc13)--(Hc12);
\node[red] at (2.8,2.9) {\large $y_2$}; \node[red] at (2.8,4.6) {\large $y_1$};
\end{tikzpicture}};

\node at (14,0){
\begin{tikzpicture}
\dimerDtype
\draw[gray, line width=\edgewidth, dashed] (0,{2*cos(30)})--(0,{6*cos(30)})--(3,{4*cos(30)})--(3,0)--(0,{2*cos(30)});
\draw[red, line width=\zigzagwidth, ->] (Hc22)--(Hc23)--(Hb22)--(Hb23)--(Ha32)--(Ha33);
\draw[red, line width=\zigzagwidth, ->] (Hc12)--(Hc13)--(Hb12)--(Hb13)--(Ha22)--(Ha23);
\node[red] at (0.2,2.9) {\large $z_2$}; \node[red] at (0.2,4.6) {\large $z_1$};
\end{tikzpicture}};
\end{tikzpicture}}}
\vspace{-2mm}

\caption{The zigzag paths of $\Gamma$.}\label{fig_zigzag_cD4}
\end{figure}
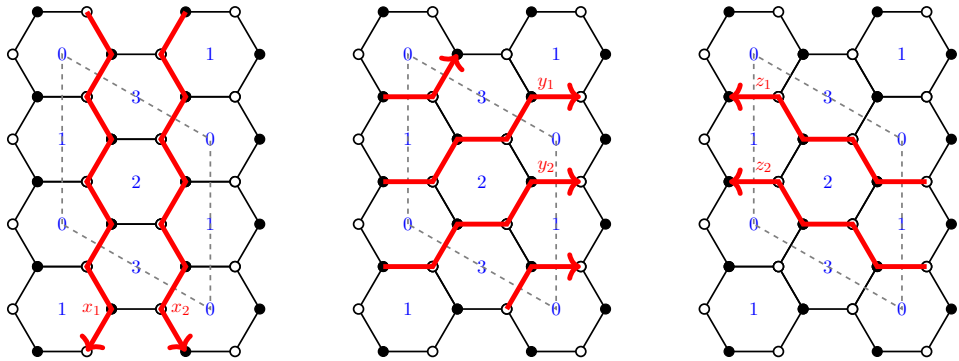

These perfect matchings are $\theta$-stable for any generic parameter $\theta\in\Theta(Q_\Gamma)_\RR$, see Proposition~\ref{prop_cPM_unique}.
On the other hand, by Proposition~\ref{prop_description_boundaryPM}, for a generic parameter $\theta\in\Theta(Q_\Gamma)_\RR$
and $\theta$-stable non-corner boundary perfect matchings
$\sfP_{(1,0)}^\theta$, $\sfP_{(1,1)}^\theta$, $\sfP_{(0,1)}^\theta$ respectively corresponding to $(1,0), (1,1), (0,1)\in\Delta$,
there exists a unique sequence $(x_{i_1}, x_{i_2}, y_{j_1}, y_{j_2}, z_{k_1}, z_{k_2})$ of zigzag paths such that
$\{i_1, i_2\}=\{j_1, j_2\}=\{k_1, k_2\}=\{1,2\}$ and
\begin{gather}
\sfP_{(1,0)}^\theta=\Zag(x_{i_1})\cup\Zig(x_{i_2}), \qquad \sfP_{(1,1)}^\theta=\Zag(y_{j_1})\cup\Zig(y_{j_2}), \nonumber\\
\sfP_{(0,1)}^\theta=\Zag(z_{k_1})\cup\Zig(z_{k_2}).\label{eq_bPM_cD4}
\end{gather}
Thus we assign the zigzag paths $x_{i_1}$, $x_{i_2}$, $y_{j_1}$, $y_{j_2}$, $z_{k_1}$, $z_{k_2}$ to primitive side segments of $\Delta$ as in
Figure~\ref{fig_zigzag_correspondnece_cD4}.

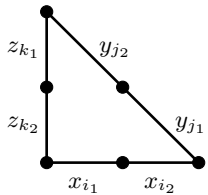
\begin{figure}[!ht]\centering
{\scalebox{1}{
\begin{tikzpicture}
\newcommand{\edgewidth}{0.035cm}
\newcommand{\noderad}{0.08} 
\foreach \n/\a/\b in {00/0/0,10/1/0, 20/2/0, 01/0/1, 02/0/2, 11/1/1} {
\coordinate (V\n) at (\a,\b);
};
\foreach \s/\t in {00/10,10/20,00/01,01/02,02/11,11/20} {
\draw [line width=\edgewidth] (V\s)--(V\t) ;};
\foreach \n in {00,10,20,01,02,11}{
\draw [fill=black] (V\n) circle [radius=\noderad] ; };
\node at (0.5,-0.3) {\small$x_{i_1}$}; \node at (1.5,-0.3) {\small$x_{i_2}$};
\node at (1.9,0.5) {\small$y_{j_1}$}; \node at (0.9,1.5) {\small$y_{j_2}$};
\node at (-0.3,0.5) {\small$z_{k_2}$}; \node at (-0.3,1.5) {\small$z_{k_1}$};
\end{tikzpicture}
}}
\vspace{-2mm}

\caption{The assignment of zigzag paths to primitive side segments of $\Delta$.}\label{fig_zigzag_correspondnece_cD4}
\end{figure}

Next, by cutting out an elementary triangle from $\Delta$, we consider three types of trapezoids as shown in Figure~\ref{fig_tree_trapezoids}.
We denote these trapezoids by $\Delta^x$, $\Delta^y$, and $\Delta^z$, respectively.
A difference from the case of $cA_{a+b-1}$ is that one of the parallel sides of a trapezoid is contained in the interior of $\Delta$.
Thus, there is no zigzag path whose slope coincides with the outer normal vector of such a side.
Nevertheless, the symmetric difference of boundary perfect matchings corresponding to endpoints of the side takes the place of a zigzag path.

\begin{figure}[!ht]\centering
{\scalebox{1}{
\begin{tikzpicture}
\newcommand{\edgewidth}{0.035cm}
\newcommand{\noderad}{0.08} 
\newcommand{\xnodeshift}{0.1cm}
\newcommand{\ynodeshift}{0.4cm}
\newcommand{\trapezoidsetting}{
\foreach \n/\a/\b in {00/0/0,10/1/0, 20/2/0, 01/0/1, 02/0/2, 11/1/1} {
\coordinate (V\n) at (\a,\b);
};}

\node at (0,-1.5) {$\Delta^x$}; \node at (4,-1.5) {$\Delta^y$}; \node at (8,-1.5) {$\Delta^z$};

\node at (0,0)
{\scalebox{0.8}{
\begin{tikzpicture}
\trapezoidsetting
\foreach \s/\t in {00/10,10/20,00/01,01/11,11/20} {
\draw [line width=\edgewidth] (V\s)--(V\t) ;};
\foreach \n in {00,10,20,01,11}{
\draw [fill=black] (V\n) circle [radius=\noderad] ; };
\node[xshift=-\xnodeshift,yshift=-\ynodeshift] at (V00) {\small$(0,0)$};
\node[xshift=\xnodeshift,yshift=-\ynodeshift] at (V20) {\small$(2,0)$};
\node[xshift=-\xnodeshift,yshift=\ynodeshift] at (V01) {\small$(0,1)$};
\node[xshift=\xnodeshift,yshift=\ynodeshift] at (V11) {\small$(1,1)$};
\end{tikzpicture}
}};

\node at (4,0)
{\scalebox{0.8}{
\begin{tikzpicture}
\trapezoidsetting
\foreach \s/\t in {10/20,10/01, 01/02,02/11,11/20} {
\draw [line width=\edgewidth] (V\s)--(V\t) ;};
\foreach \n in {10,20,01,02,11}{
\draw [fill=black] (V\n) circle [radius=\noderad] ; };
\node[xshift=-\xnodeshift,yshift=-\ynodeshift] at (V10) {\small$(1,0)$};
\node[xshift=\xnodeshift,yshift=-\ynodeshift] at (V20) {\small$(2,0)$};
\node[xshift=-\xnodeshift-0.5cm,yshift=0] at (V01) {\small$(0,1)$};
\node[xshift=-\xnodeshift-0.5cm,yshift=0] at (V02) {\small$(0,2)$};
\end{tikzpicture}
}};

\node at (8,0)
{\scalebox{0.8}{
\begin{tikzpicture}
\trapezoidsetting
\foreach \s/\t in {00/10,10/11, 00/01,01/02,11/02} {
\draw [line width=\edgewidth] (V\s)--(V\t) ;};
\foreach \n in {00,10,01,02,11}{
\draw [fill=black] (V\n) circle [radius=\noderad] ; };
\node[xshift=-\xnodeshift,yshift=-\ynodeshift] at (V00) {\small$(0,0)$};
\node[xshift=\xnodeshift,yshift=-\ynodeshift] at (V10) {\small$(1,0)$};
\node[xshift=\xnodeshift+0.5cm,yshift=0] at (V11) {\small$(1,1)$};
\node[xshift=-\xnodeshift-0.5cm,yshift=0] at (V02) {\small$(0,2)$};
\end{tikzpicture}
}};
\end{tikzpicture}
}}
\vspace{-2mm}

\caption{}\label{fig_tree_trapezoids}
\end{figure}

For each generic parameter $\theta\in\Theta(Q_\Gamma)_\RR$, the symmetric difference \smash{$x^\theta\coloneqq\sfP_{(1,1)}^\theta\ominus\sfP_{(0,1)}^\theta$}
satisfies~$\bigl[x^\theta\bigr]=-[x_1]=-[x_2]=(0,1)$.
Similarly, $y^\theta\coloneqq\sfP_{(0,1)}^\theta\ominus\sfP_{(1,0)}^\theta$ and \smash{$z^\theta\coloneqq\sfP_{(1,0)}^\theta\ominus\sfP_{(1,1)}^\theta$}
satisfy $\bigl[y^\theta\bigr]=-[y_1]=-[y_2]=(-1,-1)$, $\bigl[z^\theta\bigr]=-[z_1]=-[z_2]=(1,0)$.
The slopes $\bigl\{[x_1], [x_2], \bigl[x^\theta\bigr]\bigr\}$ correspond to outer normal vectors of two parallel lines of the left trapezoid in Figure~\ref{fig_tree_trapezoids}.
Also, the slopes $\bigl\{[y_1], [y_2], \bigl[y^\theta\bigr]\bigr\}$ and $\bigl\{[z_1], [z_2], \bigl[z^\theta\bigr]\bigr\}$ respectively correspond to outer normal vectors of two parallel lines of the center and right trapezoids in Figure~\ref{fig_tree_trapezoids}.
The path $x^\theta$ and $x_i$ $(i=1,2)$ would intersect each other, but we can check that intersections are not transversal.
Thus, we can define the regions $\calR\bigl(x_i, x^\theta\bigr)=\calR^+\bigl(x_i, x^\theta\bigr)$ and $\calR^-\bigl(x_i, x^\theta\bigr)$
in a similar way as in Section~\ref{subsec_main_wallcrossing}. The cases $y^\theta$ and $z^\theta$ are similar.
Then we fix a total order on $\bigl\{x_1, x_2, x^\theta\bigr\}$ so that
\[
\begin{cases}
x_2<x^\theta<x_1 & \text{if $x^\theta$ contained in $\calR^+(x_1,x_2)$}, \\
x^\theta<x_2<x_1 & \text{if $x^\theta$ contained in $\calR^-(x_1,x_2)$}.
\end{cases}
\]
Also, we fix a total order on $\bigl\{y_1, y_2, y^\theta\bigr\}$ and $\bigl\{z_1, z_2, z^\theta\bigr\}$ in a similar way.

For a chamber $C\subset \Theta(Q_\Gamma)_\RR$, we consider the triangulation $\Delta_C$ of $\Delta$ corresponding to $\calM_C$.
For any $\theta\in C$, we can assign the zigzag paths $x_1$, $x_2$, $y_1$, $y_2$, $z_1$, $z_2$ to primitive side segments of~$\Delta$ as in Figure~\ref{fig_zigzag_correspondnece_cD4}.
We also assign the paths $x_3\coloneqq x^\theta$, $y_3\coloneqq y^\theta$, and $z_3\coloneqq z^\theta$ to
the line segments of~$\Delta^x$, $\Delta^y$, and $\Delta^z$ whose outer normal vectors respectively correspond to $[x_3]$, $[y_3]$, and $[z_3]$.
Then we create a new sequence of paths in $\{x_1,x_2,x_3,y_1,y_2,y_3,z_1,z_2, z_3\}$ as follows.
First, the triangulation~$\Delta_C$ induces triangulations of at least two trapezoids of $\Delta^x$, $\Delta^y$, $\Delta^z$.
More precisely, the center triangulation in Figure~\ref{fig_flop_cD4} induces triangulations of all trapezoids,
and the remaining ones induce triangulations of two trapezoids out of $\Delta^x$, $\Delta^y$, $\Delta^z$.
In what follows, when we consider the center triangulation in Figure~\ref{fig_flop_cD4}, we choose two trapezoids from~$\Delta^x$,~$\Delta^y$,~$\Delta^z$
and their induced triangulations.
For example, we assume that $\Delta_C$ induces triangulations of~$\Delta^x$ and~$\Delta^y$,
and let $\{\Delta_{C,k}^x\}_{k=1}^3$ (resp.\ $\{\Delta_{C,k}^y\}_{k=1}^3$) be the set of elementary triangles
in the triangulation of $\Delta^x$ (resp.\ $\Delta^y$) induced from $\Delta_C$.
Note that we fix the index $k$ so that the line from $\bigl(0,\frac{1}{2}\bigr)$ to $\bigl(2,\frac{1}{2}\bigr)$ (resp.\ from $(\frac{3}{2}, 0)$ to $\bigl(0,\frac{3}{2}\bigr)$)
passes through $\Delta_{C,k}^x$ (resp.\ $\Delta_{C,k}^y$) first, then it passes through $\Delta_{C,k+1}^x$ (resp.\ $\Delta_{C,k+1}^y$) for any $k=1, 2$.
The assignment of $x_1$, $x_2$, $x_3$ (resp.\ $y_1$, $y_2$, $y_3$) to the primitive line segments in the triangulation of $\Delta^x$ (resp.\ $\Delta^y$)
determines the assignment of these paths to elementary triangles $\{\Delta_{C,k}^x\}_{k=1}^3$ (resp.\ $\{\Delta_{C,k}^y\}_{k=1}^3$).
Then we define the sequence $(u_1, \dots, u_6)$ so that $u_k$ is the path assigned to $\Delta_{C,k}^x$ and~$u_{k+3}$ is
the path assigned to $\Delta_{C,k}^y$ for $k=1,2,3$.
For the cases where $\Delta_C$ induces triangulations of $\Delta^y$ and $\Delta^z$ or~$\Delta^z$ and~$\Delta^x$,
we define the sequence $(u_1, \dots, u_6)$ in a similar way, but when we consider
the set $\{\Delta_{C,k}^z\}_{k=1}^3$ of elementary triangles in the triangulation of $\Delta^z$ induced from $\Delta_C$,
we fix the index $k$ so that the line from $\bigl(\frac{1}{2}, \frac{3}{2}\bigr)$ to $\bigl(\frac{1}{2}, 0\bigr)$
passes through $\Delta_{C,k}^z$ first, then it passes through $\Delta_{C,k+1}^z$ for any $k=1, 2$.
For example, if we consider the triangulation $\Delta_C$ as in the left of Figure~\ref{fig_triangulation_assign_cD4},
then this induces triangulations of $\Delta^x$ and $\Delta^y$ (see the center and the right of Figure~\ref{fig_triangulation_assign_cD4}).
If the paths $x_1$, $x_2$, $x_3$, $y_1$, $y_2$, $y_3$ are assigned as in Figure~\ref{fig_triangulation_assign_cD4},
then we have $(u_1, \dots, u_6)=(x_2, x_1, x_3, y_3, y_2, y_1)$.

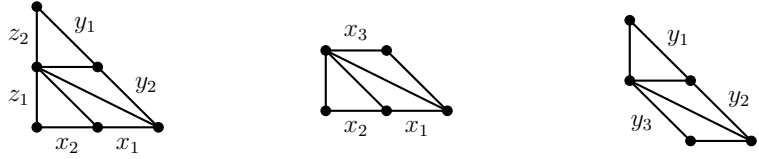
\begin{figure}[!ht]
\centering

{\scalebox{1}{
\begin{tikzpicture}
\newcommand{\edgewidth}{0.035cm}
\newcommand{\noderad}{0.08} 
\newcommand{\xnodeshift}{0.1cm}
\newcommand{\ynodeshift}{0.4cm}
\newcommand{\trapezoidsetting}{
\foreach \n/\a/\b in {00/0/0,10/1/0, 20/2/0, 01/0/1, 02/0/2, 11/1/1} {
\coordinate (V\n) at (\a,\b);
};}

\node at (0,0)
{\scalebox{0.8}{
\begin{tikzpicture}
\foreach \n/\a/\b in {00/0/0,10/1/0, 20/2/0, 01/0/1, 02/0/2, 11/1/1} {
\coordinate (V\n) at (\a,\b);
};
\foreach \s/\t in {00/10,10/20,00/01,01/02,02/11,11/20} {
\draw [line width=\edgewidth] (V\s)--(V\t) ;};
\foreach \n in {00,10,20,01,02,11}{
\draw [fill=black] (V\n) circle [radius=\noderad] ; };
\foreach \s/\t in {20/01,01/11,10/01} {
\draw [line width=\edgewidth] (V\s)--(V\t) ;};
\node at (0.5,-0.3) {\large$x_2$}; \node at (1.5,-0.3) {\large$x_1$};
\node at (1.8,0.7) {\large$y_2$}; \node at (0.8,1.7) {\large$y_1$};
\node at (-0.3,1.5) {\large$z_2$}; \node at (-0.3,0.5) {\large$z_1$};
\end{tikzpicture}
}};

\node at (4,0)
{\scalebox{0.8}{
\begin{tikzpicture}
\trapezoidsetting
\foreach \s/\t in {00/10,10/20,00/01,01/11,11/20, 01/10, 01/20} {
\draw [line width=\edgewidth] (V\s)--(V\t) ;};
\foreach \n in {00,10,20,01,11}{
\draw [fill=black] (V\n) circle [radius=\noderad] ; };
\node at (0.5,-0.3) {\large$x_2$}; \node at (1.5,-0.3) {\large$x_1$}; \node at (0.5,1.3) {\large$x_3$};
\end{tikzpicture}
}};

\node at (8,0)
{\scalebox{0.8}{
\begin{tikzpicture}
\trapezoidsetting
\foreach \s/\t in {10/20,10/01, 01/02,02/11,11/20, 01/20, 01/11} {
\draw [line width=\edgewidth] (V\s)--(V\t) ;};
\foreach \n in {10,20,01,02,11}{
\draw [fill=black] (V\n) circle [radius=\noderad] ; };
\node at (1.8,0.7) {\large$y_2$}; \node at (0.8,1.7) {\large$y_1$};
\node at (0.2,0.3) {\large$y_3$};
\end{tikzpicture}
}};
\end{tikzpicture}
}}
\vspace{-2mm}

\caption{A triangulation of $\Delta$, the induced triangulations of $\Delta^x$ and $\Delta^y$, and an example of the assignment of the paths $x_1$, $x_2$, $x_3$, $y_2$, $y_1$, $y_3$.}
\label{fig_triangulation_assign_cD4}
\end{figure}

We are now ready to state the theorem for type $cD_4$.

\begin{Theorem}\label{thm_main_cD4_1}
Let $C$ be a chamber in $\Theta(Q_\Gamma)_\RR$ and $\Delta_C$ be the triangulation of $\Delta$ giving rise to
the projective crepant resolution $\calM_C$. Let $\ell$ be an exceptional curve in $\calM_C$.
We suppose that $\Delta_{C,1}$, $\Delta_{C,2}$ are elementary triangles in $\Delta_C$ such that $\Delta_{C,1}\cap \Delta_{C,2}$ is the line segment corresponding to $\ell$,
in which case both $\Delta_{C,1}$ and $\Delta_{C,2}$ appear in two of triangulations of the trapezoids $\Delta^x$, $\Delta^y$, $\Delta^z$.
Let $(u_1, \dots, u_6)$ be the sequence defined as above and suppose that $u_k$, $u_{k+1}$ respectively correspond to $\Delta_{C,1}$, $\Delta_{C,2}$.
Note that $k$ is any of the index in $\{1,2,4,5\}$. Then we have the following.
\begin{itemize}\itemsep=0pt
\item[$(1)\hphantom{{}'}$] The equation \eqref{eq_wall_equation} derived from $\ell$ takes the form
$\sum_{v\in \calR}\theta_v=0$ where $\calR=\calR(u_k, u_{k+1})$, and
\[
W\coloneqq \biggl\{\theta\in\Theta(Q_\Gamma)_\RR \mid \sum_{v\in \calR}\theta_v=0 \biggr\}
\]
is a wall of $C$.
\item[$(2)\hphantom{{}'}$] The wall $W$ is of type {\rm I} if and only if $[u_k]=-[u_{k+1}]$.
\item[$(2)'$] The wall $W$ is of type {\typeIII} if and only if $[u_k]=[u_{k+1}]$.
\item[$(3)\hphantom{{}'}$] Any parameter $\theta\in C$ satisfies $\sum_{v\in \calR}\theta_v>0$ if $u_k< u_{k+1}$.
\item[$(3)'$] Any parameter $\theta\in C$ satisfies $\sum_{v\in \calR}\theta_v<0$ if $u_k> u_{k+1}$.
\end{itemize}
\end{Theorem}

\begin{proof}
The proof is similar to Theorem~\ref{thm_main_wall}. We note some differences:
\begin{itemize}\itemsep=0pt
\item When the path $u_k$ or $u_{k+1}$ is one of $x_3$, $y_3$, or $z_3$, we use these paths as substitutes for zigzag paths.
\item $\Hex(\sigma_+)$ (resp.\ $\Hex(\sigma_-)$) might contain an edge $e$ such that $e\in\sfP_0\cap\sfP_1\cap\sfP_2$ (resp.\ $e\in\sfP_1\cap\sfP_2\cap\sfP_3$) in its strict interior.
Thus, Lemma~\ref{lem_property_interior} is false in this situation.
Nevertheless, the arrow dual to such an edge $e$ is not supported by $\theta$-stable representations corresponding to $\sigma_+$ (resp.\ $\sigma_-$),
which means the quiver $\widetilde{Q^{\sigma_+}}$ (resp.\ $\widetilde{Q^{\sigma_-}}$) does not contain the arrow dual to $e$.
Thus, we do not need to take care of $e$ when we compute $\deg(\calL_v |_\ell)$.
\item When an edge $e$ as above exists, $\fkc_-$ (resp.\ $\fkc_+$) is not a single edge.
Nevertheless, the edges contained in $\fkc_-$ (resp.\ $\fkc_+$) and supported by $\theta$-stable representations corresponding to
$\sigma_+$ (resp.\ $\sigma_-$) are all zigs or zags of some path in $\{u_1, \dots, u_6\}$. We substitute this fact for Lemma~\ref{lem_zigzag_containPM}.
Note that although $x_3$, $y_3$, $z_3$ are not zigzag paths, we use the same terminologies ``zig'' and ``zag'' for an edge directed from white to black and from black to white.\hfill \qed
\end{itemize}\renewcommand{\qed}{}
\end{proof}

\begin{Example}
Let $C$ be a chamber in $\Theta(Q_\Gamma)_\RR$.
Suppose that the triangulation $\Delta_C$ takes the form as in the left of Figure~\ref{fig_triangulation_assign_cD4},
in which the zigzag paths $x_1$, $x_2$, $y_1$, $y_2$, $z_1$, $z_2$ are assigned to primitive side segments.
Then, for any $\theta\in C$, non-corner boundary $\theta$-stable perfect matchings take the following forms:
\[
\sfP_{(1,0)}^\theta=\Zag(x_2)\cup\Zig(x_1), \qquad \sfP_{(1,1)}^\theta=\Zag(y_2)\cup\Zig(y_1), \qquad
\sfP_{(0,1)}^\theta=\Zag(z_2)\cup\Zig(z_1).
\]
From these perfect matchings, we have the paths $x_3=x^\theta$, $y_3=y^\theta$, and $z_3=z^\theta$ as in Figure~\ref{fig_paths_PM}.
Also, we see that $x_2<x_3<x_1$, $y_2<y_3<y_1$ and $z_2<z_3<z_1$.
\begin{figure}[!ht]
\centering
{\scalebox{0.65}{
\begin{tikzpicture}
\newcommand{\edgewidth}{0.035cm} 
\newcommand{\boundaryrad}{1cm}
\newcommand{\noderad}{0.1cm} 
\newcommand{\nodewidthw}{0.035cm} 
\newcommand{\nodewidthb}{0.025cm} 
\newcommand{\zigzagwidth}{0.1cm} 

\node at (0,-4.5) {\Large $x_3=\sfP^\theta_{(1,1)}\ominus\sfP^\theta_{(0,1)}$};
\node at (7,-4.5) {\Large $y_3=\sfP^\theta_{(0,1)}\ominus\sfP^\theta_{(1,0)}$};
\node at (14,-4.5) {\Large $z_3=\sfP^\theta_{(1,0)}\ominus\sfP^\theta_{(1,1)}$};

\node at (0,0){
\begin{tikzpicture}
\dimerDtype
\draw[gray, line width=\edgewidth, dashed] (0,{2*cos(30)})--(0,{6*cos(30)})--(3,{4*cos(30)})--(3,0)--(0,{2*cos(30)});
\draw[red, line width=\zigzagwidth, ->] (Hb16)--(Hb11)--(Hb12)--(Hb13)--(Ha21)--(Ha22)--(Ha31)--(Hb22)--(Hb31)--(Hb32);
\node[red] at (2,4.5) {\large $x_3$};
\end{tikzpicture}};

\node at (7,0){
\begin{tikzpicture}
\dimerDtype
\draw[gray, line width=\edgewidth, dashed] (0,{2*cos(30)})--(0,{6*cos(30)})--(3,{4*cos(30)})--(3,0)--(0,{2*cos(30)});
\draw[red, line width=\zigzagwidth, ->] (Hc42)--(Hc43)--(Hc44)--(Hc33)--(Hb22)--(Hb23)--(Ha32)--(Ha33);
\draw[red, line width=\zigzagwidth, ->] (Hc22)--(Hc23)--(Hc24)--(Hc13)--(Hb16);
\node[red] at (0.5,3.8) {\large $y_3$};
\end{tikzpicture}};

\node at (14,0){
\begin{tikzpicture}
\dimerDtype
\draw[gray, line width=\edgewidth, dashed] (0,{2*cos(30)})--(0,{6*cos(30)})--(3,{4*cos(30)})--(3,0)--(0,{2*cos(30)});
\draw[red, line width=\zigzagwidth, ->] (Ha33)--(Ha32)--(Ha31)--(Ha22)--(Ha21)--(Hb12)--(Hb21)--(Hc22) ;
\node[red] at (2.7,2.3) {\large $z_3$};
\end{tikzpicture}};
\end{tikzpicture}}}
\vspace{-2mm}
\caption{}\label{fig_paths_PM}
\end{figure}

On the other hand, the triangulation $\Delta_C$ induces the triangulations of $\Delta^x$ and $\Delta^y$.
The assignments of paths as in the center and the right of Figure~\ref{fig_triangulation_assign_cD4} induce
the sequence $(u_1, \dots, u_6)=(x_2, x_1, x_3, y_3, y_2, y_1)$.

We first pay attention to the parallelogram whose vertices are $(1,0)$, $(2,0)$, $(1,1)$ and $(0,1)$.
We consider the elementary triangles $\Delta_{C,1}$ and $\Delta_{C,2}$ in the triangulation $\Delta_C$ which form this parallelogram,
in which $\Delta_{C,1}\cap \Delta_{C,2}$ is the diagonal connecting $(2,0)$ and $(0,1)$.
In this situation, $u_2=x_1$, $u_3=x_3$ respectively correspond to $\Delta_{C,1}$, $\Delta_{C,2}$.
Applying Theorem~\ref{thm_main_cD4_1}, we have that $\sum_{v\in\calR(u_2,u_3)}\theta_v=\theta_3=0$ is a wall of $C$, and it is of type I.
Moreover, any $\theta\in C$ satisfies $\theta_3<0$ since $u_2>u_3$.
Note that $u_4=y_3$, $u_5=y_2$ also respectively correspond to $\Delta_{C,1}$, $\Delta_{C,2}$,
and even if we use these paths we have the same conclusion.
We then pay attention to the triangle whose vertices are $(0,0)$, $(2,0)$, and $(0,1)$.
We consider the elementary triangles $\Delta_{C,1}^\prime$ and $\Delta_{C,2}^\prime=\Delta_{C,1}$ in the triangulation $\Delta_C$
which form the above triangle, in which $\Delta_{C,1}^\prime\cap \Delta_{C,2}^\prime$ is the line segment connecting $(1,0)$ and $(0,1)$.
In this situation, $u_1=x_2$, $u_2=x_1$ respectively correspond to~$\Delta_{C,1}^\prime$,~$\Delta_{C,2}^\prime$.
Applying Theorem~\ref{thm_main_cD4_1}, we have that $\sum_{v\in\calR(u_1,u_2)}\theta_v=\theta_2+\theta_3=0$ is a wall of $C$, and it is of type $\typeIII$.
Moreover, any $\theta\in C$ satisfies $\theta_2+\theta_3>0$ since $u_1<u_2$.
Similarly, considering the triangle whose vertices are $(2,0)$, $(0,2)$, and $(0,1)$,
we have the wall $\theta_1+\theta_3=0$ of type $\typeIII$, and any $\theta\in C$ satisfies $\theta_1+\theta_3>0$.
In summary, we see that
\[
C=\{\theta_3<0, \, \theta_2+\theta_3>0, \, \theta_1+\theta_3>0\}.
\]
\end{Example}

\begin{Corollary}
The number of GIT regions in $\Theta(Q_\Gamma)_\RR$ is $8$, and each GIT region contains four chambers which correspond to projective crepant resolutions as in {\rm Figure~\ref{fig_flop_cD4}}. Thus, the number of chambers in $\Theta(Q_\Gamma)_\RR$ is $32$.
\end{Corollary}

\begin{proof}
We first choose a chamber $C$, in which $\calM_C$ corresponds to one of the triangulations in Figure~\ref{fig_flop_cD4}.
Then we have the sequence $(x_{i_1}, x_{i_2}, y_{j_1}, y_{j_2}, z_{k_1}, z_{k_2})$ giving rise to $\theta$-stable perfect matchings as in \eqref{eq_bPM_cD4} for any $\theta\in C$.
Let $G$ be a GIT region of $\Theta(Q_\Gamma)_\RR$ containing $C$.
By Theorem~\ref{thm_main_cD4_1}, there exists a type I wall corresponding to any floppable curve in $\calM_C$.
Since all projective crepant resolutions as in Figure~\ref{fig_flop_cD4} are connected by repetitions of flops
and all chambers in $G$ are connected by crossings of walls of type I, these crepant resolutions can be obtained from chambers in $G$.
Thus, any GIT region contains four chambers.
Since a flop preserves any toric divisor, $\theta$-stable perfect matchings are the same for any chamber in the same GIT region.
Thus, the same sequence of $\{x_1, x_2, y_1, y_2, z_1, z_2\}$ is assigned to any chamber in the same GIT region.
Since there are eight choices of such sequences, we have the assertion.
\end{proof}

\subsection*{Acknowledgements}
The author would like to thank Alastair Craw, Wahei Hara, and Michael Wemyss for their valuable comments.
The author would also like to thank the anonymous referees for their numerous valuable comments and suggestions.
The author was supported by JSPS Grant-in-Aid for Early-Career Scientists JP20K14279,
and is supported by JSPS Grant-in-Aid for Scientific Research (C) JP24K06698.

\pdfbookmark[1]{References}{ref}
\LastPageEnding

\end{document}